\documentclass[11pt]{amsart}

 \usepackage[utf8]{inputenc}
 \usepackage[T1]{fontenc}

 \usepackage[english]{babel}

 \usepackage{amsmath,amsfonts,amssymb}
 \usepackage{mathrsfs}
 \usepackage{dsfont} 
 \usepackage{mathtools} 
 \usepackage{esint} 
 \usepackage{xfrac} 
 \usepackage{defs} 
 \usepackage{fonts} 

 \usepackage{amsthm}
 \usepackage[shortlabels]{enumitem} 
 \usepackage[colorlinks=true,urlcolor=blue, citecolor=red,linkcolor=blue,linktocpage,pdfpagelabels, bookmarksnumbered,bookmarksopen]{hyperref} 
 \usepackage[paper=a4paper,hmargin=1in,vmargin=1in, marginparwidth=0.8in]{geometry} 

 \usepackage[colorinlistoftodos,prependcaption,linecolor=blue,backgroundcolor=blue!25,bordercolor=blue,textsize=tiny]{todonotes}
 \usepackage{cite} 
 \usepackage{cleveref}
 \usepackage{comment}

plus 6pt minus 12pt
\numberwithin{equation}{section}

\newtheorem{theorem}{Theorem}[section]
\newtheorem{proposition}[theorem]{Proposition}
\theoremstyle{definition}
\newtheorem{definition}[theorem]{Definition}
\theoremstyle{plain}
\newtheorem{lemma}[theorem]{Lemma}
\newtheorem{corollary}[theorem]{Corollary}
\theoremstyle{remark}
\newtheorem{remark}[theorem]{Remark}

\title[Nonlocal boundary-value problems]{Nonlocal boundary-value problems with local boundary conditions}
 \author{James M. Scott}
 \address{Department of Applied Physics and Applied Mathematics,
 Columbia University,
 New York, NY 10027}
 \email{jms2555@columbia.edu}
 \author{Qiang Du}
 \address{Department of Applied Physics and Applied Mathematics, and the Data Science Institute,
 Columbia University,
 New York, NY 10027}
 \email{qd2125@columbia.edu}

\thanks{This work is supported in part 
by the NSF DMS-1937254  and DMS-2012562.}

\keywords{nonlocal equations, boundary-value problems, Poisson problem, Green's identity, nonlocal operators, heterogeneous localization, vanishing horizon}

\subjclass[2020]{45K05, 35J20, 46E35}

\DeclareMathOperator*{\argmin}{argmin}

\begin{document}

\maketitle

\begin{abstract}

We describe and analyze nonlocal integro-differential equations with classical local boundary conditions. The interaction kernel  of the nonlocal operator 
has horizon parameter dependent on position in the domain, and vanishes as the boundary of the domain is approached. This heterogeneous localization allows for boundary values to be captured in the trace sense.
We state and prove a nonlocal Green's identity for these nonlocal operators that involve local boundary terms. We use this identity to state and establish the well-posedness of variational formulations of the nonlocal problems with several types of classical boundary conditions.

We show the consistency of these nonlocal boundary-value problems with their classical local counterparts in the vanishing horizon limit via the convergence of solutions.
The Poisson data for the local boundary-value problem is permitted to be quite irregular, belonging to the dual of the classical Sobolev space.
Heterogeneously mollifying this Poisson data for the local problem on the same length scale as the horizon and using the regularity of the interaction kernel, we show that the solutions to the nonlocal boundary-value problem with the mollified Poisson data actually belong to the classical Sobolev space, and converge weakly to the unique variational solution of the classical Poisson problem with original Poisson data.
\end{abstract}


\section{Introduction}\label{sec:Intro}

There has been much recent interest in 
nonlocal problems on a bounded domain $\Omega \subset \bbR^d$:
\begin{equation}\label{eq:Intro:NonlocalEqn}
	\cL u = f \text{ in } \Omega\,, 
\end{equation}
associated to a nonlocal integral operator
\begin{equation}\label{eq:Intro:Operator-a}
	\cL u(\bx) := 2 \int_{\bbR^d}
\gamma
(\bx,\by) (u(\bx)-u(\by)) \, \rmd \by\, ,
\end{equation}
defined for measurable functions $u : \bbR^d \to \bbR$ and a nonlocal, symmetric and (often) nonnegative  kernel $\gamma$. 
These operators appear widely in both analysis and applications \cite{rossi,barles2014neumann,barlow2009non,braides2022compactness,bucur2016nonlocal,caffarelli2007extension,caffarelli2011regularity,craig2016blob,carrillo2019blob,Du12sirev,
Du2019book,Gilboa-Osher,grinfeld2005non,lovasz2012large,MeKl04,mogilner1999non,Nochetto2015,Coifman05,shi2017,Valdinoci:2009,Kassmann}.
Here, symmetry means ${\gamma}(\bx,\by) = {\gamma}(\by,\bx)$, so that we can identify $\cL$ with a variational form associated with a quadratic nonlocal energy. 
Earlier studies of variational formulations of \eqref{eq:Intro:NonlocalEqn} {on bounded domains} have taken several different paths. 

Along the path that ${\gamma=\gamma}(\bx,\by)$ has a singularity at the diagonal $\bx=\by$, both volume-constraint problems and classical boundary-value problems have been investigated, see for example \cite{cortazar,DyKa19,Ros16,andreu2010nonlocal} and additional references cited therein. If $\cL$ defines a hypersingular integral operator, in particular, then classical boundary values can be prescribed via the trace operator, see \cite{A75,Hitchhiker} for the cases of singular kernels that give rise to solutions in fractional 
Sobolev-Slobodeckij spaces. 

Down another path, ${\gamma=\gamma}(\bx,\by)$
is a compactly supported and translation invariant kernel, e.g., $\gamma(\bx,\by)= \delta^{-d-2}\omega(|\bx-\by|/\delta)$ for a function $\omega$ supported in the unit interval $(0,1)$ and a constant (horizon parameter $\delta$) that measures the range of nonlocal interactions. One natural route to take is to define the so-called nonlocal \textit{volumetric constraint} to complement the equation defined on $\Omega$ \cite{Du12sirev,Du-NonlocalCalculus,Du2022nonlocal}. An example is the prescription of $u(\bx)$ in a layer consisting of $\bx\in\Omega^c$ with $\dist(\bx, \Omega)<\delta$.
An alternative is to modify the nonlocal interaction rules involving $u=u(\bx)$ in a layered domain, say, for $\bx \in \Omega$ with $\dist(\bx, \p \Omega) < \delta$. These volumetric conditions can recover traditional boundary conditions in the local limit as $\delta\to 0$ under suitable conditions, see e.g., \cite{bellido2015hyperelasticity,mengesha2015localization,Du2022nonlocal,d2020physically,foss2022convergence,Lipton-2016}. Meanwhile, in the $\delta\to \infty$ limit with a rescaled fractional kernel, these problems are  related to studies of fractional differential equations defined on a bounded domain
\cite{bellido2021restricted,Foghem2022,Grubb,d2021connections}. In addition, one can find connections to the continuum limits of discrete graph operators and discrete particle interactions \cite{braides2022compactness,Coifman05,garcia2020error}.

{For various nonlocal problems, studies of their well-posedness subject to nonlocal volumetric constraints can be found, for example, in \cite{Du-NonlocalCalculus, MengeshaDuElasticity}, which offered desirable mathematical insight as demonstrated for a number of applications such as the peridynamics models developed in mechanics \cite{Silling2000,Silling2008,Du-Zhou2011,Lipton-2014,Valdinoci-peridynamic}, nonlocal diffusion and jump processes \cite{Du12sirev,Burch2014exit-time} and nonlocal Stokes equations for the analysis of smoothed particle hydrodynamics \cite{Du-Tian2020}.}

Still another path is to mix classical boundary conditions and volume-constraint conditions in constitutive models that blend local and nonlocal models. For an extensive discussion relating to the many choices of blended models in applications such as peridynamics, see the survey \cite{d2022review}.

We are interested in  boundary-value problems for nonlocal problems on a bounded domain in the classical sense, that is, the boundary conditions are prescribed on $\partial \Omega$ only.
The motivation is two-fold: first,  while the nonlocal constraints are natural, they are not perfect choices.  Theoretically, nonlocal constraints may raise unintended concerns about the regularity of solutions, for instance, non-constant functions vanishing in a layer of nonzero measure no longer enjoy analyticity, and solutions of problems with smooth kernels may experience non-physical or undesirable jumps at the boundary due to unmatched nonlocal constraints \cite{Du2022nonlocal}. In practice, developers of simulation codes for applications of nonlocal models, 
have ample practical reasons to keep local boundary conditions in implementation even though a nonlocal model might be derived and/or deemed a better modeling choice in the domain of interest.

To allow for the prescription of local boundary conditions, the nonlocal operator $\cL $ and the nonlocal solution spaces must be defined so that boundary values of the solutions make sense. For a nonlocal operator associated with a kernel $\gamma=\gamma(\bx,\by)$ that does not have sufficient singularity on the diagonal  $\bx=\by$, it means that some localizing property near the boundary should hold. For instance,
in \cite{du2022fractional, tian2017trace, tao2019nonlocal},
with a pair of constants $\beta \in (0,1)$ and $\delta>0$,
a function
$h_\delta(\bx)=\min\{\delta, \beta \dist(\bx, \partial \Omega)\}$ is introduced to characterize the extent of nonlocal interactions at a point $\bx\in \Omega$, instead of taking a constant $\delta$ as the horizon parameter everywhere in the domain. A kernel similar to 
$$
{\gamma} (\bx,\by) = \frac{1}{2h_\delta(\bx)^{d+2}}
\mathds{1}_{ \{ |\bx-\by|< h_\delta(\bx) \} }
+ \frac{1}{2h_\delta(\by)^{d+2}}
\mathds{1}_{ \{ |\bx-\by|< h_\delta(\by) \} }
$$
is adopted. Clearly,  $h_\delta(\bx)\to 0$ as $\bx\to \partial \Omega$ so that the interactions are localized on the boundary.
A consequence of this type of {\em heterogeneous localization} is that functions in $L^2(\Omega)$  with a  bounded energy can have well-defined traces on $\partial \Omega$ to allow classical,  
local boundary conditions for the problem \eqref{eq:Intro:NonlocalEqn}, see \cite{tian2017trace,Foss2021}.

Studies of nonlocal operators with heterogeneous localization also appear in the seamless coupling of local and nonlocal models \cite{tao2019nonlocal}.
Yet, there are a number of fundamental questions that remain to be answered. On one hand, there is much desire to establish rigorously the well-posedness theory of nonlocal problems with these local (and particularly inhomogeneous) boundary conditions. Establishing suitable regularity properties is also important for mathematical theory and physical consistency. At the same time, one may ask if a nonlocal analog of the Green's identity can be shown for operators that involve heterogeneous localization, so that the pointwise and variational forms can be placed in natural correspondence. An important observation, as presented later in this work, is that a new localization strategy is needed to ensure the validity of the proper nonlocal Green's identity. Intuitively, the function $\eta_\delta(\bx)$ needs to vanish at a faster rate than $ \dist(\bx, \partial \Omega)$, which also translates into a wider transition region where the range of nonlocal interactions changes from $\delta$ to $0$. This calls for further (and more delicate) mathematical analysis and also bears significant consequences in the application of localization strategies to nonlocal modeling.

\subsection{Main results}\label{sec:MainResults}
To describe our main findings, we first introduce the specific form of the nonlocal operator under consideration here
\begin{equation}\label{eq:Intro:Operator}
	\cL_{\delta} u(\bx) := 2 \int_{\Omega} \rho_{\delta,2}(\bx,\by) (u(\bx)-u(\by)) \, \rmd \by\,
\end{equation}
where $\rho_{\delta,2}$ is taken from a special class of kernels $\rho_{\delta,\alpha}$ parameterized by $\delta$ and $\alpha$. These two parameters control respectively the extent and magnitude of interaction between points. More specifically,
\begin{equation}\label{eq:OperatorKernelDef}
	\begin{split}
		\rho_{\delta,\alpha}(\bx,\by) &:= \frac{1}{2} \left[  \frac{1}{\eta_{\delta}(\bx)^{d+\alpha}} {\rho} \left( \frac{|\by-\bx|}{\eta_{\delta}(\bx)} \right) +  \frac{1}{\eta_{\delta}(\by)^{d+\alpha}} {\rho} \left( \frac{|\by-\bx|}{\eta_{\delta}(\by)} \right) \right]\,.
	\end{split}
\end{equation}
for functions $\rho:[0,\infty) \to [0,\infty)$ and
$\eta_{\delta} : \overline{\Omega} \to [0,\infty)$.

In the equation \eqref{eq:OperatorKernelDef},  the function $\rho$ satisfies the following assumption:
\begin{equation}\label{Assump:KernelSmoothness}
	\begin{gathered}
 \text{There exist constants } 0 < \rho_0 < \infty,\; 0<r_0<R_0<1,\; \text{ such that } 
 \\
		\rho \in C^1([0,\infty)) \text{ with } \rho \geq 0 \text{ and } \rho''(0) = \lim\limits_{r \to 0^+} \frac{\rho'(r)}{r} \text{ exists, } \\
		\supp \rho \subset [0,R_0),
  \text{ and }
		\rho(r) \geq \rho_0 > 0,\;\forall r \in [0,r_0].
		\tag{A1}
	\end{gathered}
\end{equation}
We also assume the normalization condition
\begin{equation}\label{Assump:KernelNormalization}
	\int_{B(0,1)} |\bz|^2 \rho(|\bz|) \, \rmd \bz = d\,,
	\tag{A2}
\end{equation}
with $B(0,1)$ denoting the unit ball centered at the origin in $\mathbb{R}^d$.

The function $\eta_\delta$ represents the
extent of interaction that
takes on a form in this paper more complex than being merely a constant function. In the latter case, for a constant horizon parameter $\delta>0$ 
we have $\eta_\delta(\bx)=\delta$ for any $\bx$, 
With such a choice, the normalization condition \eqref{Assump:KernelNormalization} implies that for $\Omega=\bbR^d$, we have $\cL_\delta v=\Delta v$ for any quadratic function $v=v(\bx)$. Indeed, 
 the choice of $\alpha = 2$ is made in the definition of $\cL_{\delta}$ so that the nonlocal operator $\cL_{\delta}$ serves as an analogue of 
classical second-order elliptic partial differential operator. Moreover,  we consider the case that $\eta_\delta=\eta_\delta(\bx)$
 is a smooth function in $\Omega$ and allows a specific form of heterogeneous localization on the boundary $\p \Omega$, i.e., it is constant outside of a thin layer near $\p \Omega$ and vanishes as $\bx$ approaches $\p \Omega$, as described below.

We make the following assumption on the domain $\Omega$.
\begin{equation}\label{assump:Domain}
\begin{gathered}
\Omega \subset \bbR^d \text{ is a bounded } C^2 \text{ domain equipped with a heterogeneous localization}\\
\text{function } \eta_{\delta} \text{ i.e., there exist positive constants $\kappa_0 \in (0,1)$, $\kappa_1$, $\kappa_2$ and $\bar{\kappa}_0$}\\ 
\text{depending only at most on $d$ and $\Omega$
and a function $\eta_{\delta}: \overline{\Omega} \to [0,\infty)$}\\
\text{such that the heterogeneous localization
properties \eqref{assump:Localization} hold.}
 \end{gathered}  \tag{$\rmA_{\Omega}$}
\end{equation}

Here, the heterogeneous localization properties refer to
\begin{equation} \label{assump:Localization}
\begin{aligned}
&\text{For every $0 < \delta < \delta_0$, where }
    \delta_0 := \min \left\{ 1, \frac{1}{9 \kappa_1^2}\,, \frac{1}{ 2 \bar{\kappa}_0^2 } \right\},\; \eta_{\delta}  \text{ satisfies }\\
& i) \text{ $\eta_{\delta}(\bx) = 
(\dist(\bx,\p \Omega))^2$ for all $\bx \in \Omega$ with $(\dist(\bx,\p \Omega))^2 < \kappa_0 \delta$,}\\
& ii) \text{ $\eta_{\delta}(\bx) = \delta$ for all $\bx \in \Omega$ with $\kappa_0 (\dist(\bx,\p \Omega))^2 > \delta$,}\\
& iii)  \text{ $\eta_{\delta} \in C^2(\overline{\Omega})$ with $|\grad \eta_{\delta}(\bx)| \leq \kappa_1 \sqrt{\delta}$ and $|\grad^2 \eta_{\delta}(\bx)| \leq \kappa_2$ for all $\bx \in \Omega$,}\\
& iv) \text{ $\eta_{\delta}(\bx) \leq \bar{\kappa}_0 \min \{ \delta, (\dist(\bx,\p \Omega))^2 \}$ for all $\bx \in \Omega$.}
\end{aligned} \tag{$\rmA_{\eta}$}
\end{equation}

Note that for any $C^2$ domain a heterogeneous localization function $\eta_{\delta}$ is guaranteed to exist. Indeed, one choice of $\eta_{\delta}$ is the following.
Choose $\kappa_0 \in (0,1)$ to be small enough so that the function $\bx \mapsto \dist(\bx, \p \Omega)$ belongs to $C^2(\overline{ \{ \bz \in \Omega \, : \, \dist(\bz,\Omega) < \sqrt{\kappa_0} \} })$; such a $\kappa_0$ will always exist and depend at most on the boundary character of $\Omega$ and $d$, since $\Omega$ is a $C^2$ domain. For any $\delta > 0 $ construct via mollification a $C^{\infty}$ cutoff function $\psi_{\delta}$ satisfying $\psi_{\delta} \in C^2(\overline{\Omega})$, $\psi_{\delta} \equiv 1$ for $\kappa_0 (\dist(\bx,\p \Omega))^2 > \delta$ and $\psi_{\delta} = 0$ for  $(\dist(\bx,\p \Omega))^2 < \kappa_0 \delta$ with $|\grad \psi_{\delta}(\bx)| \leq \frac{C_1}{\sqrt{\delta}}$ and $|\grad^2 \psi_{\delta}(\bx)| \leq \frac{C_2}{\delta}$, where $C_1$ and $C_2$ depend only on $d$, $\Omega$ and $\kappa_0$.
Then we choose $\kappa_1 = C_1$ and $\kappa_2 = C_2$, and define $\eta_{\delta}$ by
\begin{equation*}
	\eta_{\delta}(\bx) = \psi_{\delta}(\bx) \delta + (1-\psi_{\delta}(\bx)) (\dist(\bx,\p \Omega))^2\,.
\end{equation*}
The $\delta_0$ defined above guarantees that that $\eta_{\delta}$ is in $C^2(\overline{\Omega})$
for all $\delta < \delta_0$. Therefore conditions i)-iv) hold, where $\bar{\kappa}_0 = \frac{2}{\kappa_0} - 1$.

This construction is of course not unique. One could choose a larger $\kappa_0$ to ensure a wider ``transition region'' $\{ \bx \in \Omega \, : \kappa_0 \delta < (\dist(\bx,\p \Omega))^2 < \kappa_0^{-1} \delta \}$, which would require a small $\delta_0$ to be chosen so that $\eta_{\delta}$ remains in $C^2(\overline{\Omega})$ for all $\delta < \delta_0$.

Associated to the nonlocal operator \eqref{eq:Intro:Operator} and the nonlocal kernel $\rho_{\delta,2}$ we define a real symmetric bilinear form
\begin{equation*}
	B_{\rho,\delta}(u,v)  := \iintdm{\Omega}{\Omega}{ \rho_{\delta,2}(\bx,\by) \big( u(\bx)-u(\by) \big) \big( v(\bx)-v(\by) \big) }{\by}{\bx}\,,
\end{equation*}
and so Hilbert space methods are available for the solvability of the nonlocal boundary-value problem  \eqref{eq:Intro:NonlocalEqn} subject to various local boundary conditions. 

Thanks to the superlinear rate of localization at the boundary of $\Omega$, 
we are able to establish a {\em nonlocal Green's identity} for our nonlocal operator, as formulated in the following theorem, which serves as the foundation for consistent formulations of nonlocal boundary value problems.

\begin{theorem}\label{thm:GreensIdentity:Intro}
    Let $\Omega \subset \bbR^d$ satisfy \eqref{assump:Domain}. Let $\rho$ satisfy \eqref{Assump:KernelSmoothness}-\eqref{Assump:KernelNormalization}. Suppose $u$, $v \in C^2(\overline{\Omega})$. Then
    \begin{equation}\label{eq:GreensIdentity:Intro}
            B_{\rho,\delta}(u,v)
     = \int_{\Omega} \cL_{\delta} u(\bx) \cdot v(\bx) \, \rmd \bx + \int_{\p \Omega} \frac{\p u}{\p \bsnu}(\bx) \cdot v(\bx) \, \rmd \sigma(\bx) \,,
	\end{equation}
    where $\frac{\p u}{\p \bsnu}$ denotes the outward normal derivative of $u$.
\end{theorem}

We note that establishing \eqref{eq:GreensIdentity:Intro} requires 
exact parametrizations of the domain near the boundary so that the leading order corresponding integrals contained in the identity can be computed.
The proof is elaborate even for smooth functions $u$ and $v$,
which is shown first in \Cref{sec:GreensSmooth}. Then, the identity is proved in \Cref{sec:GreensWider} for more general functions based on further characterizations of the nonlocal operator and nonlocal energy space.

Using the bilinear form $
B_{\rho,\delta}(u,v)$
we can introduce weak formulations of nonlocal problems with classical boundary conditions by the following formal computation: Suppose that $u$ is smooth in $\Omega$ and satisfies the nonlocal Poisson problem for a given function $f$
\begin{equation}
        \cL_{\delta} u = f\,, \quad \text{ in } \Omega\,. \label{eq:Intro:NonlocalEq}
\end{equation}
Then for an arbitrary smooth $v$, the nonlocal Green's identity \eqref{eq:GreensIdentity:Intro} gives
\begin{equation*}
B_{\rho,\delta}(u,v)  = \int_{\Omega} f(\bx)  v(\bx) \, \rmd \bx + \int_{\p \Omega} \frac{\p u}{\p \bsnu}(\bx) v(\bx) \, \rmd \sigma(\bx) \,.
\end{equation*}

The treatment of the boundary term depends on the type of boundary conditions considered in the variational problem. In this work, for illustrative purposes, we treat Poisson problems with two different classes of boundary conditions: prescribed Dirichlet data
\begin{equation}\label{eq:Dirichlet:Inhomog:NonlocalBC}
		u = g \quad\text{ on } \p \Omega
\end{equation}
and prescribed Neumann data
\begin{equation}\label{eq:InHomogNeumannBC}
		\frac{\p u}{\p \bsnu} = g \quad \text{ on } \p \Omega\,.
\end{equation}

The energy space  associated with the energy $B_{\rho,\delta}(u,u)$ is a Hilbert space
defined as
\begin{equation*}
    \frak{W}^{\delta,2}(\Omega) := \{ \text{ closure of } C^{\infty}(\overline{\Omega}) \text{ with respect to the norm } \Vnorm{\cdot}_{\frak{W}^{\delta,2}(\Omega)} \}\,,
\end{equation*}
and where the norm $\Vnorm{\cdot}_{\frak{W}^{\delta,2}(\Omega)}$ is defined as
\begin{equation*}
    \begin{gathered}
        \Vnorm{u}_{\frak{W}^{\delta,2}(\Omega)}^2 := \Vnorm{u}_{L^2(\Omega)}^2 + [u]_{\frak{W}^{\delta,2}(\Omega)}^2\,,    \end{gathered}
\end{equation*}
with the semi-norm defined, for a prescribed constant
$ R \in (0,1)$, 
by
\begin{equation}\label{eq:NonlocalSeminormDef}
    [u]_{\frak{W}^{\delta,2}(\Omega)}^2 := \iintdm{\Omega}{\Omega \cap \{ |\by-\bx| \leq R \eta_{\delta}(\bx) \} }{  \frac{\big| u(\bx)-u(\by) \big|^2}{\eta_{\delta}(\bx)^{d+2} } }{\by}{\bx}\,. 
\end{equation}
It turns out that the semi-norm $[u]_{\frak{W}^{\delta,2}(\Omega)}$ is independent of the parameter $R \in (0,1)$ and comparable to $B_{\rho,\delta}(u,u)$ for any $\rho$ satisfying \eqref{Assump:KernelSmoothness}; see \Cref{thm:InvariantHorizon} and \Cref{thm:Coercivity}. 
The space $\frak{W}^{\delta,2}(\Omega)$ inherits the properties of similar function spaces that have heterogeneous localization at the boundary of $\Omega$; for instance, the main results of \cite{tian2017trace, du2022fractional} hold for  $\frak{W}^{\delta,2}(\Omega)$.
Further, by \Cref{thm:Embedding}
\begin{equation*}
    H^1(\Omega) \subset \frak{W}^{\delta,2}(\Omega) \subset L^2(\Omega)\,.
\end{equation*}
Note that $L^2(\Omega)$ and $H^1(\Omega)$ (as well as $H^1_0(\Omega)$) denote the conventional function spaces, see 
\Cref{sec:Notation} for other notation.

Equipped with sufficient knowledge of the nonlocal Hilbert space, we prove our
next set of main results: the variational formulation of each nonlocal boundary-value problem has a unique finite-energy solution that belongs to a weakly closed subset of the nonlocal Hilbert space. These well-posedness results for the four boundary-value problems are collected in \Cref{sec:BVPs}, but
we state here the formal well-posedness result for the 
case of the nonlocal Poisson problem  \eqref{eq:Intro:NonlocalEq} with the homogeneous Dirichlet boundary condition
\begin{equation}
        u = 0\,, \quad \text{ on } \p \Omega\,.
        \label{eq:Intro:HomogDirichletBC}
\end{equation}

\begin{theorem}
    Let $\Omega$ satisfy \eqref{assump:Domain}  and $\rho$ satisfy  \eqref{Assump:KernelSmoothness}-\eqref{Assump:KernelNormalization}. 
    Then for each $ \delta \in (0,\delta_0)$ there exists a weakly closed subspace $\frak{H}$ of the nonlocal Hilbert space $\frak{W}^{\delta,2}(\Omega)$ with 
    $$
    H^1_0(\Omega) \subset \frak{H} \subset L^2(\Omega)\,,
    $$ 
    such that the weak formulation of \eqref{eq:Intro:NonlocalEq} and \eqref{eq:Intro:HomogDirichletBC} is well-posed. That is, for every $f \in \frak{H}^*$ there exists a unique function $u \in \frak{H}$ such that
    \begin{equation}\label{eq:Intro:WeakForm}
        B_{\rho,\delta}(u,v) = \Vint{f,v}, \quad \forall v \in \frak{H}
    \end{equation}
    satisfying the energy estimate
    \begin{equation*}
        \Vnorm{u}_{\frak{H}} \leq C(d,\rho,\Omega) \Vnorm{f}_{\frak{H}^*}\,.
    \end{equation*}
\end{theorem}
The weakly closed subset of the nonlocal Hilbert space varies depending on the type of boundary conditions, that is, the choice of $\frak{H}$ varies from problem to problem.

Our third set of main results is the regularity of weak solutions to the well-posed nonlocal problems. The conclusion is formally summarized below, see \Cref{thm:Regularity:FixedDelta} for the precise statement.

\begin{theorem}
	Let $\Omega \subset \bbR^d$ satisfy \eqref{assump:Domain}. Suppose that $u \in \frak{W}^{\delta,2}(\Omega)$ is a weak solution of \eqref{eq:Intro:NonlocalEq}-\eqref{eq:Intro:HomogDirichletBC} i.e. $u$ satisfies \eqref{eq:Intro:WeakForm}, where additionally $f \in H^1(\Omega)$. Then there exists $C$ depending only on $d$, $\Omega$ and $\rho$ such that
	\begin{equation}\label{eq:Intro:H1apriori}
 \Vnorm{u}_{H^1(\Omega)}  
  \leq C \left( \Vnorm{f}_{\frak{H}^*(\Omega)} + \Vnorm{ \Phi_{\delta,2}^{-1} f}_{H^1(\Omega)} \right) \,,
	\end{equation}
        where $\Phi_{\delta,2}(\bx)$ is defined as in 
    \eqref{eq:ConvolutionOperatorAlpha}.
\end{theorem}
The heart of the matter in the proof is to rewrite 
\eqref{eq:Intro:NonlocalEq}, valid for almost every $\bx\in\Omega$,
as a pointwise relation
\begin{equation*}
	u(\bx) =  
 	K_{\delta,2} u (\bx) 
  + \frac{f(\bx)}{2 \Phi_{\delta,2}(\bx)}  \quad \text{ for a.e. } \bx \in \Omega\,,
\end{equation*}
where the operator $K_{\delta,\alpha}$ and the function $\Phi_{\delta,\alpha}$ for $\alpha\geq 0$ are defined as
\begin{equation}\label{eq:ConvolutionOperatorAlpha}
	K_{\delta,\alpha}u (\bx) := \frac{1}{\Phi_{\delta,\alpha}(\bx)} \int_{\Omega} \rho_{\delta,\alpha}(\bx,\by) u(\by) \, \rmd \by\,
 \;\text{ and } \;
\Phi_{\delta,\alpha}(\bx):= \int_{\Omega} \rho_{\delta,\alpha}(\bx,\by) \, \rmd \by\,,
\;\; \forall \bx \in \Omega\,.
\end{equation}

We call the operators $K_{\delta,\alpha}$ 
 {\em boundary-localized 
convolutions} as they are of the convolution type where $\eta_{\delta}$ is constant, which is the case away from the boundary, while the integral gets localized at the boundary.
Similar operators with mollification parameters that vanish at the boundary were introduced in \cite{burenkov1998sobolev, burenkov1982mollifying}. The particular form of operator we study in this work is similar to the operators considered in \cite{hintermuller2020variable}, where boundedness in classical function spaces and approximation properties were proved.
With the estimates on the boundary-localized convolutions in \Cref{sec:HSEstimates}, one can use the smoothness of $\rho$ and of $f$ to show that $u$ is in fact differentiable thanks to the convolution structure of the right-hand side.

We note that the $H^1$ regularity of the solutions to nonlocal problems with local boundary conditions studied here extends to the whole domain $\Omega$.  This serves as a motivation behind the use of local boundary conditions since this type of regularity does not hold in general without heterogeneous localization.
The fact that smooth data leads to a smooth nonlocal solution for the linear problems considered here highlights the difference from nonlocal boundary conditions, where smooth data may still yield 
 a nonsmooth nonlocal solution. The latter would be nonphysical in general in the linear regime.

The regularity results allow us to establish the convergence of the unique variational solution to \eqref{eq:Intro:NonlocalEqn} to the unique variational solution of \eqref{eq:Intro:LocalEqn} for both settings of boundary conditions, which is our fourth set of main results.
In fact, these convergence results remain true when the Poisson data $f$ in \eqref{eq:Intro:LocalEqn} belongs to the dual of the relevant Hilbert space. In the case of Dirichlet data, $f$ belongs to $[H^1_0(\Omega)]^* = H^{-1}(\Omega)$, and in the case of Neumann boundary data $f$ belongs to a subspace of $[H^1(\Omega)]^*$.
For both of the problems, we mollify the distribution $f$, so that the regularized $f_{\delta}$ 
 belongs to a smaller dual space and thus allows for a well-posed nonlocal problem. We then solve \eqref{eq:Intro:NonlocalEqn} with data $f_{\delta}$, and obtain convergence of solutions to the weak solution of \eqref{eq:Intro:LocalEqn}. 
The following is a formal statement in the case of homogeneous Dirichlet boundary data:
\begin{theorem}
    Let $\Omega$ satisfy \eqref{assump:Domain}  and $\rho$ satisfy  \eqref{Assump:KernelSmoothness}-\eqref{Assump:KernelNormalization}. 
    Let $f \in H^{-1}(\Omega) := [H^1_0(\Omega)]^*$ be a given function.
    Then for each $\delta \in (0,\delta_0)$, if there exists a regularized approximation $f_{\delta} \in \frak{H}^*$ of $f$ that satisfies the estimates \eqref{eq:ApproximationInequality} (see below) as well as
    \begin{equation*}
        f_{\delta} \rightharpoonup f \text{ weakly in } H^{-1}(\Omega)\,,
    \end{equation*}
    then the variational solutions $u_{\delta} \in \frak{H}$ to 
    \begin{equation}
    \label{eq:Intro:NonlocalEqReg}
            \cL_{\delta} u = f_{\delta}\,, \quad \text{ in } \Omega\,,
    \end{equation}
  and homogeneous Dirichlet boundary conditions  \eqref{eq:Intro:HomogDirichletBC}
    in fact belong to $H^1_0(\Omega)$ and converge weakly in $H^1(\Omega)$  as $\delta \to 0$ to a function $u \in H^1_0(\Omega)$. The function $u$ is the unique variational solution of the classical Poisson problem
  \begin{equation}\label{eq:Intro:LocalEqn}
	-\Delta u = f \text{ in } \Omega
\end{equation}
with the homogeneneous Dirichlet boundary conditions \eqref{eq:Intro:HomogDirichletBC}.
    Moreover, for any $f \in H^{-1}(\Omega)$ such a sequence $f_{\delta}$ is guaranteed to exist.
\end{theorem}
The uniform estimate \eqref{eq:ConvolutionOperatorAlpha}, combined with the convergence results from \Cref{sec:HSEstimates} and \Cref{sec:LocalLimit}, ensure that the limit function $u$ satisfies the classical Poisson equation.
These convergence results are stated precisely and proved for the general Dirichlet and Neumann problems in \Cref{sec:LocalLimit}.

To establish the main results as summarized above, we structure the rest of the paper as follows: 
first, a brief introduction to the local and nonlocal function spaces is given before we end \Cref{sec:Intro}. \Cref{sec:ParaProp} contains technical details necessary for manipulating objects that depend on $\rho$ and $\eta_{\delta}$. In \Cref{sec:NonlocalOp}, some properties of the nonlocal operator are studied in connection to their actions on function spaces, and the nonlocal Green's identity is shown for smooth functions. The nonlocal function space is further analyzed in \Cref{sec:EnergySpace}, providing necessary tools for investigating boundary-value problems,  including studies of the trace operator and the proof of the nonlocal Green's identity for wider classes of functions, together with the statements of  Poincar\'{e} inequalities.
The variational formulations of nonlocal problems with local boundary conditions and their well-posedness are then stated in \Cref{sec:BVPs}.
In \Cref{sec:HSEstimates}, some integral operators related to $\cL_{\delta}$ 
 and their local limits  are investigated, leading to the proofs of Poincar\'{e} inequalities and regularity studies on the solution. We then examine the consistency with the classical boundary-value problems in \Cref{sec:LocalLimit}.
 Finally, in \Cref{sec:Conclusion}, we discuss various 
issues such as the assumptions made in the current work,  some possible extensions and future research directions.

\subsection{Notation and function spaces}\label{sec:Notation}

For $r > 0$ and a bounded domain $\Omega \subset \bbR^d$, the open set $\Omega_r$ is defined as
\begin{equation*}
    \Omega_r := \{ \bx \in \Omega \, : \, \dist(\bx,\p \Omega) < r \}\,.
\end{equation*}

The set of infinitely-differentiable functions with support compactly contained in $\Omega$ is denoted $C^{\infty}_c(\Omega)$. The set of distributions is denoted $\cD'(\Omega)$. Lebesgue spaces are denoted $L^p(\Omega)$; Sobolev spaces $W^{k,p}(\Omega)$ for $k \in \bbN$ are defined via the integrability properties of weak derivatives as in \cite{evans2010partial}, with the convention $W^{k,2}(\Omega) = H^k(\Omega)$. 
The functional dual of a Banach space $V$ is denoted $V^*$.

For functions $u: \Omega \to \bbR$ and $v: \Omega \to \bbR$, we use $\Vint{u,v}$ to denote the standard $L^2$ inner product on their domain of definition  and also the duality pair where appropriate.
The integral average of $u$ on $\Omega$, or the duality pair with the constant function $v\equiv 1/|\Omega|$, is denoted by 
$$
(u)_\Omega := \fint_{\Omega} u(\bx) \, \rmd \bx :=\frac{1}{|\Omega|} \int_{\Omega} u(\bx) \, \rmd \bx
= \Vint{u, 1/|\Omega|}\,,$$
where $|\Omega|$ denotes the volume of $\Omega$.

The space $H^1_0(\Omega)$ is the closure of $C^{\infty}_c(\Omega)$ with respect to the $H^1(\Omega)$ norm.
We define the subspace $\mathring{H}^1(\Omega) = \{ u \in H^1(\Omega) \, : \, (u)_\Omega = 0 \}$.
We use the common convention $H^{-1}(\Omega) = [H^1_0(\Omega)]^*$.

It is expected that the local space $H^1(\Omega)$ is continuously embedded in the nonlocal energy space $\frak{W}^{\delta,2}(\Omega)$. We state the result in the following theorem, which is analogous to similar results given in \cite{tian2017trace,du2022fractional} concerning nonlocal energy spaces involving heterogeneous localization.

\begin{theorem}[Embedding]\label{thm:Embedding}
	Let $\Omega \subset \bbR^d$ satisfy \eqref{assump:Domain}. 
 Let $\rho$ be a kernel satisfying \eqref{Assump:KernelSmoothness}-\eqref{Assump:KernelNormalization}. Then there exists a constant $C = C(\Omega,\delta_0)$ such that for all $\delta < \delta_0$
	\begin{equation*}
		[u]_{\frak{W}^{\delta,2}(\Omega)} \leq C \Vnorm{\grad u}_{L^2(\Omega)}
	\end{equation*}
	for all $u \in H^1(\Omega)$.
\end{theorem}
\begin{proof}
    It suffices to prove the estimate for $u \in C^{\infty}(\overline{\Omega})$. By the mean value theorem,
    \begin{equation*}
        \begin{split}
            [u]_{ \frak{W}^{\delta,2}(\Omega) }^2 &= \int_{\Omega} \int_{|\bh| \leq R_0 \eta_{\delta}(\bx)} \frac{|u(\bx+\bh) - u(\bx)|^2}{\eta_{\delta}(\bx)^{d+2}} \, \rmd \bh \, \rmd \bx \\
            &\leq \int_{\Omega} \int_{|\bh| \leq R_0 \eta_{\delta}(\bx)} \int_0^1 |\grad u(\bx + t \bh)|^2 \, \rmd t \frac{|\bh|^2}{\eta_{\delta}(\bx)^{d+2}} \, \rmd \bh \, \rmd \bx \\
            &= \int_{\Omega} \int_{|\bh| \leq R_0} |\bh|^2 \int_0^1 |\grad u(\bx + t \eta_{\delta}(\bx)\bh)|^2 \, \rmd t  \, \rmd \bh \, \rmd \bx\,.
        \end{split}
    \end{equation*}
    On the domain of integration the function $\bg_{\bh}(\bx) := \bx + t \eta_{\delta}(\bx) \bh$ satisfies
    \begin{equation*}
        \det \grad \bg_{\bh}(\bx) = 1 + t \grad \eta_{\delta}(\bx) \cdot \bh > 1 - t \kappa_1 R_0 \sqrt{\delta} > 0
    \end{equation*}
    for all $\delta < \delta_0$.
    Therefore, the matrix inverse of $\grad \bg_{\bh}(\bx)$ is bounded from above. It follows that
    \begin{equation*}
        [u]_{ \frak{W}^{\delta,2}(\Omega) }^2 \leq C(\delta_0)  \int_{|\bh| \leq R_0} |\bh|^2 \int_{\Omega} |\grad u(\bx)|^2 \, \rmd \bx  \, \rmd \bh = C \Vnorm{\grad u}_{L^2(\Omega)}\,.
    \end{equation*}
\end{proof}

In a spirit similar to the studies in \cite{tian2017trace,du2022fractional}, we also have the independence of the nonlocal energy norms on the scaling factor used in the horizon.

\begin{theorem}[The nonlocal energy space is independent of the horizon]\label{thm:InvariantHorizon}
	Let $\Omega \subset \bbR^d$ satisfy \eqref{assump:Domain}. For $R \in (0,1)$, define the norm
    \begin{equation*}
        \Vnorm{u}_{\frak{W}^{\delta,2}(\Omega),R}^2 := \Vnorm{u}_{L^2(\Omega)}^2 + [u]_{\frak{W}^{\delta,2}(\Omega),R}^2\,,
    \end{equation*}
    where the semi-norm is defined as
    \begin{equation*}
        [u]_{\frak{W}^{\delta,2}(\Omega),R}^2 := \iintdm{\Omega}{\Omega \cap \{ |\by-\bx| \leq R \eta_{\delta}(\bx) \} }{  \frac{\big| u(\bx)-u(\by) \big|^2}{\eta_{\delta}(\bx)^{d+2} } }{\by}{\bx}\,.
    \end{equation*}
    Then, for constants $0 < r_0 \leq R_0 < 1$, there exists a constant $C$ depending only on $d$, $R_0$, $r_0$ and $\Omega$ such that
	\begin{equation*}
		C^{-1} \Vnorm{u}_{\frak{W}^{\delta,2}(\Omega),R_0} \leq \Vnorm{u}_{\frak{W}^{\delta,2}(\Omega),r_0} \leq C \Vnorm{u}_{\frak{W}^{\delta,2}(\Omega),R_0}
	\end{equation*}
	for any $ u \in C^{\infty}(\overline{\Omega})$.
\end{theorem}

\begin{proof}
	The proof follows essentially the same method from \cite[Lemma 6.2]{tian2017trace} and \cite[Lemma 2.2]{du2022fractional}.
\end{proof}

We now present the independence of the nonlocal energy norm space on the specific form of the nonlocal kernel $\rho$ so long the assumption \eqref{Assump:KernelSmoothness} holds.
A by-product is the continuity of the bilinear form.

\begin{theorem}[The nonlocal energy space is independent of the kernel]\label{thm:Coercivity}
	Let $\Omega \subset \bbR^d$ satisfy \eqref{assump:Domain} and $\rho$ satisfy \eqref{Assump:KernelSmoothness}. Then there exists a constant $C$ depending only on $d$, $\rho$ and $\Omega$ such that
	\begin{equation}\label{eq:KernelEquivalence}
		C^{-1} B_{\rho,\delta}(u,u) \leq [u]_{\frak{W}^{\delta,2}(\Omega)} \leq C B_{\rho,\delta}(u,u)\,,\quad \forall u \in \mathfrak{W}^{\delta,2}(\Omega)\,.
	\end{equation}
        Moreover, we have the continuity of the bilinear form in the nonlocal energy space
        \begin{equation}\label{eq:BilinearForm:Continuity}
		  B_{\rho,\delta}(u,v) \leq C \Vnorm{u}_{\frak{W}^{\delta,2}(\Omega)}  \Vnorm{v}_{\frak{W}^{\delta,2}(\Omega)} \,,\quad \forall u,v \in \mathfrak{W}^{\delta,2}(\Omega).
	\end{equation}
        Continuity also holds for $B_{\underline{\rho},\delta}$, where $\underline{\rho}$ is defined using $\rho$ as in \eqref{eq:AuxKernelDef}:
	\begin{equation}\label{eq:BilinearForm:Continuity:LowerKernel}
		  B_{\underline{\rho},\delta}(u,v) \leq C \Vnorm{u}_{\frak{W}^{\delta,2}(\Omega)}  \Vnorm{v}_{\frak{W}^{\delta,2}(\Omega)} \,,\quad \forall u,v \in \mathfrak{W}^{\delta,2}(\Omega).
	\end{equation}
\end{theorem}

\begin{proof}
	The first result  of the theorem follows in a straightforward way from the upper and lower bounds on $\rho$
 and \Cref{thm:InvariantHorizon}.
 The inequality \eqref{eq:BilinearForm:Continuity} then
 follows from H\"older's inequality and \eqref{eq:KernelEquivalence}. The inequality \eqref{eq:BilinearForm:Continuity:LowerKernel} follows similarly.
\end{proof}

\section{Properties of the heterogeneous localization function and the associated kernels}\label{sec:ParaProp}

We now present some properties related to the function $\eta_{\delta}$ and various kernels used in this work. First, we note that any
 kernel satisfying \eqref{Assump:KernelNormalization} also satisfies
\begin{equation}\label{eq:KernelNormalizationConsequence}
	\int_{B(0,1)} (\bz \otimes \bz) \rho(|\bz|) \, \rmd \bz = \bbI\,,
\end{equation}
since for each pair of indices $i$, $j$, a coordinate change by the appropriate rotation gives
\begin{equation*}
    \int_{B(0,1)} z_i z_j \rho(|\bz|) \, \rmd \bz = \delta_{ij} \int_{B(0,1)} z_i^2 \rho(|\bz|) \, \rmd \bz  = \frac{1}{d} \left( \int_{B(0,1)} |\bz|^2 \rho(|\bz|) \, \rmd \bz \right) \delta_{ij} = \delta_{ij}\,.
\end{equation*}
 
In addition, we adopt the notation that
 for vectors $\ba$ and $\bfb$ in $\bbR^d$, $\ba \otimes \bfb$ denotes the $d \times d$ matrix with $ij$ coordinate equal to $a_i b_j$, and $\bbI$ denotes the $d \times d$ identity matrix.

\subsection{Spatial variations of the heterogeneous localization function}\label{sec:LocalizationFunc}
For ease of access, we record the following comparisons of 
the heterogeneous localization function $\eta_{\delta}$ that are frequently referred to in later discussions:

\begin{lemma}\label{lma:ComparabilityOfXandY}
	Let $R_0 \in (0,1)$, and let $\Omega \subset \bbR^d$ satisfy \eqref{assump:Domain}. Then for all $\delta < \delta_0$,	\begin{equation}\label{eq:ComparabilityOfDistanceFxn1:SqDist}
		(1-\kappa_1 R_0 \sqrt{\delta}) \eta_{\delta}(\bx) \leq \eta_{\delta}(\by) \leq (1+\kappa_1 R_0 \sqrt{\delta}) \eta_{\delta}(\bx), \quad \forall \bx, \by \in \Omega \text{ with } |\bx-\by| \leq R_0 \eta_{\delta}(\bx)
	\end{equation}
	and 
	\begin{equation}\label{eq:ComparabilityOfDistanceFxn2:SqDist}
		(1-\kappa_1 R_0 \sqrt{\delta}) \eta_{\delta}(\by) \leq \eta_{\delta}(\bx) \leq (1+\kappa_1 R_0 \sqrt{\delta}) \eta_{\delta}(\by), \quad \forall \bx, \by \in \Omega \text{ with } |\bx-\by| \leq R_0 \eta_{\delta}(\by)\,.
	\end{equation}
\end{lemma}

\begin{proof}
	It suffices to show \eqref{eq:ComparabilityOfDistanceFxn1:SqDist}, since \eqref{eq:ComparabilityOfDistanceFxn2:SqDist} will follow from the same arguments with the roles of $\bx$ and $\by$ interchanged.
	From the properties of $\grad \eta_{\delta}$ and Taylor expansion we get
	\begin{equation*}
		\eta_{\delta}(\by) \leq \eta_{\delta}(\bx) + \kappa_1 \sqrt{\delta}|\bx-\by| \leq (1+\kappa_1 R_0 \sqrt{\delta} ) \eta_{\delta}(\bx)
	\end{equation*}
	and
	\begin{equation*}
		\eta_{\delta}(\bx) \leq \eta_{\delta}(\by) + \kappa_1 \sqrt{\delta}|\bx-\by| \leq  \eta_{\delta}(\by) + \kappa_1 R_0 \sqrt{\delta}  \eta_{\delta}(\bx)\,.
	\end{equation*}
\end{proof}

\begin{lemma}\label{lma:BoundaryRegions:SqDist}
	Let $R_0 \in (0,1)$. Let $\Omega \subset \bbR^d$ satisfy \eqref{assump:Domain}, and let $r > 0$ be such that $\bbR^d \setminus \Omega_{{r}} \neq \emptyset$. Then for all $\delta < \delta_0$
	\begin{equation}\label{eq:XSetInsideRLayer}
		\{ \by \, : \, |\bx-\by| \leq R_0 \eta_{\delta}(\bx) \} \subset \Omega_{{r}} \text{ whenever } \dist(\bx,\p \Omega) < \frac{r}{ 1 +  \bar{\kappa}_0 R_0 \sqrt{\delta}  }
	\end{equation}
	and
	\begin{equation}\label{eq:YSetInsideRLayer}
		\{ \by \, : \, |\bx-\by| \leq R_0 \eta_{\delta}(\by) \} \subset \Omega_{{r}} \text{ whenever } \dist(\bx,\p \Omega) < (1- \bar{\kappa}_0 R_0 \sqrt{\delta} ) r\,.
	\end{equation}	
 Also,
	\begin{equation}\label{eq:XSetOutsideRLayer}
		\{ \by \, : \, |\bx-\by| \leq R_0 \eta_{\delta}(\bx) \} \subset \bbR^d \setminus \Omega_{{r}} \text{ whenever } \dist(\bx,\p \Omega) \geq \frac{r}{ 1 - \bar{\kappa}_0 R_0 \sqrt{\delta}}
	\end{equation}
	and
\begin{equation}\label{eq:YSetOutsideRLayer}
		\{ \by \, : \, |\bx-\by| \leq R_0 \eta_{\delta}(\by) \} \subset \bbR^d \setminus \Omega_{{r}} \text{ whenever } \dist(\bx,\p \Omega) \geq (1 + \bar{\kappa}_0  R_0 \sqrt{\delta}) r\,.
	\end{equation}
\end{lemma}

\begin{proof}
	For \eqref{eq:XSetInsideRLayer}, since $\eta_{\delta}(\bx) \leq \bar{\kappa}_0 \min\{ \delta,(\dist(\bx,\p \Omega))^2 \}$
	\begin{equation*}
		\begin{split}
		\dist(\by,\p \Omega) &\leq \dist(\bx,\p \Omega) + |\bx-\by| \leq \dist(\bx,\p \Omega) + R_0 \eta_{\delta}(\bx) \\
		&\leq \dist(\bx,\p \Omega) + \bar{\kappa}_0 R_0 \sqrt{\delta} \dist(\bx, \p \Omega)\,,
		\end{split}
	\end{equation*}
	and so
	\begin{equation*}
		\dist(\by,\p \Omega) \leq  \big( 1 + \bar{\kappa}_0 R_0 \sqrt{\delta} \big)  (\dist(\bx,\p \Omega)) < r\,. 
	\end{equation*}
 
	For \eqref{eq:YSetInsideRLayer}, since $\eta_{\delta}(\by) \leq \bar{\kappa}_0 \min\{ \delta,(\dist(\by,\p \Omega))^2 \}$
	\begin{equation*}
		\begin{split}
		\dist(\by,\p \Omega) &\leq \dist(\bx,\p \Omega) + |\bx-\by| \leq \dist(\bx,\p \Omega) +
        R_0 \eta_{\delta}(\by) \\
		&\leq \dist(\bx,\p \Omega) + \bar{\kappa}_0 R_0 \sqrt{\delta} \dist(\by,\p \Omega)\,,
		\end{split}
	\end{equation*}
	and so
	\begin{equation*}
		(1- \bar{\kappa}_0 R_0 \sqrt{\delta} ) \dist(\by,\p \Omega) \leq \dist(\bx,\p \Omega) \leq (1- \bar{\kappa}_0 R_0 \sqrt{\delta} ) r\,.
	\end{equation*}

	For \eqref{eq:XSetOutsideRLayer},
	\begin{equation*}
		\begin{split}
			\dist(\by,\p \Omega) & \geq \dist(\bx,\p \Omega) - |\by-\bx| \geq \dist(\bx,\p \Omega) - \bar{\kappa}_0 R_0 \min \{ \delta, (\dist(\bx,\p \Omega))^2 \\
  & \geq \dist(\bx,\p \Omega) (1- \bar{\kappa}_0 R_0\sqrt{\delta})\,,
		\end{split}
	\end{equation*}
	and so
	\begin{equation*}
		\frac{\dist(\by,\p \Omega)}{(1 - \bar{\kappa}_0 R_0 \sqrt{\delta}) } \geq \dist(\bx,\p \Omega) \geq \frac{r}{(1- \bar{\kappa}_0 R_0 \sqrt{\delta})}\,.
	\end{equation*}
	
 For \eqref{eq:YSetOutsideRLayer},
	\begin{equation*}
		\begin{split}
			\dist(\by,\p \Omega) \geq \dist(\bx,\p \Omega) - |\by-\bx| \geq \dist(\bx,\p \Omega) - \bar{\kappa}_0 R_0 \min \{ \delta, (\dist(\by,\p \Omega))^2 \}\,,
		\end{split}
	\end{equation*}
	and so
	\begin{equation*}
		(1+\bar{\kappa}_0 R_0 \sqrt{\delta}) \dist(\by,\p \Omega) \geq \dist(\bx,\p \Omega) \geq (1+\bar{\kappa}_0 R_0 \sqrt{\delta}) r\,.
	\end{equation*}
\end{proof}

\subsection{Auxiliary kernels}\label{sec:AuxKernels}
To study properties of nonlocal problems corresponding to operators associated with the kernel $\rho$, such as the regularity of solutions, it is convenient to
introduce 
an additional kernel derived from $\rho$. We define
$\underline{\rho}$ as
\begin{equation}\label{eq:AuxKernelDef}
	\begin{split}
		\underline{\rho}(r) &:= \frac{-\rho'(r)}{r}\,, \qquad i.e. \quad \rho(r) = \int_r^1 s \underline{\rho}(s) \, \rmd s\,.
	\end{split}
\end{equation}
Then by the assumptions on $\rho$, $\underline{\rho}$ is in $C^0([0,\infty))$ with support in $[0,R_0]$, and satisfies
\begin{equation}\label{eq:normalr}
	\int_{B(0,1)} \underline{\rho}(|\bz|) \, \rmd \bz = \underline{\rho}_{d}\,, \text{ where } 
	\underline{\rho}_d := 
	\begin{cases}
        -2 \int_0^1 \frac{\rho'(r)}{r} \, \rmd r\,, & d =1 \,, \\
		2 \pi \rho(0)\,, & d = 2\,, \\
		\frac{1}{d-2} \int_{B(0,1)} \frac{\rho(|\bz|)}{|\bz|^2} \, \rmd \bz\,, & d \geq 3\,.
	\end{cases}
\end{equation}

We also define the rescaled kernels
\begin{equation*}
	\rho_{a,\alpha}(|\bz|) := \frac{1}{a^{d+\alpha}} \rho \left( \frac{|\bz|}{a} \right)\,, \text{ for any quantity } a > 0 \text{ and } \alpha \in \bbR\,.
\end{equation*}
When $\alpha = 0$ we write $\rho_{a}(|\bz|)=\rho_{a,0}(|\bz|)
$.
We note the difference in the notation: while $\rho_{\delta,\alpha}(|\bx-\by|) = \delta^{-d-\alpha} \rho(\frac{|\bx-\by|}{\delta})$, 
$\rho_{\delta,\alpha}(\bx,\by)$ is defined in \eqref{eq:OperatorKernelDef}. In fact, we have
\begin{equation*}
	2 \rho_{\delta,\alpha}(\bx,\by) = {\rho}_{\eta_{\delta}(\by),\alpha}(|\by-\bx|) + {\rho}_{\eta_{\delta}(\bx),\alpha}(|\by-\bx|) 
	=\frac{{\rho}_{\eta_{\delta}(\by)}(|\by-\bx|)}{\eta_{\delta}(\by)^{\alpha}} + \frac{{\rho}_{\eta_{\delta}(\bx)}(|\by-\bx|)}{\eta_{\delta}(\bx)^{\alpha}}\,.
\end{equation*}

Likewise, we use
$\tilde{\rho}_{\delta,\alpha}(\bx,\by)$ and $\tilde{\rho}_{\delta,\alpha}(|\bx-\by|)$ to denote the
 rescaled kernels corresponding to a generic  kernel $\tilde{\rho}$. 
Consequently, $\underline{\rho}_{\delta,\alpha}(\bx,\by)$ 
and $\underline{\rho}_{\delta,\alpha}(|\bx-\by|)$
denote the rescaled kernels corresponding to $\underline{\rho}$.

The integrals of the kernels $\Phi_{\delta,\alpha}(\bx)$ are defined as in \eqref{eq:ConvolutionOperatorAlpha}, and when $\alpha = 0$ we write $\Phi_{\delta}(\bx) := \Phi_{\delta,0}(\bx)$.

\subsection{The integral averages of the kernels}
\label{sec:AverageKernels}
\begin{lemma}\label{lma:KernelIntegral:SqDist}
	Suppose $\Omega \subset \bbR^d$ satisfies \eqref{assump:Domain}. Fix $R_0 \in (0,1)$ and let $\wt{\rho}$ be any function in $C^0([0,\infty))$ with support in $[0,R_0]$. 
	Then there exists a constant $C$ depending only on $d$, $\wt{\rho}$, $R_0$, $\Omega$ and $\kappa_1$ such that 
	\begin{equation}\label{eq:KernelIntFunction:UpperBound:SqDist}
		\int_{\Omega} \wt{\rho}_{\eta_{\delta}(\bx),\alpha}(|\by-\bx|) \, \rmd \by \leq \frac{C}{\eta_{\delta}(\bx)^\alpha} \text{ and } \int_{\Omega} \wt{\rho}_{\eta_{\delta}(\by),\alpha}(|\by-\bx|) \, \rmd \by \leq \frac{C}{\eta_{\delta}(\bx)^\alpha}
	\end{equation}
	for any $\bx \in \Omega$ and any $\alpha \in \bbR$.
	
	If in addition $\wt{\rho}(r) \geq \rho_0 > 0$ for all $r \in [0,r_0]$ for some $r_0 < R_0$, then 	\begin{equation}\label{eq:KernelIntFunction:LowerBound:SqDist}
		\int_{\Omega} \wt{\rho}_{\eta_{\delta}(\bx),\alpha}(|\by-\bx|) \, \rmd \by \geq \frac{1}{C \eta_{\delta}(\bx)^\alpha} \text{ and } \int_{\Omega} \wt{\rho}_{\eta_{\delta}(\by),\alpha}(|\by-\bx|) \, \rmd \by \geq \frac{1}{C \eta_{\delta}(\bx)^\alpha}
	\end{equation}
	for any $\bx \in \Omega$ and any $\alpha \in \bbR$.
\end{lemma}

\begin{proof}	
	Since $\supp \wt{\rho} \subset [0,R_0]$ we obtain one of the lower bounds
	\begin{equation*}
		\int_{\Omega} \wt{\rho}_{\eta_{\delta}(\bx),\alpha}(|\by-\bx|) \, \rmd \by \leq \Vnorm{\wt{\rho}}_{L^{\infty}([0,1])} \int_{ \{ |\by-\bx| \leq R_0 \eta_{ \delta}(\bx) \} } \frac{1}{\eta_{\delta}(\bx)^{d+\alpha}} \, \rmd \by = \frac{C(d,\wt{\rho})}{\eta_{\delta}(\bx)^{\alpha}}\,.
	\end{equation*}
	As for the upper bound on the same integral, since $\Omega$ satisfies 
	\begin{equation*}
		| B(\bx,r) \cap \Omega | \geq C r^d
	\end{equation*}
	for some constant $C = C(d,\Omega)$ for any $\bx \in \Omega$ and $r < \diam(\Omega)$,
	\begin{equation*}
		\int_{\Omega} \wt{\rho}_{\eta_{\delta}(\bx),\alpha}(|\by-\bx|) \, \rmd \by \geq \rho_0 \int_{ \Omega \cap \{ |\by-\bx| \leq r_0 \eta_{ \delta}(\bx) \} } \frac{1}{\eta_{\delta}(\bx)^{d+\alpha}} \, \rmd \by \geq \frac{C(d,\wt{\rho},\Omega)}{\eta_{\delta}(\bx)^{\alpha}}\,.
	\end{equation*}
	Therefore we need only to prove the same inequalities for the other piece:
	\begin{equation}\label{eq:IntOfDist:Pf1:UpperBd:SqDist}
		\int_{\Omega} \wt{\rho}_{\eta_{\delta}(\by),\alpha}(|\by-\bx|) \, \rmd \by \leq \frac{C(d,\Omega,\wt{\rho})}{\eta_{\delta}(\bx)^\alpha}
	\end{equation}
	and
	\begin{equation}\label{eq:IntOfDist:Pf1:LowerBd:SqDist}
		\int_{\Omega} \wt{\rho}_{\eta_{\delta}(\by),\alpha}(|\by-\bx|) \, \rmd \by \geq \frac{C(d,\Omega,\wt{\rho})}{\eta_{\delta}(\bx)^\alpha} \qquad \text{ if } \wt{\rho} \geq \rho_0 \text{ on } [0,r_0]\,.
	\end{equation}
	
	Both of these inequalities follow in the same way using \eqref{eq:ComparabilityOfDistanceFxn2:SqDist}:
	\begin{equation*}
		\int_{\Omega} \wt{\rho}_{\eta_{\delta}(\by),\alpha}(|\by-\bx|) \, \rmd \by \leq C \int_{ \{ |\by-\bx| \leq C_2 R_0 \eta_{\delta}(\bx) \} } \frac{1}{\eta_{\delta}(\bx)^{d+\alpha}} \wt{\rho} \left(  \frac{|\by-\bx|}{C_2 \eta_{\delta}(\bx)} \right) \, \rmd \by = \frac{C(d,\wt{\rho})}{\eta_{\delta}(\bx)^{\alpha}}
	\end{equation*}
	and
	\begin{equation*}
		\int_{\Omega} \wt{\rho}_{\eta_{\delta}(\by),\alpha}(|\by-\bx|) \, \rmd \by \geq C \rho_0 \int_{ \Omega \cap \{ |\by-\bx| \leq \frac{r_0}{C_2} \eta_{ \delta}(\bx) \} } \frac{1}{\eta_{\delta}(\bx)^{d+\alpha}} \, \rmd \by \geq \frac{C(d,\wt{\rho},\Omega)}{\eta_{\delta}(\bx)^{\alpha}}\,.
	\end{equation*}
\end{proof}

\begin{corollary}\label{cor:KernelIntegral:Weighted:SqDist}
	With the assumptions of \Cref{lma:KernelIntegral:SqDist}, for any $\alpha \geq 0$ and any $\beta \in [0,\alpha]$, there exists a constant $C(d,\wt{\rho},\alpha,\beta)>0$ such that
	\begin{equation}\label{eq:KernelIntFunction:Weighted:UpperBound:SqDist}
		\int_{\Omega} |\eta_{\delta}(\by)|^{\beta} \wt{\rho}_{\eta_{\delta}(\bx),\alpha}(|\by-\bx|) \, \rmd \by \leq \frac{C}{\eta_{\delta}(\bx)^{\alpha-\beta}} \text{ and } \int_{\Omega} |\eta_{\delta}(\by)|^{\beta} \wt{\rho}_{\eta_{\delta}(\by),\alpha}(|\by-\bx|) \, \rmd \by \leq \frac{C}{\eta_{\delta}(\bx)^{\alpha-\beta}}
	\end{equation}
	and
	\begin{equation}\label{eq:KernelIntFunction:Weighted:LowerBound:SqDist}
		\int_{\Omega} |\eta_{\delta}(\by)|^{\beta} \wt{\rho}_{\eta_{\delta}(\bx),\alpha}(|\by-\bx|) \, \rmd \by \geq \frac{1}{C \eta_{\delta}(\bx)^{\alpha-\beta}} \text{ and } \int_{\Omega} |\eta_{\delta}(\by)|^{\beta} \wt{\rho}_{\eta_{\delta}(\by),\alpha}(|\by-\bx|) \, \rmd \by \geq \frac{1}{C \eta_{\delta}(\bx)^{\alpha-\beta}}
	\end{equation}
	for any $\bx \in \Omega$.
\end{corollary}

We note that for any $\rho$ satisfying \eqref{Assump:KernelSmoothness}, the 
estimates stated for $\wt{\rho}$ in \Cref{lma:KernelIntegral:SqDist} and \Cref{cor:KernelIntegral:Weighted:SqDist}  
also hold for $\rho$.

From the definition \eqref{eq:ConvolutionOperatorAlpha}, the normalization in \eqref{eq:normalr}, and \Cref{lma:KernelIntegral:SqDist}, we also immediately get the following.
\begin{corollary}\label{cor:boundonPhi}
	With the assumptions of \Cref{lma:KernelIntegral:SqDist}, 
 there exist constants $\underline{\mu}_{\alpha}$ and $\bar{\mu}_{\alpha}$ depending only on $d$, $\rho$, $\alpha$ and $\Omega$ such that
\begin{equation}\label{eq:PhiBounds}
	0 < \underline{\mu}_{\alpha} \leq |\eta_{\delta}(\bx)|^{\alpha} \Phi_{\delta,\alpha}(\bx) \leq \bar{\mu}_{\alpha} < \infty, \qquad \forall \bx \in \Omega\,.
\end{equation}
\end{corollary}

\subsection{Properties of kernel derivatives}
\label{sec:KernelDeriv}

\begin{theorem}\label{thm:KernelDerivativeEstimates:SqDist}
Let  $\Omega \subset \bbR^d$ satisfy \eqref{assump:Domain} and $\rho$  satisfy \eqref{Assump:KernelSmoothness}, with 
$\underline{\rho}$ being defined as in \eqref{eq:AuxKernelDef}. Then for any $\alpha \in \bbR$, there exists $C = C(d,\rho,\Omega,\alpha)$ such that
\begin{equation}\label{eq:KernelDerivativeEstimate2:SqDist}
		|\grad _{\bx} \rho_{\delta,\alpha}(\bx,\by)| \leq C \Big( \underline{\rho}_{\eta_{\delta}(\bx),\alpha+1}(|\by-\bx|) 
		+ \rho_{\eta_{\delta}(\bx),\alpha+1}(|\by-\bx|)
		+ \underline{\rho}_{\eta_{\delta}(\by),\alpha+1}(|\by-\bx|) \Big)
	\end{equation}
	for all $\bx$, $\by \in \Omega$.
\end{theorem}

\begin{proof} 
	This follows by direct computation and the properties of $\rho$:
	\begin{equation*}
		\begin{split}
			\grad_{\bx} {\rho}_{\delta,\alpha}(\bx,\by) &= \frac{1}{2} \left( \frac{\underline{\rho}_{\eta_{\delta}(\bx),\alpha}(|\by-\bx|)}{\eta_{\delta}(\bx)^2} + \frac{\underline{\rho}_{\eta_{\delta}(\by),\alpha}(|\by-\bx|)}{\eta_{\delta}(\by)^2} \right)(\by-\bx) \\
			&\qquad - \frac{|\by-\bx|^2}{2 \eta_{\delta}(\bx)^3} \underline{\rho}_{\eta_{\delta}(\bx),\alpha}(|\by-\bx|)  \grad \eta_{\delta}(\bx) 
		 - \frac{d+\alpha}{2 \eta_{\delta}(\bx)} \rho_{\eta_{\delta}(\bx),\alpha}(|\by-\bx|) \grad \eta_{\delta}(\bx)\,.
		\end{split}
	\end{equation*}
	Thus, using the support of $\rho$, we see the result.
\end{proof}

Since $\rho$
and $\underline{\rho}$
satisfy the hypotheses for the upper bound of \Cref{lma:KernelIntegral:SqDist}, we have the following corollary:
\begin{corollary}
    Let $\Omega \subset \bbR^d$ satisfy \eqref{assump:Domain},  $\rho$ satisfy \eqref{Assump:KernelSmoothness}, $\alpha \geq 0$ and $\beta \in [0,\alpha+1]$. Then there exists $C = C(d,\rho,\Omega,\alpha,\beta)$ such that
\begin{equation}\label{eq:KernelIntegralDerivativeEstimates:SqDist}
		\int_{\Omega} |\eta_{\delta}(\by)|^{\beta} |\grad_{\bx} \rho_{\delta,\alpha}(\bx,\by)| \, \rmd \by
		\leq \frac{C}{\eta_{\delta}(\bx)^{1+\alpha-\beta}}
	\end{equation}
	for all $\bx \in \Omega$.
\end{corollary}

\begin{proof}
    We first apply  \eqref{eq:KernelDerivativeEstimate2:SqDist} then repeatedly use  \eqref{eq:ComparabilityOfDistanceFxn1:SqDist}-\eqref{eq:ComparabilityOfDistanceFxn2:SqDist} and \eqref{eq:KernelIntFunction:UpperBound:SqDist}.
\end{proof}

As a special case of \eqref{eq:KernelIntegralDerivativeEstimates:SqDist} we have 
the estimate for the derivative of $\Phi_{\delta,\alpha}$:

\begin{corollary}
    Let $\Omega \subset \bbR^d$ satisfy \eqref{assump:Domain},  $\rho$ satisfy \eqref{Assump:KernelSmoothness}, $\alpha \geq 0$ and $\beta \in [0,\alpha+1]$. Then there exists a
constant $C$ depending only on $d$, $\rho$, and $\Omega$ such that
\begin{equation}\label{eq:PhiDerivativeBounds}
	|\grad \Phi_{\delta,\alpha}(\bx)| \leq \int_{\Omega} |\grad_{\bx} \rho_{\delta,\alpha}(\bx,\by) | \, \rmd \by \leq \frac{C}{\eta_{\delta}(\bx)^{1+\alpha}} \quad \forall \, \bx \in \Omega.
\end{equation}
\end{corollary}

\section{The nonlocal operator and the nonlocal Green's identity}\label{sec:NonlocalOp}
We first present some estimates on the nonlocal operator which helps us establish the nonlocal Green's identity. In turn, the identity allows us to further interpret the results of the nonlocal operator acting on more general functions.

\subsection{The operator on Lebesgue and Sobolev spaces}\label{sec:OpLeb}

\begin{theorem}\label{thm:NonlocalOpWellDefd:Lp}
	Let $\Omega \subset \bbR^d$ satisfy \eqref{assump:Domain}, $\rho$ satisfy \eqref{Assump:KernelSmoothness}, and $p \in [1,\infty]$.
 Then $\cL_{\delta} u(\bx)$ defines a measurable function on $\Omega$ for
 any $u \in L^p(\Omega)$. Moreover,
  with $\Phi_{\delta,2}$  defined as in \eqref{eq:ConvolutionOperatorAlpha}, we have $\frac{\cL_{\delta} u(\bx)}{\Phi_{\delta,2}(\bx)} \in L^p(\Omega)$ and
 there exists a constant $C$ depending only on $d$, $p$ and $\Omega$ such that
	\begin{equation*}
		\Vnorm{\frac{\cL_{\delta} u}{\Phi_{\delta,2}}}_{L^p(\Omega)} \leq C \Vnorm{u}_{L^p(\Omega)}\,.
	\end{equation*}
\end{theorem}
\begin{proof}
	By H\"older's inequality,
	\begin{equation*}
		\begin{split}
			\Vnorm{\frac{\cL_{\delta} u}{\Phi_{\delta,2}}}_{L^p(\Omega)}^p &=  \int_{\Omega} \frac{2^p}{\Phi_{\delta,2}(\bx)^p} \left| \int_{\Omega} \rho_{\delta,2}(\bx,\by) (u(\bx)-u(\by)) \, \rmd \by \right|^p \, \rmd \bx \\
			&\leq \int_{\Omega} \frac{2^p}{\Phi_{\delta,2}(\bx)^p} \left( \int_{\Omega} \rho_{\delta,2}(\bx,\by) \, \rmd \by \right)^{p/p'} \left(  \int_{\Omega} \rho_{\delta,2}(\bx,\by)  |u(\bx)-u(\by)|^p \, \rmd \by \right) \, \rmd \bx  \\
			&= \int_{\Omega} \frac{2}{\Phi_{\delta,2}(\bx)} \int_{\Omega} \rho_{\delta,2}(\bx,\by)  |u(\bx)-u(\by)|^p \, \rmd \by \, \rmd \bx\,.
		\end{split}
	\end{equation*}
	Then by \eqref{eq:PhiBounds} and \Cref{lma:ComparabilityOfXandY}
	\begin{equation*}
	    \frac{\rho_{\delta,2}(\bx,\by)}{\Phi_{\delta,2}(\bx)} 
	    \leq \frac{1}{\kappa_{2,\ell}} \left( \frac{\rho \left( \frac{|\by-\bx|}{\eta_{\delta}(\bx)} \right)}{\eta_{\delta}(\bx)^d} + \frac{\rho \left( \frac{|\by-\bx|}{\eta_{\delta}(\by)} \right)}{\eta_{\delta}(\by)^d} \frac{\eta_{\delta}(\bx)^2}{\eta_{\delta}(\by)^2} \right) 
	    \leq C \left( \frac{\rho \left( \frac{|\by-\bx|}{\eta_{\delta}(\bx)} \right)}{\eta_{\delta}(\bx)^d} + \frac{\rho \left( \frac{|\by-\bx|}{\eta_{\delta}(\by)} \right)}{\eta_{\delta}(\by)^d} \right)
	\end{equation*}
	for all $\bx$ $\by \in \Omega$, so 
	\begin{equation*}
	    \begin{split}
	        \Vnorm{\frac{\cL_{\delta} u}{\Phi_{\delta,2}}}_{L^p(\Omega)}^p &\leq C \int_{\Omega} \int_{\Omega} \left( \rho_{\eta_{\delta}(\bx)}(|\by-\bx|) + \rho_{\eta_{\delta}(\by)}(|\by-\bx|) \right) |u(\bx)-u(\by)|^p \, \rmd \by \, \rmd \bx \\
			&\leq C \int_{\Omega} \int_{\Omega} \rho_{\eta_{\delta}(\bx)} \left( |\by-\bx| \right)  ( |u(\bx)|^p+|u(\by)|^p ) \, \rmd \by \, \rmd \bx \\
			&\qquad + C \int_{\Omega} \int_{\Omega}  \rho_{\eta_{\delta}(\by)}(|\by-\bx|) (|u(\bx)|^p+|u(\by)|^p )\, \rmd \by \, \rmd \bx\,.
	    \end{split}
	\end{equation*}
	Then by \Cref{lma:KernelIntegral:SqDist} we have
	\begin{equation*}
		\Vnorm{\frac{\cL_{\delta} u}{\Phi_{\delta,2}}}_{L^p(\Omega)}
  \leq C(\Omega,p) \Vnorm{u}_{L^p(\Omega)}
	\end{equation*}
	as desired. For $p = \infty$ and $\bx\in\Omega$,
	\begin{equation*}
		\left| \frac{\cL_{\delta}u(\bx)}{\Phi_{\delta,2}(\bx)} \right| = |u(\bx) - K_{\delta,2}u(\bx)| \leq 2 \Vnorm{u}_{L^{\infty}(\Omega)}\,.
	\end{equation*}
\end{proof}

By \Cref{thm:OpIsDefined:Inhomog:SqDist}, for any $u \in L^p(\Omega)$, $\cL_{\delta} u(\bx)$ coincides with a Lebesgue-measurable function $\Phi_{\delta,2}(\bx) w(\bx)$ on $\Omega$, where $w \in L^p(\Omega)$. Since $\Phi_{\delta,2}(\bx) \approx \frac{1}{\eta_{\delta}(\bx)^2}$,  $\cL_{\delta} u(\bx)$  clearly belongs to $L^p_{loc}(\Omega)$, but it may not be globally integrable. 

If we assume that $u$ is smooth we can get a better result. To do this,  we first present the following pointwise characterization.

\begin{lemma}\label{lma:OpTaylorExp:SqDist}
	 Let $\Omega \subset \bbR^d$ satisfy \eqref{assump:Domain}  and $\rho$ satisfy  \eqref{Assump:KernelSmoothness}. Suppose that $u \in C^2(\overline{\Omega})$.Then for all $\delta \in (0,\delta_0)$ the quantity $\cL_{\delta} u(\bx)$ defines a function in $L^{\infty}(\Omega)$. More precisely
	\begin{equation}\label{eq:OpTaylorExp:Def}
		\cL_{\delta} u(\bx) =   2 \bF_1(\bx) \cdot \grad u(\bx) + F_2(\bx, u)\,,
	\end{equation} 
	where $\bF_1 \in L^{\infty}(\Omega;\bbR^d)$ 
	is defined by
\begin{equation}\label{eq:OpTaylorExp:F1Def}
		\bF_1(\bx) := 
  \int_{\Omega} \rho_{\delta,2}(\bx,\by) (\by-\bx) \, \rmd \by
	\end{equation}
	and satisfies	\begin{equation}\label{eq:OpTaylorExp:F1Est}
		\sup_{\bx \in \Omega} |\bF_1(\bx)| \leq C \chi_{ \{ \dist(\bx,\p \Omega) \leq \bar{C} \sqrt{\delta} \} }\,, \quad \text{ for } C(\rho,\Omega) >0 \text{ and } \bar{C}(\rho,\Omega) > 0\,,
	\end{equation}
	and $F_2 \in L^{p}(\Omega)$ depends on $\rho$, $\Omega$, $\delta$ and $\grad^2 u$, and satisfies
	\begin{equation}\label{eq:OpTaylorExp:F2Est}
		 \Vnorm{F_2(\cdot,u)}_{L^p(\Omega)} \leq C(\rho,\Omega) \Vnorm{\grad^2 u}_{L^{p}(\Omega)} \quad \text{ for any } p \in [1,\infty]\,.
	\end{equation}
\end{lemma}

\begin{proof}
	Taylor expansion gives 
	\begin{equation*}
		\begin{split}
			-\cL_{\delta} u(\bx) &= 2\int_{\Omega} \rho_{\delta,2}(\bx,\by) (\by-\bx) \, \rmd \by \cdot \grad u (\bx) \\
			&\qquad + 2\int_{\Omega} \rho_{\delta,2}(\bx,\by) \int_0^1 \Vint{ \grad^2 u(\by + t(\bx-\by)) (\bx-\by),(\bx-\by)} (1-t) \, \rmd t \, \rmd \by\,.
		\end{split}
	\end{equation*}
	By \Cref{lma:KernelIntegral:SqDist} applied to the kernel $r^2 \rho(r)$ and $\alpha = 0$, the $L^p(\Omega)$ norm of the last term (which defines $F_2$) is bounded by $C(\rho,\Omega) \Vnorm{\grad^2 u}_{L^{p}(\Omega)}$.
	Further, since the set $\{ \by \, : \, |\bx-\by| < R_0 \eta_{\delta}(\bx) \}$ is compactly contained in $\Omega$ for all $\bx \in \Omega$, 
	\begin{equation*}
            \int_{\Omega} \rho_{\eta_{\delta}(\bx),2} (|\by-\bx|) (\by-\bx) \, \rmd \by = \frac{1}{\eta_{\delta}(\bx)} \int_{B(0,R_0)} \bz \rho(|\bz|) \, \rmd \bz = {\bf 0} ,\quad\forall \bx \in \Omega\,.
	\end{equation*}
	So we just need to show that	\begin{equation}\label{eq:DefOp:Pf:MainEst:SqDist}
		\left| \bF_1(\bx) \right| =  \left| \int_{\Omega} \rho_{\eta_{\delta}(\by),2} (|\by-\bx|) (\by-\bx) \, \rmd \by  \right| \leq C(\rho,d,\Omega) \chi_{ \{ \dist(\bx,\p \Omega) \leq \bar{C} \sqrt{\delta} \} }  ,\qquad\forall \bx \in \Omega\,.
	\end{equation}
	Fix $\bx \in \Omega$. Note that the integral is absolutely convergent and bounded by $C \eta_{\delta}(\bx)^{-1}$ by \Cref{lma:KernelIntegral:SqDist} applied to the kernel $r \rho(r)$ and $\alpha = 1$. We will use a coordinate change.
	Define 
	$$
	\bg_{\bx}(\by) := \frac{\by-\bx}{\eta_{\delta}(\by)}\,.
	$$
	Then
	\begin{equation*}
		\grad \bg_{\bx}(\by) = \frac{1}{\eta_{\delta}(\by)} \bbI - (\by-\bx) \otimes \frac{\grad \eta_{\delta}(\by)}{\eta_{\delta}(\by)^2} = \frac{1}{\eta_{\delta}(\by)} \left( \bbI -  \frac{(\by-\bx)}{\eta_{\delta}(\by)}  \otimes \grad \eta_{\delta}(\by)  \right)\,.
	\end{equation*}
	Now, recalling the definition of $\delta_0$ in \eqref{assump:Localization}, $\bg_{\bx}(\by)$ is one-to-one on the boundary of the set $\{ \by \, : \, |\by-\bx| \leq \frac{R_0}{1- \kappa_1 R_0 \sqrt{\delta}} \eta_{\delta}(\bx) \}$,
	which for $\delta < \delta_0$ 
	contains $\{ \by \, : \, |\by-\bx| \leq R_0 \eta_{\delta}(\by) \}$  by \eqref{eq:ComparabilityOfDistanceFxn2:SqDist} and is contained in $\Omega$.
	Additionally, we use \eqref{eq:ComparabilityOfDistanceFxn1:SqDist} and that $|\grad \eta_{\delta}(\by)| \leq \kappa_1 \sqrt{\delta}$ to obtain
	\begin{equation*}
		\eta_{\delta}(\by) + \grad \eta_{\delta}(\by) \cdot (\bx-\by) \geq \eta_{\delta}(\by) - \kappa_1 \sqrt{\delta} |\by-\bx| \geq \left(1 - \frac{2 \kappa_1 \sqrt{\delta} R_0}{1 - \kappa_1 \sqrt{\delta} R_0 } \right) \eta_{\delta}(\bx)  > 0
	\end{equation*}
	for all $\by \in \Omega$ satisfying $|\by-\bx| \leq \frac{R_0}{1-\kappa_1 R_0 \sqrt{\delta}} \eta_{\delta}(\bx) $ (which again holds assuming that $\delta < \delta_0$).
	Therefore on this set,
	$$
	\det \grad \bg_{\bx}(\by) = \frac{\eta_{\delta}(\by) + \grad \eta_{\delta}(\by) \cdot (\bx-\by)}{\eta_{\delta}(\by)^{d+1}}
	$$ 
	is bounded away from zero independent of $\by$.
	It follows that $\bg_{\bx}(\by)$ is globally invertible on $\{ \by \, : \, |\by-\bx| \leq R_0 \eta_{\delta}(\by) \}$ and so
	\begin{equation}\label{eq:DefOp:Pf1:SqDist}
		\begin{split}
			&
   \bF_1(\bx)= \int_{\Omega} \rho_{\eta_{\delta}(\by),2} (|\by-\bx|) (\by-\bx) \, \rmd \by \\
			&=  \frac{1}{\eta_{\delta}(\bx)} \int_{ \{ |\bx-\by| < R_0 \eta_{\delta}(\by) \} } \rho \left( \frac{|\by-\bx|}{\eta_{\delta}(\by)} \right) \frac{(\by-\bx)}{\eta_{\delta}(\by)} \frac{\eta_{\delta}(\by) + \grad \eta_{\delta}(\by) \cdot (\bx-\by) }{\eta_{\delta}(\by)^{d+1}}  \, \rmd \by \\	
			&\quad + \frac{1}{\eta_{\delta}(\bx)} \int_{ \{ |\bx-\by| < R_0 \eta_{\delta}(\by) \} } \rho \left( \frac{|\by-\bx|}{\eta_{\delta}(\by)} \right) \frac{(\by-\bx)}{\eta_{\delta}(\by)} \frac{\eta_{\delta}(\bx) -  \eta_{\delta}(\by) - \grad \eta_{\delta}(\by) \cdot (\bx-\by) }{\eta_{\delta}(\by)^{d+1}}  \, \rmd \by \\
			&=  \frac{1}{\eta_{\delta}(\bx)} \int_{B(0,R_0)} \rho \left( |\bz| \right) \bz \, \rmd \bz \\	
			&\quad + \frac{1}{\eta_{\delta}(\bx)} \int_{ \{ |\bx-\by| < R_0 \eta_{\delta}(\by) \} } \rho \left( \frac{|\by-\bx|}{\eta_{\delta}(\by)} \right) \frac{(\by-\bx)}{\eta_{\delta}(\by)} \frac{\eta_{\delta}(\bx) -  \eta_{\delta}(\by) - \grad \eta_{\delta}(\by) \cdot (\bx-\by) }{\eta_{\delta}(\by)^{d+1}}  \, \rmd \by \\
			&= \frac{1}{\eta_{\delta}(\bx)} \int_{ \{ |\bx-\by| < R_0 \eta_{\delta}(\by) \} } \rho \left( \frac{|\by-\bx|}{\eta_{\delta}(\by)} \right) \frac{(\by-\bx)}{\eta_{\delta}(\by)} \frac{\eta_{\delta}(\bx) -  \eta_{\delta}(\by) - \grad \eta_{\delta}(\by) \cdot (\bx-\by) }{\eta_{\delta}(\by)^{d+1}}  \, \rmd \by\,.
		\end{split}
	\end{equation}
	
	Now, note by \eqref{eq:YSetOutsideRLayer} with $r = \frac{\delta}{\kappa_0}$ that whenever $\dist(\bx,\p \Omega) \geq \frac{1+\bar{\kappa}_0 R_0}{\sqrt{\kappa_0}} \sqrt{\delta}$,
	\begin{equation*}
		\begin{split}
			& \bF_1(\bx)= \int_{\Omega} \rho_{\eta_{\delta}(\by),2} (|\by-\bx|) (\by-\bx) \, \rmd \by \\
			&= \frac{1}{\eta_{\delta}(\bx)} \int_{ \{ |\bx-\by| < R_0 \eta_{\delta}(\by) \} } \rho \left( \frac{|\by-\bx|}{\eta_{\delta}(\by)} \right) \frac{(\by-\bx)}{\eta_{\delta}(\by)} \frac{\eta_{\delta}(\bx) -  \eta_{\delta}(\by) - \grad \eta_{\delta}(\by) \cdot (\bx-\by) }{\eta_{\delta}(\by)^{d+1}}  \, \rmd \by \\
			&=\frac{1}{\delta} \int_{ \{ |\bx-\by| < R_0 \delta \} } \rho \left( \frac{|\by-\bx|}{\delta} \right) \frac{(\by-\bx)}{\delta} \frac{\delta -  \delta }{\delta^{d+1}}  \, \rmd \by = {\bf 0}
		\end{split}
	\end{equation*}
	by definition of $\eta_{\delta}$. 
	Therefore $\bF_1(\bx)$ is only nonzero for $\dist(\bx,\p \Omega) \leq \bar{C} \sqrt{\delta}$, where $\bar{C} := \frac{1 + \bar{\kappa}_0 R_0}{\kappa_0}$.
	For these values of $\bx$, we use the conditions on $\eta_{\delta}$ along with \eqref{eq:ComparabilityOfDistanceFxn2:SqDist} and \Cref{lma:KernelIntegral:SqDist} applied to the kernel $r^3 \rho(r)$ and $\alpha = 0$ to estimate the last line of \eqref{eq:DefOp:Pf1:SqDist} by
	\begin{equation}\label{eq:YIntegralRemainderTerm:SqDist}
		C \Vnorm{\grad^2 \eta_{\delta}}_{L^{\infty}(\Omega)} \int_{ \{ |\bx-\by| < R_0 \eta_{\delta}(\by) \} } \rho \left( \frac{|\by-\bx|}{\eta_{\delta}(\by)} \right) \frac{|\by-\bx|^3}{\eta_{\delta}(\by)^3} \frac{1}{\eta_{\delta}(\by)^{d}} \frac{\eta_{\delta}(\by)}{\eta_{\delta}(\bx)}  \, \rmd \by \leq C(\rho,\Omega)\,.
	\end{equation}
	Thus \eqref{eq:DefOp:Pf:MainEst:SqDist} holds and the proof is complete.
\end{proof}

We note that understanding the impact of the non-vanishing $\bF_1$ in 
\eqref{eq:OpTaylorExp:F1Def} near the boundary can be very useful in practice, for example, in  the study of {\em ghost forces} for coupled local and nonlocal models \cite{tao2019nonlocal}. 
The pointwise characterization \eqref{eq:OpTaylorExp:Def} also leads to the following estimate for smooth functions.

\begin{theorem}\label{thm:OpIsDefined:Inhomog:SqDist}
	Let $\Omega \subset \bbR^d$ satisfy \eqref{assump:Domain}  and $\rho$ satisfy  \eqref{Assump:KernelSmoothness}.
	Let $p \in [1,\infty]$, and let $u \in W^{2,p}(\Omega)$. Then $\cL_{\delta} u(\bx) \in L^p(\Omega)$, and there exists a constant $C$ depending only on $d$, $p$ and $\Omega$ such that
	\begin{equation}\label{eq:OpIsDefined:Inhomog:SqDist}
		\Vnorm{\cL_{\delta} u}_{L^p(\Omega)} \leq C \Vnorm{u}_{W^{2,p}(\Omega)}\,.
	\end{equation}
\end{theorem}

\begin{proof}
	Let $\{u_n\}$ be a sequence in $C^2(\overline{\Omega})$ that converges to $u$ in $W^{2,p}(\Omega)$.
	Thanks to \Cref{thm:NonlocalOpWellDefd:Lp} the sequence $\frac{\cL_{\delta}u_n}{\Phi_{\delta,2}}$ converges to $\frac{\cL_{\delta}u}{\Phi_{\delta,2}}$ in $L^p(\Omega)$ as $n \to \infty$, and so since $\Phi_{\delta,2}(\bx) > 0$ for all $\bx \in \Omega$ it follows that $\cL_{\delta}u_n(\bx)$ converges to $\cL_{\delta}u(\bx)$ as $n \to \infty$ for almost every $\bx \in \Omega$.
	Therefore by Fatou's lemma and \Cref{lma:OpTaylorExp:SqDist}
	\begin{equation*}
		\begin{split}
			\Vnorm{ \cL_{\delta} u}_{L^p(\Omega)} \leq \liminf_{n \to \infty} \Vnorm{ \cL_{\delta} u_n }_{L^p(\Omega)} 
			&\leq \liminf_{n \to \infty}  \left( \Vnorm{F_1 \grad u_n}_{L^p(\Omega)} +  \Vnorm{F_2(\cdot,u_n) }_{L^p(\Omega)} \right) \\
			&\leq C(\rho,\Omega)  \liminf_{n \to \infty}  \left(  \Vnorm{\grad u_n}_{L^p(\Omega)} +  \Vnorm{ \grad^2 u_n }_{L^p(\Omega)} \right) \\
			&\leq C \Vnorm{ u }_{W^{2,p}(\Omega)}\,.
		\end{split}
	\end{equation*}
	
\end{proof}

\subsection{The operator on the nonlocal function space}\label{sec:OpNonlocalSpace}

Concerning the action of the nonlocal operator on functions in the nonlocal energy space, we have
\begin{theorem}\label{thm:NonlocalOpWellDefd:Energy}
	Let $\Omega \subset \bbR^d$ satisfy \eqref{assump:Domain}  and $\rho$ satisfy  \eqref{Assump:KernelSmoothness}. Let $u \in \frak{W}^{\delta,2}(\Omega)$. Then $\frac{\cL_{\delta} u(\bx)}{\sqrt{\Phi_{\delta,2}(\bx)}} \in L^2(\Omega)$, and there exists constant $C$ depending only on $\rho$ and $\Omega$ such that
	\begin{equation*}
		\Vnorm{\frac{\cL_{\delta} u}{\sqrt{\Phi_{\delta,2}}}}_{L^2(\Omega)} \leq C [u]_{\frak{W}^{\delta,2}(\Omega)}\,.
	\end{equation*}
\end{theorem}
\begin{proof}
	The conclusion follows from an application of H\"older's inequality and the definition of $\Phi_{\delta,2}$:
	\begin{equation*}
		\begin{split}
			\Vnorm{\frac{\cL_{\delta} u}{\sqrt{\Phi_{\delta,2}}}}_{L^2(\Omega)}^2 &= \int_{\Omega} \frac{2^2}{\Phi_{\delta,2}(\bx)} \left| \int_{\Omega} \rho_{\delta,2}(\bx,\by) (u(\bx)-u(\by)) \, \rmd \by \right|^2 \, \rmd \bx \\
			&\leq \int_{\Omega} \frac{2^2}{\Phi_{\delta,2}(\bx)} \left( \int_{\Omega} \rho_{\delta,2}(\bx,\by) \, \rmd \by \right) \left(  \int_{\Omega} \rho_{\delta,2}(\bx,\by)  |u(\bx)-u(\by)|^2 \, \rmd \by \right) \, \rmd \bx  \\
			&= 4\int_{\Omega} \int_{\Omega} \rho_{\delta,2}(\bx,\by)  |u(\bx)-u(\by)|^2 \, \rmd \by \, \rmd \bx \leq C [u]_{\frak{W}^{\delta,2}(\Omega)}^{2}\,.
		\end{split}
	\end{equation*}
\end{proof}

\subsection{Nonlocal Green's identity for smooth functions}\label{sec:GreensSmooth}

\begin{theorem}\label{thm:GreensIdentity:SqDist}
	Let $\Omega \subset \bbR^d$ satisfy \eqref{assump:Domain} and $\rho$ satisfy \eqref{Assump:KernelSmoothness}-\eqref{Assump:KernelNormalization}. Suppose $u$, $v \in C^2(\overline{\Omega})$. Then \eqref{eq:GreensIdentity:Intro} holds. That is, 
	\begin{equation*} 
            B_{\rho,\delta}(u,v)= \int_{\Omega} \cL_{\delta} u(\bx) \cdot v(\bx) \, \rmd \bx + \int_{\p \Omega} \frac{\p u}{\p \bsnu}(\bx) \cdot v(\bx) \, \rmd \sigma(\bx) \,.
	\end{equation*}
\end{theorem}

We begin with two lemmas.

\begin{lemma}\label{lma:GreensId1:SqDist}
	Let $\Omega \subset \bbR^d$ satisfy \eqref{assump:Domain}. Suppose $\rho \in C^0([0,\infty))$ with support in $[0,R_0]$ for fixed $R_0 \in (0,1)$. Then there exist constants $C = C(\rho,R_0,\bar{\kappa_0},\delta) >0 $ and $b = b(R_0,\bar{\kappa}_0,\delta) > 1$ for which the following holds: for any $\veps >0$ with $\veps \ll \delta$
	\begin{equation*}
		\int_{\Omega \setminus \Omega_{\sqrt{\veps}}} \rho_{\eta_{\delta}(\by),2}(|\by-\bx|) \, \rmd \by = \int_{\Omega_{\sqrt{b\veps}} \setminus \Omega_{\sqrt{\veps}}} \rho_{\eta_{\delta}(\by),2}(|\by-\bx|) \, \rmd \by \leq C  \veps^{-2} \chi_{ \Omega_{\sqrt{\veps}} \setminus \Omega_{\sqrt{\veps / b}} }(\bx)
	\end{equation*}
	for all $\bx \in \Omega_{\sqrt{\veps}}$.
\end{lemma}

\begin{proof}
	Note that by \Cref{lma:KernelIntegral:SqDist} the integral is finite for any fixed $\bx \in \Omega$.

 Set $b := \frac{1}{(1 - \bar{\kappa}_0 R_0 \sqrt{\delta})^2}$. First, for $\veps \ll \delta$ we apply \eqref{eq:YSetInsideRLayer} with $r = \sqrt{b \veps}$ to see that $\{ \by \, : \, |\by-\bx| \leq R_0 \eta_{\delta}(\by) \} \subset \Omega_{\sqrt{b\veps}}$ whenever $\dist(\bx,\p \Omega) <  \sqrt{\veps}$.
    This gives the first equality.
	
	Second, for $\veps \ll \delta$ we again apply \eqref{eq:YSetInsideRLayer} but with $r = \sqrt{\veps}$ to see that, whenever $\dist(\bx,\p \Omega) < \sqrt{\frac{\veps}{b}}$, we have
 $\{ \by \, : \, |\by-\bx| \leq R_0 \eta_{\delta}(\by) \} \subset \Omega_{\sqrt{\veps}}$.
	Therefore thanks to the support of $\rho$, the integral is zero if $\bx \in \Omega_{ \sqrt{\veps/b} }$.
	Thus by \Cref{lma:KernelIntegral:SqDist}
	\begin{equation*}
		\int_{\Omega \setminus \Omega_{\sqrt{\veps}}} \rho_{\eta_{\delta}(\by),2}(|\by-\bx|) \, \rmd \by \leq \frac{ C \chi_{ \Omega_{\sqrt{\veps}} \setminus \Omega_{\sqrt{\veps / b}} }(\bx)}{\eta_{\delta}(\bx)^2} \leq \frac{ C \chi_{ \Omega_{\sqrt{\veps}} \setminus \Omega_{\sqrt{\veps / b}} }(\bx)}{\veps^2}\,. 
	\end{equation*}
\end{proof}

\begin{lemma}\label{lma:GreensId2:SqDist}
		Let $\Omega \subset \bbR^d$ satisfy \eqref{assump:Domain}. Suppose $\rho \in C^0([0,\infty))$ with support in $[0,R_0]$ for fixed $R_0 \in (0,1)$. Then there exist constants $C = C(\rho,R_0,\bar{\kappa_0},\delta) >0 $ and $a = a(R_0,\bar{\kappa}_0,\delta) > 1$ for which the following holds: for any $\veps >0$ with $\veps \ll \delta$
	\begin{equation*}
		\int_{\Omega \setminus \Omega_{\sqrt{\veps}}} \rho_{\eta_{\delta}(\bx),2}(|\by-\bx|) \, \rmd \by = \int_{\Omega_{\sqrt{a\veps}} \setminus \Omega_{\sqrt{\veps}}} \rho_{\eta_{\delta}(\bx),2}(|\by-\bx|) \, \rmd \by \leq C  
  \veps^{-2}
  \chi_{ \Omega_{\sqrt{\veps}} \setminus \Omega_{\sqrt{\veps / a}} }(\bx) 
	\end{equation*}
	for all $\bx \in \Omega_{\sqrt{\veps}}$.
\end{lemma}

\begin{proof}
Note that by \Cref{lma:KernelIntegral:SqDist} the integral is finite for any fixed $\bx \in \Omega$.

Set $a := (1+ \bar{\kappa}_0 R_0 \sqrt{\delta})^2$. First, for $\veps \ll \delta$ we apply \eqref{eq:XSetInsideRLayer} with $r = \sqrt{a\veps}$ to see that, whenever $\dist(\bx,\p \Omega) <  \sqrt{\veps}$,
$\{ \by \, : \, |\by-\bx| \leq R_0 \eta_{\delta}(\bx) \} \subset \Omega_{\sqrt{a \veps}}$.
This gives the first equality.

Second, for $\veps \ll \delta$ we again apply \eqref{eq:XSetInsideRLayer} but with $r = \sqrt{\veps}$ to see that $\{ \by \, : \, |\by-\bx| \leq R_0 \eta_{\delta}(\bx) \} \subset \Omega_{\sqrt{\veps}}$ whenever $\dist(\bx,\p \Omega) <  \sqrt{\frac{\veps}{a}}$.
Therefore thanks to the support of $\rho$ the integral is zero if $\bx \in \Omega_{ \sqrt{\veps/a}  }$.
Thus by \Cref{lma:KernelIntegral:SqDist}
\begin{equation*}
	\int_{\Omega \setminus \Omega_{\sqrt{\veps}}} \rho_{\eta_{\delta}(\bx),2}(|\by-\bx|) \, \rmd \by \leq \frac{ C \chi_{ \Omega_{\sqrt{\veps}} \setminus \Omega_{\sqrt{\veps / a}} }(\bx)}{\eta_{\delta}(\bx)^2} \leq \frac{ C \chi_{ \Omega_{\sqrt{\veps}} \setminus \Omega_{\sqrt{\veps / a}} }(\bx)}{\veps^2}\,. 
\end{equation*}
\end{proof}

\begin{proof}[Proof of \Cref{thm:GreensIdentity:SqDist}]
	Let $\veps > 0$ with $\veps \ll \delta$. Since
	$$
	\rho_{\delta,2}(\bx,\by) |u(\bx)-u(\by)| \, |v(\bx)-v(\by)| \leq \Vnorm{\grad u}_{L^{\infty}(\Omega)} \Vnorm{\grad v}_{L^{\infty}(\Omega)} \rho_{\delta,2}(\bx,\by) |\bx-\by|^2 \in L^1(\Omega \times \Omega)\,,
	$$
	by the dominated convergence theorem
	\begin{equation*}
		\begin{split}
		&B_{\rho,\delta}(u,v) =\iintdm{\Omega}{\Omega}{ \rho_{\delta,2}(\bx,\by) (u(\bx)-u(\by)) (v(\bx)-v(\by)) }{\by}{\bx} \\
		&\quad= \lim\limits_{\veps \to 0} \iintdm{\Omega \setminus \Omega_{\sqrt{\veps}}}{\Omega \setminus \Omega_{\sqrt{\veps}} }{ \rho_{\delta,2}(\bx,\by) (u(\bx)-u(\by)) (v(\bx)-v(\by)) }{\by}{\bx}\,.
		\end{split}
	\end{equation*}
	Now we can use nonlocal integration by parts on the truncated form as well as linearity of the integrals thanks to their absolute convergence guaranteed by \Cref{lma:GreensId1:SqDist} and \Cref{lma:OpTaylorExp:SqDist}; we get
	\begin{equation*}
		\begin{split}
			&B_{\rho,\delta}(u,v) =\iintdm{\Omega \setminus \Omega_{\sqrt{\veps}} }{\Omega \setminus \Omega_{\sqrt{\veps}} }{ \rho_{\delta,2}(\bx,\by) (u(\bx)-u(\by)) (v(\bx)-v(\by)) }{\by}{\bx} \\
			&= 2\int_{\Omega \setminus \Omega_{\sqrt{\veps}}} \int_{\Omega \setminus \Omega_{\sqrt{\veps}}} \rho_{\delta,2}(\bx,\by) (u(\bx)-u(\by)) \, \rmd \by \; v(\bx) \, \rmd \bx \\
			&= - 2 \int_{\Omega \setminus \Omega_{\sqrt{\veps}}} \int_{\Omega_{\sqrt{\veps}}} \rho_{\delta,2}(\bx,\by) (u(\bx)-u(\by)) v(\bx) \, \rmd \by  \, \rmd \bx + \int_{\Omega \setminus \Omega_{\sqrt{\veps}}} \cL_{\delta} u(\bx) v(\bx) \, \rmd \bx \\
			&= - 2 \int_{\Omega \setminus \Omega_{\sqrt{\veps}}} \int_{\Omega_{\sqrt{\veps}}} \rho_{\delta,2}(\bx,\by) (u(\bx)-u(\by))  v(\by) \, \rmd \by \, \rmd \bx \\
			&\quad  - 2 \int_{\Omega \setminus \Omega_{\sqrt{\veps}}} \int_{\Omega_{\sqrt{\veps}}} \rho_{\delta,2}(\bx,\by) (u(\bx)-u(\by)) ( v(\bx) - v(\by) ) \, \rmd \by \, \rmd \bx 
			+ \int_{\Omega \setminus \Omega_{\sqrt{\veps}}} \cL_{\delta} u(\bx) v(\bx) \, \rmd \bx\,.
		\end{split}
	\end{equation*}
	The dominated convergence theorem gives that the second integral vanishes and that
	\begin{equation*}
		\lim\limits_{\veps \to 0} \int_{\Omega \setminus \Omega_{\sqrt{\veps}}} \cL_{\delta} u(\bx) v(\bx) \, \rmd \bx = \int_{\Omega} \cL_{\delta} u(\bx) v(\bx) \, \rmd \bx\,,
	\end{equation*}
	since $|\cL_{\delta}u(\bx) v(\bx)| \leq C(\rho,d,\Omega) \Vnorm{u}_{W^{2,\infty}(\Omega)} \Vnorm{v}_{L^{\infty}(\Omega)} < \infty$.
	Therefore it remains to show that
	\begin{equation}\label{eq:GreensIdPf:Limit1:SqDist}
		\lim\limits_{\veps \to 0}  \int_{\Omega_{\sqrt{\veps}}} \int_{\Omega \setminus \Omega_{\sqrt{\veps}}} 2 \, \rho_{\delta,2}(\bx,\by) (u(\bx)-u(\by)) v(\bx) \, \rmd \by  \, \rmd \bx = \int_{\p \Omega} \frac{\p u}{\p \bsnu}(\bx) v(\bx) \, \rmd \bx\,.
	\end{equation}
	Now on the left-hand side integrand we perform a second-order Taylor expansion of $u(\by)$ at $\bx$ and estimate the remainder term involving $\grad^2 u$. Using the arguments of \Cref{lma:GreensId1:SqDist} and \Cref{lma:GreensId2:SqDist} we see that
	\begin{equation*}
		\int_{\Omega_{\sqrt{\veps}}} \int_{\Omega \setminus \Omega_{\sqrt{\veps}}} \rho_{\delta,2}(\bx,\by) |\bx-\by|^2 \, \rmd \by  \, \rmd \bx \leq C(\Omega) |\Omega_{\sqrt{\veps}}| = o(1)\,,
	\end{equation*}
	so \eqref{eq:GreensIdPf:Limit1:SqDist} is equivalent to showing that
	\begin{equation}\label{eq:GreensIdPf:Limit2:SqDist}
		\lim\limits_{\veps \to 0}  \int_{\Omega_{\sqrt{\veps}}} \int_{\Omega \setminus \Omega_{\sqrt{\veps}}} 2 \, \rho_{\delta,2}(\bx,\by) (\bx-\by) \cdot \grad u(\bx) v(\bx) \, \rmd \by  \, \rmd \bx = \int_{\p \Omega} \frac{\p u}{\p \bsnu}(\bx) v(\bx) \, \rmd \bx \,.
	\end{equation}
	
	\textit{Step 1}: Localization. For any $\bx_0 \in \p \Omega$, denote the outward and inward unit normal vectors as $\bsnu(\bx_0)$ and $\bn(\bx_0)$, respectively.
	Since $\Omega$ is a $C^2$ domain, it satisfies the uniform interior-exterior sphere condition. As a consequence, there exists $\veps_0 > 0$ depending only on $\Omega$ such that the following holds: for any $\bx_0 \in \p \Omega$ there exists a Euclidean ball $B = B(\bx_0, \beta) \subset \bbR^d$ such that the map $\Psi : \p \Omega \cap B \times [0,\veps_0) \to \overline{\Omega_{2\veps_0}} \cap B(\bx_0,2 \beta)$ given by
	\begin{equation*}
		\Psi(\bx',t) := \bx' + t \bn(\bx')
	\end{equation*}
	is a diffeomorphism onto its image.
	Furthermore, (choosing $\veps_0$ smaller if necessary) for every $\bx \in \Omega_{\veps_0}$ there exists a unique nearest-point projection $\bsxi(\bx) \in \p \Omega$ such that $|\bx-\bsxi(\bx)| = \dist(\bx, \p \Omega)$. We collect the relevant computations concerning the distance function and the map $\Psi$:
	\begin{equation}\label{eq:GreensIdPf:Parametrization:SqDist}
		\begin{split}
			\dist(\bx' + t \bn(\bx') ,\p \Omega ) &= t ,\quad\forall (\bx',t) \in \p \Omega \cap B \times [0,\veps_0)\,, \\
			\bsxi(\bx'+t \bn(\bx')) & = \bx' ,\quad\forall (\bx',t) \in \p \Omega \cap B \times [0,\veps_0)\,, \\
			\grad \dist(\bx' + t \bn(\bx') ,\p \Omega) &= \bn(\bx') ,\quad\forall (\bx',t) \in \p \Omega \cap B \times [0,\veps_0)\,.
		\end{split}
	\end{equation}
	See \cite{gilbarg2015elliptic} for further details. It is clear upon parametrizing $\p \Omega$ (choosing the radius $\beta$ smaller if necessary so that $\p \Omega \cap B$ can be written as the graph of a function) that \begin{equation}\label{eq:GreensIdPf:Parametrization:SqDist:2}
	       \det \grad \Psi(\bx',0)  \rmd \scH^{d-1}(\bx') = \rmd \sigma(\bx') \text{ on } \p \Omega \cap B\,,
	\end{equation}
	where $\scH^{d-1}$ denotes $d-1$-dimensional Hausdorff measure. Hereafter we abbreviate $\rmd \scH^{d-1}(\bx')$ as $\rmd \bx'$ whenever it is clear from context.
	
	Now, there exists a finite collection of Euclidean balls $\{ B_i \}_{i=1}^N$, $B_i = B((\bx_0)_i,\beta_i)$ whose centers are points on $\p \Omega$ that cover $\p \Omega$ and such that \eqref{eq:GreensIdPf:Parametrization:SqDist}-\eqref{eq:GreensIdPf:Parametrization:SqDist:2} are satisfied for $\p \Omega \cap B_i$. 
	Therefore by a localization argument with a partition of unity subordinate to the cover 
	$\{ B_i \}_{i=1}^N$ it suffices to find the limit \eqref{eq:GreensIdPf:Limit2:SqDist} for $\grad u$ and $v$ compactly supported on $\overline{\Omega} \cap B_i$ for some $i$.
	In what follows we suppress the index to simplify notation.
	
	\textit{Step 2}: Let $\veps \ll \min \{ \veps_0,\delta \}$. We show that
	\begin{equation}\label{eq:GreensIdPf:Step2:SqDist}
		\begin{split}
			\lim\limits_{\veps \to 0}  &\int_{\Omega_{\sqrt{\veps}}} \int_{\Omega \setminus \Omega_{\sqrt{\veps}}} \rho_{\eta_{\delta}(\bx),2}(|\bx-\by|) (\by-\bx) \cdot \grad u(\bx) v(\bx) \, \rmd \by  \, \rmd \bx \\
			&= - \int_{B(0,R_0) \cap \{ z_d > 0 \} } z_d^2 \rho \left( |\bz| \right) \, \rmd \bz \cdot \int_{\p \Omega} \frac{\p u}{\p \bsnu}(\bx) v(\bx) \, \rmd \sigma (\bx)\,,
		\end{split}
	\end{equation}
	where $u$ and $v$ are supported on $\p \Omega \cap B$ where $B$ is as in Step 1.
	To begin, we write the outer integral using the change of variables $\bx = \Psi(\bx',t)$ as well as using \eqref{eq:GreensIdPf:Parametrization:SqDist} and \Cref{lma:GreensId2:SqDist}:
	\begin{equation*}
		\begin{split}
& \int_{\Omega_{\sqrt{\veps}}}  \int_{\Omega \setminus \Omega_{\sqrt{\veps}}} \rho_{\eta_{\delta}(\bx),2}(|\bx-\by|) (\by-\bx) \cdot \grad u(\bx) v(\bx) \, \rmd \by  \, \rmd \bx \\
&\; =\int_{\Omega_{\sqrt{\veps}}}  \int_{\bbR^d} \chi_{\Omega_{\sqrt{a \veps}} \setminus \Omega_{\sqrt{\veps}} }(\by+\bx) \; \rho_{\eta_{\delta}(\bx),2}(|\by|) \by \cdot \grad u(\bx) v(\bx) \, \rmd \by  \, \rmd \bx \\
			&\; = \int_{\p \Omega \cap B} \int_0^{\sqrt{\veps}} \int_{\bbR^d} \left\{  \chi_{\Omega_{\sqrt{a \veps}} \setminus \Omega_{\sqrt{\veps}} }( \Psi(\bx',t) + \by )\; \frac{1}{t^{2d+4}} \rho \left( \frac{|\by|}{t^2} \right) \right. \\
		&\left. \qquad\qquad\qquad 
   \by \cdot \grad u(\Psi(\bx',t)) v(\Psi(\bx',t)) \det \grad \Psi(\bx',t) \, \right\} \, \rmd \by  \, \rmd t \, \rmd \bx' \\
			&= \int_{\p \Omega \cap B} \int_0^{\sqrt{\veps}} \int_{\bbR^d} 
  \left\{ \chi_{\Omega_{\sqrt{a \veps}} \setminus \Omega_{\sqrt{\veps}} }( \Psi(\bx',t) + t^2 \by ) \; \frac{1}{t^{2}} \rho \left( |\by| \right) \right.\\
			&\qquad \qquad  \qquad \left.\by \cdot \grad u(\Psi(\bx',t)) v(\Psi(\bx',t)) \det \grad \Psi(\bx',t) \right\} \, \rmd \by  \, \rmd t \, \rmd \bx'\,.
		\end{split}
	\end{equation*}
	Now, fix $\bx' \in \p \Omega \cap B$. We define
 $\Vint{\bn(\bx')}^\perp := \{ \by' \in \bbR^d \, : \, \by' \cdot \bn(\bx') = 0 \}$. Use the change of variables $\by = \by' + \tau \bn(\bx')$ for $\by' \in \Vint{\bn(\bx')}^\perp$ in the inner integral as well as the fact that $\supp \rho \subset [0,R_0]$ to get
	\begin{equation*}
		\begin{split}	&\int_{\Omega_{\sqrt{\veps}}}  \int_{\Omega \setminus \Omega_{\sqrt{\veps}}} \rho_{\eta_{\delta}(\bx),2}(|\bx-\by|) (\by-\bx) \cdot \grad u(\bx) v(\bx) \, \rmd \by  \, \rmd \bx \\
			&\; = \int_{\p \Omega \cap B} \int_0^{\sqrt{\veps}} \int_{\Vint{\bn(\bx') }^\perp } \int_{\bbR}
   \left\{\chi_{\Omega_{\sqrt{a \veps}} \setminus \Omega_{\sqrt{\veps}}} ( \Psi(\bx',t) + t^2 (\by' + \tau \bn(\bx') ) ) \; \chi_{B(0,R_0)} ( \by' + \tau \bn(\bx')  ) \right. \\
			&\left. \quad\quad  
   \frac{\rho\left( |\by' + \tau \bn(\bx')| \right) }{t^2} \big( \by' + \tau \bn(\bx') \big) \cdot \grad u(\Psi(\bx',t)) v(\Psi(\bx',t)) \det \grad \Psi(\bx',t) \,  \right\} \rmd \tau \, \rmd \by'  \, \rmd t \, \rmd \bx'\,.
		\end{split}
	\end{equation*}
	
	Now we analyze the indicator functions. Taylor-expanding $\eta_{\delta}(\bx) = (\dist(\bx,\p \Omega))^2$ at the point $\Psi(\bx',t) = \bx' + t \bn(\bx')$ gives
	\begin{equation*}
		\begin{split}
			\eta_{\delta}(\Psi(\bx',t) + t^2 (\by' + \tau \bn(\bx') )) &= t^2 + 2 t \bn(\bx') \cdot (t^2 (\by' + \tau \bn(\bx') ) ) + \Pi(t) \\
			&= t^2 (1+ 2 t \tau) + \Pi(t)\,.
		\end{split}
	\end{equation*}
	The function $\Pi(t)$ depends on the Hessian of $(\dist(\bx,\p \Omega))^2$, and is also a function of  $\bx' \in \p \Omega$, $\by' \in \Vint{\bn(\bx') }^\perp$ and $\tau \in \bbR$. However, since $\Omega$ is $C^2$ it is continuous and uniformly bounded with respect to these variables and with respect to $\veps \in (0,\veps_0)$, and satisfies
	\begin{equation*}
		\frac{1}{M} t^4 \leq \Pi(t) \leq M t^4\,,
	\end{equation*}
	see \cite{gilbarg2015elliptic}.
	Therefore
	\begin{equation*}
		\begin{split}
	&	\int_{\Omega_{\sqrt{\veps}}}  \int_{\Omega \setminus \Omega_{\sqrt{\veps}}} \rho_{\eta_{\delta}(\bx),2}(|\bx-\by|) (\by-\bx) \cdot \grad u(\bx) v(\bx) \, \rmd \by  \, \rmd \bx \\
			&\;= \int_{\p \Omega \cap B} \int_0^{\sqrt{\veps}} \int_{\Vint{\bn(\bx') }^\perp } \int_{\bbR}
    \left\{\right. \chi_{ \{\veps < t^2(1+2 t\tau) + \Pi(t) < a \veps \} } \; \chi_{B(0,R_0)} ( \by' + \tau \bn(\bx')  ) \frac{\rho \left( |\by' + \tau \bn(\bx')| \right) 
    }{t^2}  \\
			&\left. \qquad\qquad 
    \big( \by' + \tau \bn(\bx') \big) \cdot \grad u(\Psi(\bx',t)) v(\Psi(\bx',t)) \det \grad \Psi(\bx',t) \,    \right\} \, \rmd \tau  \, \rmd \by'  \, \rmd t \, \rmd \bx' \\
			&\;= \int_{\p \Omega \cap B} \int_0^{1} \int_{\Vint{\bn(\bx') }^\perp } \int_{\bbR}
       \left\{\right. \chi_{ \{1 < t^2(1+2 \sqrt{\veps} t\tau) + \frac{\Pi( \sqrt{\veps} t)}{\veps} < a \} } \; \chi_{B(0,R_0)} ( \by' + \tau \bn(\bx')  ) \frac{\rho \left( |\by' + \tau \bn(\bx')| \right)}{\sqrt{\veps} t^2} \\
			&\left. \qquad\qquad  
    \big( \by' + \tau \bn(\bx') \big) \cdot \grad u(\Psi(\bx',\sqrt{\veps}t)) v(\Psi(\bx',\sqrt{\veps}t)) \det \grad \Psi(\bx',\sqrt{\veps}t) \,  \right\}\, \rmd \tau  \, \rmd \by'  \, \rmd t \, \rmd \bx'\,.
		\end{split}
	\end{equation*}
 Define 
 $$
 h_{\veps,\tau} (t) := t^2(1+2 \sqrt{\veps} t\tau)\,, \qquad 
 \wt{h}_{\veps,\tau}(t) := h_{\veps,\tau} (t)  + \frac{\Pi( \sqrt{\veps} t)}{\veps}\,.
 $$
For $\tau \in (-R_0,R_0)$ and  $0 < \veps \ll \delta$ both $h_{\veps,\tau}$ and $\wt{h}_{\veps,\tau}$ are strictly increasing for $t \in [0,1]$.
	It is also straightforward to see that for $\veps > 0$ sufficiently small 
	$0 \leq \wt{h}_{\veps,\tau}^{-1}(1) \leq 1 \leq \wt{h}_{\veps,\tau}^{-1}(a)$ for all $t \in (0,1)$ and for all $\tau \in (-R_0,R_0)$. 
	It is also clear that on the region of integration the implicit relation $\tau \geq \frac{\Pi(\sqrt{\veps})}{2\veps}$ is satisfied. Further, the Lipschitz continuity of $\grad u$, $v$ and $\grad \Psi$ means that the limit as $\veps \to 0$ will be unchanged if we replace $\sqrt{\veps}t$ in their arguments with $0$. Combining all of this, we use Fubini's theorem to get
	\begin{equation*}
		\begin{split}
			&\int_{\p \Omega \cap B} \int_{\Vint{\bn(\bx') }^\perp } \int_0^{1}  \int_{\bbR} 
   \left\{ \right. \chi_{ \{1 < \wt{h}_{\veps,\tau}(t) < a \} } \;  \chi_{B(0,R_0)} ( \by' + \tau \bn(\bx')  )  \frac{\rho \left( |\by' + \tau \bn(\bx')| \right)
   }{\sqrt{\veps} t^2}  \\
			&\left. \qquad\qquad  
   \big( \by' + \tau \bn(\bx') \big) \cdot \grad u(\Psi(\bx',0)) v(\Psi(\bx',0)) \det \grad \Psi(\bx',0) \,  \right\}\, \rmd \tau  \, \rmd t  \, \rmd \by'  \, \rmd \bx' \\
			&= \int_{\p \Omega \cap B} \int_{\Vint{\bn(\bx') }^\perp } \int_{ \frac{\Pi(\sqrt{\veps})}{2 \veps} }^{R_0} 
   \left\{ \right. \left( \int_{ \wt{h}_{\veps,\tau}^{-1}(1)  }^1 \frac{1}{\sqrt{\veps} t^2}  \, \rmd t\right) \chi_{B(0,R_0)} ( \by' + \tau \bn(\bx')  )   \\
			&\left. \qquad\qquad  
  \rho \left( |\by' + \tau \bn(\bx')| \right) \big( \by' + \tau \bn(\bx') \big) \cdot \grad u(\bx') v(\bx') \det \grad \Psi(\bx',0)\right\}  \, \rmd \tau    \, \rmd \by'  \, \rmd \bx' \\
			&= \int_{\p \Omega \cap B} \int_{\Vint{\bn(\bx') }^\perp } \int_{ \frac{\Pi(\sqrt{\veps})}{2 \veps} }^{R_0} \left\{ \right. \frac{1}{\sqrt{\veps} } \left( \frac{1}{\wt{h}_{\veps,\tau}^{-1}(1)} - 1 \right)\,  \chi_{B(0,R_0)} ( \by' + \tau \bn(\bx')  ) \rho \left( |\by' + \tau \bn(\bx')| \right)  \\
			&\left. \qquad\qquad 
  \big( \by' + \tau \bn(\bx') \big) \cdot \grad u(\bx') v(\bx') \det \grad \Psi(\bx',0)\,\right\} \, \rmd \tau    \, \rmd \by'  \, \rmd \bx'\,.
		\end{split}
	\end{equation*}
	Now, set 
	\begin{equation*}
		A := \wt{h}_{\veps,\tau}^{-1}(1)\,, \qquad B := h_{\veps,\tau}^{-1}(1)\,.
	\end{equation*}
	Then it is clear by their definitions that both $A$ and $B$ are both positive and bounded from above and below away from zero for all $\tau \in (-R_0,R_0)$ and all $\veps > 0$ sufficiently small, with bounds depending only on $R_0$. Thus
	\begin{equation*}
		\left| \frac{1}{ A } - \frac{1}{ B } \right| = \frac{1}{|A| |B| } \left| A - B \right| \leq C |h_{\veps,\tau}(A) - h_{\veps,\tau}(B)| = |1 - \frac{\Pi(\sqrt{\veps} A)}{\veps}  - 1| \leq C \veps\,.
	\end{equation*}
	So we have
	\begin{equation*}
		\begin{split}
			&\int_{\p \Omega \cap B} \int_{\Vint{\bn(\bx') }^\perp } \int_{ \frac{\Pi(\sqrt{\veps})}{2 \veps} }^{R_0}  
     \left\{ \right.
     \frac{1}{\sqrt{\veps} } \left( \frac{1}{\wt{h}_{\veps,\tau}^{-1}(1)} - 1 \right)  \chi_{B(0,R_0)} ( \by' + \tau \bn(\bx')  )   \rho \left( |\by' + \tau \bn(\bx')| \right) \\
			&\left. \qquad\qquad  
  \big( \by' + \tau \bn(\bx') \big) \cdot \grad u(\bx') v(\bx') \det \grad \Psi(\bx',0)  \right\} \, \rmd \tau    \, \rmd \by'  \, \rmd \bx' \\
			&\;= \int_{\p \Omega \cap B} \int_{\Vint{\bn(\bx') }^\perp } \int_{ \frac{\Pi(\sqrt{\veps})}{2 \veps} }^{R_0}    \left\{ \right. \frac{1}{\sqrt{\veps} } \left( \frac{1}{h_{\veps,\tau}^{-1}(1)} - 1 \right)  \chi_{B(0,R_0)} ( \by' + \tau \bn(\bx')  ) \rho \left( |\by' + \tau \bn(\bx')| \right) \\
			&\left. \qquad\qquad \big( \by' + \tau \bn(\bx') \big) \cdot \grad u(\bx') v(\bx') \det \grad \Psi(\bx',0) 
  \right\}
   \, \rmd \tau    \, \rmd \by'  \, \rmd \bx' + O(\sqrt{\veps})\,.
		\end{split}
	\end{equation*}
	Thanks to the explicit formula for $h_{\veps,\tau}(t)$, we can find the unique real root of the function $h_{\veps,\tau}(t) - 1$ in terms of $\veps$ and $\tau$. Then a direct computation gives
	\begin{equation*}
		\lim\limits_{\veps \to 0} \frac{1}{\sqrt{\veps} } \left( \frac{1}{h_{\veps,\tau}^{-1}(1)} - 1 \right)  = \tau\,.
	\end{equation*}
	By the dominated convergence theorem we can conclude the equality of limits
	\begin{equation*}
		\begin{split}
		&	\lim\limits_{\veps \to 0}  \int_{\Omega_{\sqrt{\veps}}}  \int_{\Omega \setminus \Omega_{\sqrt{\veps}}} \rho_{\eta_{\delta}(\bx),2}(|\bx-\by|) (\by-\bx) \cdot \grad u(\bx) v(\bx) \, \rmd \by  \, \rmd \bx \\
			&\;=\lim\limits_{\veps \to 0} \int_{\p \Omega \cap B} \int_{\Vint{\bn(\bx') }^\perp } \int_{ \frac{\Pi(\sqrt{\veps})}{2 \veps} }^{R_0} \left\{   \frac{1}{\sqrt{\veps} } \left( \frac{1}{h_{\veps,\tau}^{-1}(1)} - 1 \right)  \chi_{B(0,R_0)} ( \by' + \tau \bn(\bx')  ) \right. \\
			&\left. \qquad\qquad \rho \left( |\by' + \tau \bn(\bx')| \right) \big( \by' + \tau \bn(\bx') \big) \cdot \grad u(\bx') v(\bx') \det \grad \Psi(\bx',0) \, \right\}\,\rmd \tau    \, \rmd \by'  \, \rmd \bx' \\
			&\;= \int_{\p \Omega \cap B} \int_{\Vint{\bn(\bx') }^\perp } \int_{ 0 }^{R_0}  \left\{ \right. \chi_{B(0,R_0)} ( \by' + \tau \bn(\bx')  ) \tau \, \rho \left( |\by' + \tau \bn(\bx')| \right)  \\
			&\left. \qquad\qquad \big( \by' + \tau \bn(\bx') \big) \cdot \grad u(\bx') v(\bx') \det \grad \Psi(\bx',0) \,\right\} \, \rmd \tau    \, \rmd \by'  \, \rmd \bx'\,.
		\end{split}
	\end{equation*}
	We can see that
	\begin{equation*}
		\begin{split}
			& \int_{\p \Omega \cap B} \int_{\Vint{\bn(\bx') }^\perp } \int_{ 0 }^{R_0} \left\{ \chi_{B(0,R_0)} ( \by' + \tau \bn(\bx')  )  \tau \, \rho \left( |\by' + \tau \bn(\bx')| \right)\right.  \\
			&\left. \qquad 
   \by' \cdot \grad u(\bx') v(\bx') \det \grad \Psi(\bx',0) \,\right\}\, \rmd \tau    \, \rmd \by'  \, \rmd \bx' = 0\,,
		\end{split}
	\end{equation*}
	since the region of integration is radially symmetric with respect to $\by'$. Further, since $(\by' + \tau \bn(\bx')) \cdot \bn(\bx') = \tau$ and
 $\{ \by' + \tau \bn(\bx') \, : \, \by' \cdot \bn(\bx') = 0\,, \tau > 0\,, |\by'|^2 + \tau^2 \leq R_0^2  \} = B(0,R_0) \cap \{ \bz \, : \, \bz \cdot \bn(\bx') > 0 \}$, we can rewrite the $\tau$-$\by'$ integral to get
	\begin{equation*}
		\begin{split}
		&	\lim\limits_{\veps \to 0}  \int_{\Omega_{\sqrt{\veps}}}  \int_{\Omega \setminus \Omega_{\sqrt{\veps}}}  \rho_{\eta_{\delta}(\bx),2}(|\bx-\by|) (\by-\bx) \cdot \grad u(\bx) v(\bx) \, \rmd \by  \, \rmd \bx \\
			&\;= \int_{\p \Omega \cap B} \int_{\Vint{\bn(\bx') }^\perp } \int_{ 0 }^{R_0}  \left\{ \chi_{B(0,R_0)} ( \by' + \tau \bn(\bx')  )    \tau^2 \, \rho \left( |\by' + \tau \bn(\bx')| \right)\right.\\
			&\left.\qquad\qquad\qquad  
 \bn(\bx') \cdot \grad u(\bx') v(\bx') \det \grad \Psi(\bx',0) \,\right\} \, \rmd \tau    \, \rmd \by'  \, \rmd \bx' \\
			&\;= \int_{\p \Omega \cap B} \int_{B(0,R_0) \cap \{ \bz \cdot \bn(\bx') > 0 \} } 
   \left\{ |\bz \cdot \bn(\bx')|^2 \rho \left( |\bz| \right) 
   \bn(\bx') \cdot \grad u(\bx') v(\bx') \det \grad \Psi(\bx',0) \,\right\} \, \rmd \bz \, \rmd \bx\,.
		\end{split}
	\end{equation*}
	Applying the appropriate coordinate rotation in the $\bz$ integral and noting the definition of surface measure from Step 1, we obtain \eqref{eq:GreensIdPf:Step2:SqDist}.

	\textit{Step 3}: Let $\veps \ll \min \{ \veps_0,\delta \}$. We show that
	\begin{equation}\label{eq:GreensIdPf:Step3:SqDist}
		\begin{split}
			\lim\limits_{\veps \to 0}  &\int_{\Omega_{\sqrt{\veps}}} \int_{\Omega \setminus \Omega_{\sqrt{\veps}}} \rho_{\eta_{\delta}(\by),2}(|\bx-\by|) (\by-\bx) \cdot \grad u(\bx) v(\bx) \, \rmd \by  \, \rmd \bx \\
			&= -\int_{B(0,R_0) \cap \{ z_d > 0 \} } z_d^2 \rho \left( |\bz| \right) \, \rmd \bz \cdot \int_{\p \Omega} \frac{\p u}{\p \bsnu}(\bx) v(\bx) \, \rmd \sigma (\bx)\,,
		\end{split}
	\end{equation}
	where $u$ and $v$ are supported on $\p \Omega \cap B$ where $B$ is as in Step 1.
	
	We proceed in a similar way. Since $\veps \ll \delta$, we have that $\eta_{\delta}(\by) = \dist(\by,\p \Omega)^2$ on $\Omega_{\sqrt{b \veps }} \setminus \Omega_{\sqrt{\veps}}$. Now by \Cref{lma:GreensId1:SqDist} we can use linearity of the integrals to write
	\begin{equation*}
		\begin{split}			&\int_{\Omega \setminus \Omega_{\sqrt{\veps}} } \rho_{\eta_{\delta}(\by),2} (|\by-\bx|) (\by-\bx) \, \rmd \by \\
			&=
   \int_{ \Omega_{\sqrt{b\veps}} \setminus \Omega_{\sqrt{\veps}} } \rho\left( \frac{|\by-\bx|}{\eta_{\delta}(\by)} \right) \frac{(\by-\bx)}{\eta_{\delta}(\bx)} \frac{\eta_{\delta}(\by) + \grad \eta_{\delta}(\by) \cdot (\bx-\by) }{\eta_{\delta}(\by)^{d+2}}  \, \rmd \by \\	
			&\quad +
   \int_{ \Omega_{\sqrt{b \veps} } \setminus \Omega_{\sqrt{\veps}} } \rho \left( \frac{|\by-\bx|}{\eta_{\delta}(\by)} \right) \frac{(\by-\bx)}{\eta_{\delta}(\bx)} \frac{\eta_{\delta}(\bx) -  \eta_{\delta}(\by) - \grad \eta_{\delta}(\by) \cdot (\bx-\by) }{\eta_{\delta}(\by)^{d+2}}  \, \rmd \by\,.
		\end{split}
	\end{equation*}
	Taking the inner product with $v \grad u$ and integrating with respect to $\bx \in \Omega_{\sqrt{\veps}}$, we see that by an argument identical to the proof of \Cref{lma:OpTaylorExp:SqDist} the second integral is majorized by 
	$$
	C(d,\rho,\Omega) \Vnorm{\grad u}_{L^{\infty}(\Omega)} \Vnorm{v}_{L^{\infty}(\Omega)} |\Omega_{\sqrt{\veps}}|\,,
	$$
	and thus tends to zero as $\veps \to 0$. Therefore
\begin{equation}\label{eq:GreensIdPf:Equivalence1:SqDist}
		\begin{split}
			&\lim\limits_{\veps \to 0}  \int_{\Omega_{\sqrt{\veps}}} \int_{\Omega \setminus \Omega_{\sqrt{\veps}}} \rho_{\eta_{\delta}(\by),2}(|\bx-\by|) (\by-\bx) \cdot \grad u(\bx) v(\bx) \, \rmd \by  \, \rmd \bx \\
			&= \lim\limits_{\veps \to 0} \int_{\Omega_{\sqrt{\veps}}}
\int_{\Omega_{\sqrt{b\veps}} \setminus \Omega_{\sqrt{\veps}} } \rho \left( \frac{|\by-\bx|}{\eta_{\delta}(\by)} \right) \frac{(\by-\bx)}{\eta_{\delta}(\bx)} \cdot \grad u(\bx) v(\bx) \frac{\eta_{\delta}(\by) + \grad \eta_{\delta}(\by) \cdot (\bx-\by) }{\eta_{\delta}(\by)^{d+2}}  \, \rmd \by \, \rmd \bx\,.
		\end{split}
	\end{equation}
	We recall the definition $\bg_{\bx}(\by) = \frac{\by-\bx}{\eta_{\delta}(\by)}$, and that $\det \grad \bg_{\bx}(\by) = \frac{\eta_{\delta}(\by) + \grad \eta_{\delta}(\by) \cdot (\bx-\by) }{\eta_{\delta}(\by)^{d+1}}$. Thus
	\begin{equation}\label{eq:GreensIdPf:Equivalence2:SqDist}
		\begin{split}
			\int_{\Omega_{\sqrt{\veps}}} &
   \int_{ \Omega_{\sqrt{b\veps}} \setminus \Omega_{\sqrt{\veps}} } \rho \left( \frac{|\by-\bx|}{\eta_{\delta}(\by)} \right) \frac{(\by-\bx)}{\eta_{\delta}(\bx)} \cdot \grad u(\bx) v(\bx) \frac{\eta_{\delta}(\by) + \grad \eta_{\delta}(\by) \cdot (\bx-\by) }{\eta_{\delta}(\by)^{d+2}}  \, \rmd \by \, \rmd \bx \\	
   &=\int_{\Omega_{\sqrt{\veps}} } \int_{\bbR^d} \chi_{ \Omega_{\sqrt{b\veps}} \setminus \Omega_{\sqrt{\veps}} } (\bg_{\bx}^{-1}(\bz) ) \;\; \frac{1}{\eta_{\delta}(\bx)} \rho \left( |\bz| \right) \bz \cdot \grad u(\bx) v(\bx) \, \rmd \bz \, \rmd \bx\,.
		\end{split}
	\end{equation}
	Now we Taylor expand the distance function; since the integration region is $\veps \leq \eta_{\delta}(\bg_{\bx}^{-1}(\bz)) \leq b \veps $ and by definition of $\bg_{\bx}$ we have
	\begin{equation*}
		\begin{split}
			\eta_{\delta}(\bg_{\bx}^{-1}(\bz)) &= \eta_{\delta}(\bx) + \grad \eta_{\delta}(\bx) \cdot (\bg_{\bx}^{-1}(\bz)-\bx) + O( |\bg_{\bx}^{-1}(\bz) - \bx|^2 ) \\
			&= \eta_{\delta}(\bx) + \eta_{\delta}(\bg_{\bx}^{-1}(\bz)) \grad \eta_{\delta}(\bx) \cdot \bz + O( |\bz|^2 \eta_{\delta}(\bg_{\bx}^{-1}(\bz))^2 )\,,
		\end{split}
	\end{equation*}
	which implies
\begin{equation}\label{eq:GreensIdPf:TaylorExpand:YIntegral:SqDist}
		\eta_{\delta}(\bg_{\bx}^{-1}(\bz)) = \frac{\eta_{\delta}(\bx)}{1-\grad \eta_{\delta}(\bx) \cdot \bz } + O(\veps^2)\,.
	\end{equation}
	We perform the two changes of variables just as in Step 2:
	\begin{equation*}
		\bx = \Psi(\bx',t) = \bx' + t \bn(\bx')\,, \qquad \bz = \bz' + \tau \bn(\bx')\,, \bz' \in \Vint{\bn(\bx')}^\perp\,.
	\end{equation*}
	Therefore by \eqref{eq:GreensIdPf:Parametrization:SqDist} and \eqref{eq:GreensIdPf:TaylorExpand:YIntegral:SqDist}, we get
	\begin{equation*}
		\begin{split}  
&\int_{\Omega_{\sqrt{\veps}}} \int_{\bbR^d} \chi_{ \Omega_{\sqrt{b\veps}} \setminus \Omega_{\sqrt{\veps}} } (\bg_{\bx}^{-1}(\bz) ) \;\; \frac{1}{\eta_{\delta}(\bx)} \rho \left( |\bz| \right) \bz \cdot \grad u(\bx) v(\bx) \, \rmd \bz \, \rmd \bx \\
			&\;=	\int_{\Omega \cap B} \int_0^{\sqrt{\veps}} \int_{\Vint{\bn(\bx')}^\perp} \int_{\bbR}  \chi_{ \{ \veps < \frac{t^2}{1-2 \delta t \tau} + O(\veps^2) < b \veps  \} } \left\{ \right. \chi_{B(0,R_0)} (\bz'+\tau \bn(\bx')) \;\; \frac{1}{t^2} \rho \left( |\bz'+\tau \bn(\bx')| \right) \\
			&\left. \qquad \qquad (\bz'+\tau \bn(\bx')) \cdot \grad u(\Psi(\bx',t)) v(\Psi(\bx',t)) \det \grad \Psi(\bx',t)\,\right\} \, \rmd \tau \, \rmd \bz' \, \rmd t \, \rmd \bx\,.
		\end{split}
	\end{equation*}
	From here the proof proceeds identically to Step 2 (albeit with different functions of $t$ and $\tau$ in the integration region), and we eventually come to
	\begin{equation*}
		\begin{split}
			&\lim\limits_{\veps \to 0} \int_{\Omega_{\sqrt{\veps}}} \int_{\bbR^d} \chi_{ \Omega_{\sqrt{b \veps}} \setminus \Omega_{\sqrt{\veps}} } (\bg_{\bx}^{-1}(\bz) ) \;\; \frac{1}{\eta_{\delta}(\bx)} \rho \left( |\bz| \right) \bz \cdot \grad u(\bx) v(\bx) \, \rmd \bz \, \rmd \bx \\
			&\;= \lim\limits_{\veps \to 0} \int_{\p \Omega \cap B} \int_{\Vint{\bn(\bx')}^{\perp}} \int_{O(\veps)}^{R_0} \left\{\frac{\sqrt{\veps}\tau + \sqrt{1+\veps \tau^2} - 1  }{\sqrt{\veps}}  \;\; \rho \left( |\bz' + \tau \bn(\bx')| \right) \right.\\
			&\left. \qquad\qquad 
   ( \bz' + \tau \bn(\bx') ) \cdot \grad u(\bx') v(\bx') \,\right\}\, \rmd \tau \, \rmd \bz' \, \rmd \bx + O(\sqrt{\veps})\\
			&\;= -\int_{B(0,R_0) \cap \{ z_d > 0 \} } z_d^2 \rho \left( |\bz| \right) \, \rmd \bz \cdot \int_{\p \Omega} \frac{\p u}{\p \bsnu}(\bx) v(\bx) \, \rmd \sigma (\bx)\,. 
		\end{split}
	\end{equation*}
	Combining this with \eqref{eq:GreensIdPf:Equivalence2:SqDist} and \eqref{eq:GreensIdPf:Equivalence1:SqDist} completes Step 3, 
	and combining \eqref{eq:GreensIdPf:Step2:SqDist}-\eqref{eq:GreensIdPf:Step3:SqDist} gives
	\begin{equation*}
		\begin{split}
			\lim\limits_{\veps \to 0}  & \int_{\Omega_{\sqrt{\veps}}} \int_{\Omega \setminus \Omega_{\sqrt{\veps}}} 2 \, \rho_{\delta,2}(\bx,\by) (\bx-\by) \cdot \grad u(\bx) v(\bx) \, \rmd \by  \, \rmd \bx \\
			&= 2 \int_{B(0,R_0) \cap \{ z_d > 0 \} } z_d^2 \rho \left( |\bz| \right) \, \rmd \bz \cdot \int_{\p \Omega} \frac{\p u}{\p \bsnu}(\bx) v(\bx) \, \rmd \sigma (\bx)\,.
		\end{split}
	\end{equation*}
	By \eqref{Assump:KernelNormalization} the first integral is equal to $1$, and so \eqref{eq:GreensIdPf:Limit2:SqDist} is proved.
\end{proof}

\subsection{Interpreting the action of the operator as a distribution}\label{sec:OpDistribution}
With the help of the nonlocal Green's identity, we can interpret the action of the nonlocal operator on Sobolev spaces and the nonlocal energy space in the sense of distributions, as demonstrated in the following theorems. Such interpretations are very useful in the proofs of the well-posedness of nonlocal boundary value problems.

\begin{theorem}\label{thm:OperatorAsDist:HMinus1}
	Let $\Omega \subset \bbR^d$ satisfy \eqref{assump:Domain}  and $\rho$ satisfy  \eqref{Assump:KernelSmoothness}. Let $u \in H^1(\Omega)$. 
	Define the distribution $\cL_{\delta} u \in \cD'(\Omega)$ by
\begin{equation}\label{eq:OperatorAsDist:HMinus1:1}
		\Vint{ \cL_{\delta}u, v } := \lim\limits_{n \to \infty} \int_{\Omega} \cL_{\delta} u_n(\bx) v(\bx) \, \rmd \bx\,, \qquad v \in C^{\infty}_c(\Omega)\,,
	\end{equation}
	where $\{ u_n \}$ is a sequence of functions in $C^2(\overline{\Omega})$ that converges to $u$ in $H^1(\Omega)$.
	Then the definition of $\cL_{\delta} u$ is independent of the sequence chosen, and the action of $\cL_{\delta} u$ can be written explicitly as	\begin{equation}\label{eq:OperatorAsDist:HMinus1:2}
	    \Vint{\cL_{\delta}u , v} = B_{\rho,\delta} (u, v) \,, \qquad v \in C^{\infty}_c(\Omega)\,.
	\end{equation}
	Furthermore, $\cL_{\delta} u \in H^{-1}(\Omega)$, there exists $C = C(\rho,\Omega)$ such that 
	\begin{equation*}
		\Vnorm{\cL_{\delta}u }_{H^{-1}(\Omega)} \leq C \Vnorm{u}_{H^1(\Omega)}\,,
	\end{equation*} 
	and the action of $\cL_{\delta} u$ can be written as \eqref{eq:OperatorAsDist:HMinus1:2} for any $v \in H^1_0(\Omega)$.
\end{theorem}

\begin{proof}
	For each $n$ the integral defining $\Vint{ \cL_{\delta} u_n, v}$ is absolutely convergent thanks to \Cref{thm:OpIsDefined:Inhomog:SqDist}, and from the nonlocal Green's identity for smooth functions established in \Cref{thm:GreensIdentity:SqDist}, \eqref{eq:BilinearForm:Continuity}, and \Cref{thm:Embedding} we obtain
	\begin{equation*}
	\begin{split}
		\left| \int_{\Omega} \cL_{\delta} u_n(\bx) v(\bx) \, \rmd \bx \right| 
		&= \left| B_{\rho,\delta} (u, v) \right| \\
		&\leq C(\rho,\Omega) \Vnorm{u_n}_{H^1(\Omega)} \Vnorm{v}_{H^1(\Omega)}\,.
	\end{split}
	\end{equation*}
	Each conclusion in the theorem then follows from this estimate.
\end{proof}

A similar result holds when considering the action on functions in the nonlocal energy space.

\begin{theorem}\label{thm:OperatorAsDist:NonlocalSpace}
    Let $\Omega \subset \bbR^d$ satisfy \eqref{assump:Domain} and $\rho$ satisfy \eqref{Assump:KernelSmoothness}. Let $u \in \frak{W}^{\delta,2}(\Omega)$, and
	define the distribution $\cL_{\delta} u \in \cD'(\Omega)$ by \eqref{eq:OperatorAsDist:HMinus1:1},
	where $\{ u_n \}$ is a sequence of functions in $C^{\infty}(\overline{\Omega})$ that converges to $u$ in $\frak{W}^{\delta,2}(\Omega)$.
	Then the definition of $\cL_{\delta} u$ is independent of the sequence chosen, and the action of $\cL_{\delta} u$ can be written explicitly as \eqref{eq:OperatorAsDist:HMinus1:2}.
	Furthermore, $\cL_{\delta} u \in [\frak{W}^{\delta,2}_0(\Omega)]^*$, there exists $C = C(\rho,\Omega)$ such that 
	\begin{equation}\label{eq:OperatorAsDist:NonlocalSpace}
		\Vnorm{\cL_{\delta}u }_{[\frak{W}^{\delta,2}_0(\Omega)]^*} \leq C \Vnorm{u}_{\frak{W}^{\delta,2}(\Omega)}\,,
	\end{equation}
	and the action of $\cL_{\delta} u$ can be written as \eqref{eq:OperatorAsDist:HMinus1:2} for any $v \in \frak{W}^{\delta,2}_0(\Omega)$.
\end{theorem}

\begin{proof}
	For each $n$ the integral defining $\Vint{ \cL_{\delta} u_n, v}$ is absolutely convergent thanks to \Cref{thm:OpIsDefined:Inhomog:SqDist}, 
	and from the nonlocal Green's identity for smooth functions established in \Cref{thm:GreensIdentity:SqDist} and  \eqref{eq:BilinearForm:Continuity} we obtain the estimate
	\begin{equation*}
	\begin{split}
		\left| \int_{\Omega} \cL_{\delta} u_n(\bx) v(\bx) \, \rmd \bx \right| 
		= \left| B_{\rho,\delta} (u, v)  \right|
		\leq C(\rho,\Omega) \Vnorm{u_n}_{\frak{W}^{\delta,2}(\Omega)} \Vnorm{v}_{\frak{W}^{\delta,2}(\Omega)}\,.
	\end{split}
	\end{equation*}
	Each conclusion in the theorem then follows from this estimate.
\end{proof}

\section{The nonlocal function space: trace theorem and nonlocal Poincar\'{e} inequalities}\label{sec:EnergySpace}

In this section, we present a few important properties on the nonlocal energy spaces such as the trace theorems and Poincare inequalities. The nonlocal Green's identity in more general function spaces is also established. 
All of these results are important ingredients in the later proofs of the well-posedness of nonlocal problems with local boundary conditions.

\subsection{The trace operator}\label{sec:Trace}

Intuitively, we expect that the choice of the heterogenenous localization function $\eta_\delta$
picked in the current work leads to a stronger localization effect on the boundary than the case of 
 localization function  being linearly proportional to $\dist(\bx, \p \Omega)$. Since the nonlocal energy space associated with the latter choice enjoys the same trace inequality 
 as $W^{1,p}$ \cite{tian2017trace,du2022fractional,Foss2021}, we expect that the new nonlocal energy space associated with $\eta_\delta$ would share a similar trace inequality. While this can indeed be shown rigorously, we only briefly present the main results 
 but will leave full details (in more general forms) to separate works.

\begin{theorem}\label{thm:SeminormEquivalence}
    Let $\Omega \subset \bbR^d$ satisfy \eqref{assump:Domain}, and let $R_0 \in (0,1)$. For all $\delta < \delta_0$, there exists a constant $C$ depending on $\Omega$, $R_0$, and $\delta$ such that
    \begin{equation*}
        \int_{\Omega} \int_{ \{|\by-\bx| \leq R_0 \dist(\bx, \p \Omega) \} } \frac{|u(\bx)-u(\by)|^2}{\dist(\bx, \p \Omega)^{d+2}} \, \rmd \by \, \rmd \bx \leq C \left( \Vnorm{u}_{L^2(\Omega)}^2 + [u]_{\frak{W}^{\delta,2}(\Omega)}^2 \right)\,.
    \end{equation*}
\end{theorem}

It is a consequence of the above theorem that the properties of the nonlocal function spaces studied in \cite{tian2017trace,du2022fractional} are inherited by the space $\frak{W}^{\delta,2}(\Omega)$.
For instance, we have a trace theorem.

\begin{theorem}\label{thm:TraceTheorem}
	Let $\Omega \subset \bbR^d$ satisfy \eqref{assump:Domain}. Let $T$ denote the trace operator, i.e. for $u \in C^{\infty}(\overline{\Omega})$
	\begin{equation*}
	    T u = u \big|_{\p \Omega}\,.
	\end{equation*}
	Then for each $\delta < \delta_0$ the trace operator extends to a bounded linear operator $T : \frak{W}^{\delta,2}(\Omega) \to H^{\frac{1}{2}}(\p \Omega)$,
	and there exists $C = C(\Omega,R_0,\delta)$ such that
	$$
	\Vnorm{Tu}_{H^{\frac{1}{2}}(\p \Omega)} \leq C \Vnorm{u}_{\frak{W}^{\delta,2}(\Omega)} ,\qquad\forall u \in \frak{W}^{\delta,2}(\Omega)\,.
	$$
\end{theorem}

\begin{proof}
	Follows from \cite[Theorem 1.3]{tian2017trace} and \Cref{thm:SeminormEquivalence}.
\end{proof}

\begin{corollary}\label{cor:TraceZero}
    Let $\Omega \subset \bbR^d$ satisfy \eqref{assump:Domain}. Suppose $u \in \frak{W}^{\delta,2}_0(\Omega)$. Then $T u = 0$\,.
\end{corollary}

The trace theorems ensure that proper local boundary conditions can be imposed for the associated nonlocal problems.

\subsection{The nonlocal Green's identity for wider classes of functions}\label{sec:GreensWider}

\begin{theorem}\label{thm:NormalDerivativeFxnal}
    Let $\Omega \subset \bbR^d$ satisfy \eqref{assump:Domain} and $\rho$ satisfy \eqref{Assump:KernelSmoothness}. Let $u \in \frak{W}^{\delta,2}(\Omega)$ and suppose additionally that $\cL_{\delta} u \in L^2(\Omega)$.
    Then the operator $\frac{\p }{\p \bsnu}$ defined for functions $v \in C^{\infty}(\overline{\Omega})$ as
    \begin{equation*}
        \frac{\p v}{\p \bsnu} = \grad v(\bx) \cdot \bsnu(\bx)\,, \qquad \bx \in \p \Omega\,,
    \end{equation*}
    extends to an action on the function $u$. To be precise, $\frac{\p }{\p \bsnu}$ maps $u$ to 
    $\frac{\p u}{\p \bsnu} \in H^{-\frac{1}{2}}(\p \Omega)$,
    with
    \begin{equation*}
        \Vnorm{\frac{\p u}{\p \bsnu}}_{H^{-\frac{1}{2}}(\Omega)} \leq C(\Omega,\rho) \left( [u]_{\frak{W}^{\delta,2}(\Omega)} + \Vnorm{ \cL_{\delta} u}_{L^2(\Omega)} \right)\,.
    \end{equation*}
\end{theorem}

\begin{proof}
    Define the Banach space
    \begin{equation*}
        X^{\delta,2}(\Omega) := \{ \text{ closure of } C^{\infty}(\overline{\Omega}) \text{ with respect to the norm } \Vnorm{\cdot}_{X^{\delta,2}(\Omega)} \}\,,
    \end{equation*}
    where the norm is defined as
    \begin{equation*}
        \Vnorm{u}_{X^{\delta,2}(\Omega)}^2 := \Vnorm{u}_{\frak{W}^{\delta,2}(\Omega)}^2 + \Vnorm{ \cL_{\delta} u }_{L^2(\Omega)}^2\,.
    \end{equation*}
    It is clear that $X^{\delta,2}(\Omega)$ is a closed subspace of $\frak{W}^{\delta,2}(\Omega)$.
    Let $\{ u_n\}$ be a sequence in $C^{\infty}(\overline{\Omega})$ that converges to $u$ in $X^{\delta,2}(\Omega)$.
    Then by the nonlocal Green's identity \eqref{eq:GreensIdentity:Intro} already established in \Cref{thm:GreensIdentity:SqDist} for smooth functions we have for any $v \in C^{\infty}(\overline{\Omega})$
    \begin{equation*}
    \begin{split}
    \int_{\p \Omega} \frac{\p u_n}{\p \bsnu}(\bx) \cdot v(\bx) \, \rmd \sigma(\bx) &=
        B_{\rho,\delta}(u_n, v)
        - \int_{\Omega} \cL_{\delta} u_n(\bx) \cdot v(\bx) \, \rmd \bx\,.
    \end{split}
	\end{equation*}
    By applying H\"older's inequality to the right-hand side, we get
    \begin{equation}\label{eq:NormalDerivFxnal:Pf1}
    \left|
     \int_{\p \Omega} \frac{\p u_n}{\p \bsnu}(\bx) \cdot v(\bx) \, \rmd \sigma(\bx)
    \right|\leq
        C(\Omega,\rho) \left( [u_n]_{\frak{W}^{\delta,2}(\Omega)} + \Vnorm{ \cL_{\delta} u_n}_{L^2(\Omega)} \right) \Vnorm{v}_{\frak{W}^{\delta,2}(\Omega)}\,.
    \end{equation}
    Now for $\bar{v} \in H^{\frac{1}{2}}(\p \Omega)$, let $v \in H^1(\Omega)$ be any extension of $\bar{v}$ to $\Omega$ with $\Vnorm{v}_{H^1(\Omega)} \leq C \Vnorm{\bar{v}}_{H^{\frac{1}{2}}(\p \Omega)}$, where $c$ is independent of $\bar{v}$ and $v$. Let $\{ v_j \}$ be a sequence in $C^{\infty}(\overline{\Omega})$ converging to $v$ in $H^1(\Omega)$. Then for any $n \in \bbN$ we can use the dominated convergence theorem and the classical trace theorem for $H^1(\Omega)$ to conclude that
    \begin{equation*}
    \begin{split}
        \int_{\p \Omega} \frac{\p u_n}{\p \bsnu}(\bx) \cdot \bar{v}(\bx) \, \rmd \sigma(\bx)
        &= \lim\limits_{j \to \infty} \int_{\p \Omega} \frac{\p u_n}{\p \bsnu}(\bx) \cdot T v_j(\bx) \, \rmd \sigma(\bx) \\
        &= \lim\limits_{j \to \infty}
        B_{\rho,\delta}(u_n, v_j)
            - \lim\limits_{j \to \infty}  \int_{\Omega} \cL_{\delta} u_n(\bx) \cdot v_j(\bx) \, \rmd \bx \\
        &=
      B_{\rho,\delta}(u_n, v)
        -  \int_{\Omega} \cL_{\delta} u_n(\bx) \cdot v(\bx) \, \rmd \bx\,.
    \end{split}
    \end{equation*}
    By \eqref{eq:NormalDerivFxnal:Pf1} and by \Cref{thm:Embedding}
    \begin{equation}\label{eq:NormalDerivFxnal:Pf2}
    \begin{split}
        \left| \int_{\p \Omega} \frac{\p u_n}{\p \bsnu}(\bx) \cdot \bar{v}(\bx) \, \rmd \sigma(\bx) \right| 
        &\leq C \left( [u_n]_{\frak{W}^{\delta,2}(\Omega)} + \Vnorm{ \cL_{\delta} u_n}_{L^2(\Omega)} \right) \Vnorm{ v }_{\frak{W}^{\delta,2}(\Omega)} \\
        &\leq C(\Omega,\rho) \left( [u]_{\frak{W}^{\delta,2}(\Omega)} + \Vnorm{ \cL_{\delta} u}_{L^2(\Omega)} \right) \Vnorm{ \bar{v} }_{H^{ \frac{1}{2} }(\Omega)}\,,
    \end{split}
    \end{equation}
    and so $\frac{\p u_n}{\p \bsnu}$ defines an object in $H^{-\frac{1}{2}}(\p \Omega)$.

    Finally we define the functional $\frac{\p u}{\p \bsnu}$ by taking $n \to \infty$:
    \begin{equation*}
        \frac{\p u}{\p \bsnu} := \lim\limits_{n \to \infty} \frac{\p u_n}{\p \bsnu}\,,
    \end{equation*}
    where the limit is taken in the $H^{-\frac{1}{2}}(\Omega)$-norm. The estimate \eqref{eq:NormalDerivFxnal:Pf2} shows that the definition of $\frac{\p u}{\p \bsnu}$ is independent of the approximating sequence $\{ u_n \} \subset X^{\delta,2}(\Omega)$ chosen, and that
    \begin{equation*}
        \Vnorm{\frac{\p u}{\p \bsnu}}_{H^{-\frac{1}{2}}(\Omega)} \leq C(\Omega,\rho) \left( [u]_{\frak{W}^{\delta,2}(\Omega)} + \Vnorm{ \cL_{\delta} u}_{L^2(\Omega)} \right)\,.
    \end{equation*}
\end{proof}

\begin{proposition}\label{prop:GreensId:GeneralFxns}
    Let $\Omega \subset \bbR^d$ satisfy \eqref{assump:Domain}  and $\rho$ satisfy \eqref{Assump:KernelSmoothness}. Then the nonlocal Green's identity 
 \eqref{eq:GreensIdentity:Intro}
    holds in the following cases:
    \begin{enumerate}
        \item[i)] $u \in \frak{W}^{\delta,2}(\Omega)$, $v \in C^{\infty}_c(\Omega)$,
        \item[ii)] $u \in \frak{W}^{\delta,2}(\Omega)$ with $\cL_{\delta} u \in L^2(\Omega)$, $v \in H^{1}(\Omega)$.
    \end{enumerate}
\end{proposition}

\begin{proof}
    We begin with case i). 
    Let $\{ u_n \}$ be a sequence in $C^{\infty}(\overline{\Omega})$ converging to $u$ in the $\frak{W}^{\delta,2}(\Omega)$ norm. Then for each $n$ we can apply the nonlocal Green's identity \eqref{eq:GreensIdentity:Intro} proved for smooth functions in \Cref{thm:GreensIdentity:SqDist}
    to get
    \begin{equation*}
    \begin{split}
        &
        B_{\rho,\delta}(u_n, v)
        = \int_{\Omega} \cL_{\delta} u_n(\bx) \cdot v(\bx) \, \rmd \bx\,,
    \end{split}
	\end{equation*}
    since $v$ vanishes near $\p \Omega$. Now, an application of H\"older's inequality shows that
    \begin{equation*}
    \begin{split}
        \lim\limits_{n \to \infty} & 
   B_{\rho,\delta}(u_n, v) = B_{\rho,\delta}(u, v)\,.
    \end{split}
	\end{equation*}
	To complete the proof in case i) we need to show that
	\begin{equation}\label{eq:GreensId:GeneralFxns:Pf1}
	     \lim\limits_{n \to \infty}  \int_{\Omega} \cL_{\delta} u_n(\bx) \cdot v(\bx) \, \rmd \bx = \int_{\Omega} \cL_{\delta} u(\bx) \cdot v(\bx) \, \rmd \bx \,.
	\end{equation}
	By \eqref{eq:PhiBounds} and since $\dist(\supp v, \p \Omega) > 0$
	\begin{equation*}
	    \begin{split}
	    |(\cL_{\delta} u_n(\bx) - \cL_{\delta} u(\bx)) v(\bx)| 
	    &\leq \frac{|\cL_{\delta} u_n(\bx) - \cL_{\delta} u(\bx)|}{\sqrt{\Phi_{\delta,2}(\bx)}} \cdot \sqrt{\Phi_{\delta,2}(\bx)} |v(\bx)| \\
	    &\leq \frac{|\cL_{\delta} u_n(\bx) - \cL_{\delta} u(\bx)|}{\sqrt{\Phi_{\delta,2}(\bx)}} \cdot \frac{\sqrt{\kappa_{2,u}}}{ \eta_{\delta}(\bx) } |v(\bx)| \\
	    &\leq \frac{C(\Omega,\rho)}{ \dist(\supp v, \p \Omega) } \frac{|\cL_{\delta} u_n(\bx) - \cL_{\delta} u(\bx)|}{\sqrt{\Phi_{\delta,2}(\bx)}} \cdot |v(\bx)|\,.
	    \end{split}
	\end{equation*}
	Therefore \eqref{eq:GreensId:GeneralFxns:Pf1} follows from H\"older's inequality and \Cref{thm:NonlocalOpWellDefd:Energy}.
	
	For case ii), recall the definition of $X^{\delta,2}(\Omega)$ from \Cref{thm:NormalDerivativeFxnal}. let $\{ u_n \}$ and $\{ v_m \}$ be sequences in $C^{\infty}(\overline{\Omega})$ with $\{u_n\}$ converging to $u$ in the $X^{\delta,2}(\Omega)$ norm and $\{v_m\}$ converging to $v$ in the $H^{1}(\Omega)$ norm. Then for each $n$ and $m$ we can apply the nonlocal Green's identity \eqref{eq:GreensIdentity:Intro} proved in \Cref{thm:GreensIdentity:SqDist} to get
    \begin{equation*}
    \begin{split}
        & B_{\rho,\delta}(u_n, v_m)
        = \int_{\Omega} \cL_{\delta} u_n(\bx) \cdot v_m(\bx) \, \rmd \bx + \int_{\p \Omega} \frac{\p u_n}{\p \bsnu}(\bx) \cdot v_m(\bx) \, \rmd \sigma(\bx)\,.
    \end{split}
	\end{equation*}

    An application of H\"older's inequality and the convergence properties of $\{u_n \}$ and $\{ v_m \}$ shows that
    \begin{equation*}
    \begin{split}
        \lim\limits_{n \to \infty} \lim\limits_{m \to \infty}  &
    B_{\rho,\delta}(u_n, v_m) = B_{\rho,\delta}(u, v)
    \end{split}
    \end{equation*}
    and
    \begin{equation*}
    \begin{split}
        \lim\limits_{n \to \infty} \lim\limits_{m \to \infty}  & \int_{\Omega} \cL_{\delta} u_n(\bx) v_m(\bx) \, \rmd \bx = \int_{\Omega} \cL_{\delta} u(\bx) v(\bx) \, \rmd \bx\,.
    \end{split}
    \end{equation*}
	To complete the proof in case ii) we need to show that
    \begin{equation}\label{eq:GreensId:GeneralFxns:Pf2}
	     \lim\limits_{n \to \infty} \lim\limits_{m \to \infty} \int_{\p \Omega} \frac{\p u_n}{\p \bsnu}(\bx) \cdot v_m(\bx) \, \rmd \sigma(\bx) = \int_{\p \Omega} \frac{\p u}{\p \bsnu}(\bx) \cdot v(\bx) \, \rmd \sigma(\bx) \,.
	\end{equation}
    But this follows thanks to the classical trace theorem for $H^1(\Omega)$ and \Cref{thm:NormalDerivativeFxnal}.
\end{proof}

\begin{remark}
    In the classical case, the additional assumption $\Delta u \in L^2(\Omega)$ is nec $\frac{\p u}{\p \bsnu} \in H^{-1/2}(\p \Omega)$ for functions 
\end{remark}
    
\subsection{Poincar\'e inequalities}\label{sec:Poincare}

We define the Hilbert space
\begin{equation*}
	\begin{split}
		\frak{W}^{\delta,2}_0(\Omega) := \{ \text{ closure of } C^{\infty}_c(\Omega) \text{ with respect to the norm } \Vnorm{\cdot}_{\frak{W}^{\delta,2}(\Omega)} \}\,.
	\end{split}
\end{equation*}

The following nonlocal Poincar'{e} inequalities, designed for either Dirichlet or Neumann problems respectively, are proved in \Cref{sec:PoincareProofs}.

\begin{theorem}\label{thm:PoincareDirichlet}
	Let $\Omega \subset \bbR^d$ satisfy \eqref{assump:Domain} and $\rho$  satisfy \eqref{Assump:KernelSmoothness}. Then there exists a constant $C_D=C_D(d,\Omega,\rho) > 0$ such that for all $\delta < \delta_0$ and 
 $u \in \frak{W}^{\delta,2}_0(\Omega)$,
	\begin{equation*}
		\Vnorm{u}_{L^2(\Omega)} \leq C_D [u]_{\frak{W}^{\delta,2}(\Omega)}\,.
	\end{equation*}
\end{theorem}

\begin{theorem}\label{thm:PoincareNeumann}
	Let $\Omega \subset \bbR^d$ satisfy \eqref{assump:Domain} and $\rho$ satisfy \eqref{Assump:KernelSmoothness}, and recall the definition of $\Phi_{\delta} = \Phi_{\delta,0}$ in \eqref{eq:ConvolutionOperatorAlpha}. Then there exists a constant $C_N(d,\Omega,\rho) > 0$ such that for all $\delta < \delta_0$ 
 and 
 $u \in \frak{W}^{\delta,2}(\Omega)$,
	\begin{equation*}
		\Vnorm{u - (\Phi_{\delta} u)_{\Omega} }_{L^2(\Omega)} \leq C_N [u]_{\frak{W}^{\delta,2}(\Omega)}\,.
	\end{equation*}
\end{theorem}

\section{Boundary-value problems: well-posedness results}\label{sec:BVPs}

Equipped with the nonlocal Green's identity, we can now state weak versions of boundary-value problems associated to $\cL_{\delta}$ with different boundary conditions.
In the case of Dirichlet data, the well-posedness of the homogeneous problem is established, and then the general inhomogeneous problem is treated by using auxiliary functions to reduce to the homogeneous case.
In the case of Neumann data, we can treat the inhomogeneous problem directly.

\subsection{The Dirichlet problem: homogeneous boundary conditions}\label{sec:HomoDiri}

The first boundary-value problem we treat is a nonlocal Poisson problem with homogeneous Dirichlet boundary conditions
as stated in \eqref{eq:Intro:NonlocalEq}-\eqref{eq:Intro:HomogDirichletBC}.

\begin{definition}
	Let $\Omega \subset \bbR^d$ satisfy \eqref{assump:Domain} and $\rho$ satisfy  \eqref{Assump:KernelSmoothness}. For $f \in [\frak{W}^{\delta,2}_0(\Omega)]^*$, we say that $u \in \frak{W}^{\delta,2}_0(\Omega)$ is a \textit{weak solution} to the nonlocal problem \eqref{eq:Intro:NonlocalEq} with homogeneous Dirichlet data \eqref{eq:Intro:HomogDirichletBC} if	\begin{equation}\label{eq:NonlocalProblem:WeakForm}
		B_{\rho,\delta}(u,v) = \Vint{f,v}\, , \quad\forall v \in \frak{W}^{\delta,2}_0(\Omega)\,.
	\end{equation}
\end{definition}

\begin{theorem}[Well-posedness]\label{thm:WellPosedness}
	Let $\Omega \subset \bbR^d$ satisfy \eqref{assump:Domain}  and $\rho$ satisfy  \eqref{Assump:KernelSmoothness}. Let $\delta < \delta_0$. Then there exists a constant $C = C(d,\rho,\Omega,C_D)$ such that
\begin{equation}\label{eq:BilinearForm:Coercivity}
		\Vnorm{u}_{\frak{W}^{\delta,2}(\Omega)}^2 \leq  C  B_{\rho,\delta}(u,u) ,\qquad\forall u \in \mathfrak{W}^{\delta,2}_0(\Omega)\,.
	\end{equation}
	Moreover, for any $f \in [\frak{W}^{\delta,2}_0(\Omega)]^*$ there exists a unique solution $u \in \frak{W}^{\delta,2}_0(\Omega)$ to \eqref{eq:NonlocalProblem:WeakForm} satisfying the energy estimate
	\begin{equation}\label{eq:HomogDirichlet:EnergyEstimate}
		\Vnorm{u}_{\frak{W}^{\delta,2}(\Omega)} \leq C(d,\rho,\Omega) \Vnorm{f}_{[\frak{W}^{\delta,2}_0(\Omega)]^*}\,.
	\end{equation}
\end{theorem}

\begin{proof}
    The coercivity \eqref{eq:BilinearForm:Coercivity} follows from \eqref{eq:KernelEquivalence} and \Cref{thm:PoincareDirichlet}.
    We conclude the existence and uniqueness of a weak solution via the Lax-Milgram theorem thanks to the continuity of $B_{\rho,\delta}$ established in \eqref{eq:BilinearForm:Continuity}.
\end{proof}

\subsection{The Dirichlet problem: inhomogeneous boundary conditions}\label{sec:InHomoDiri} 
Consider the problem \eqref{eq:Intro:NonlocalEq} with inhomogenenous Dirichlet boundary data \eqref{eq:Dirichlet:Inhomog:NonlocalBC}.
This problem can be reduced to the case of homogeneous Dirichlet boundary conditions by the following argument.   Let $G \in H^1(\Omega)$ be an extension of $g$ to $\Omega$ with $\Vnorm{G}_{H^1(\Omega)} \leq C \Vnorm{g }_{H^{\frac{1}{2}}(\p \Omega) }$.
Then can we define the solution $u$ of \eqref{eq:Intro:NonlocalEq}-\eqref{eq:Dirichlet:Inhomog:NonlocalBC} as
\begin{equation}\label{eq:Dirichlet:Inhomog:Decomp}
    u(\bx) := w(\bx) + G(\bx)\,,
\end{equation}
where $w(\bx)$ is the unique weak solution of
\begin{equation}\label{eq:Dirichlet:Inhomog:NonlocalEqn:Aux}
	\begin{cases}
		\cL_{\delta} w = f - \cL_{\delta}G &\text{ in } \Omega\,, \\
		w = 0 &\text{ on } \p \Omega\,.
	\end{cases}
\end{equation}

\begin{theorem}[Well-posedness]\label{thm:WellPosedness:Dirichlet:Inhomog}
    Let $\Omega \subset \bbR^d$ satisfy \eqref{assump:Domain} and $\rho$ satisfy \eqref{Assump:KernelSmoothness}.  Let $\delta < \delta_0$. Let $f \in [\frak{W}^{\delta,2}_0(\Omega)]^*$ and $g \in H^{\frac{1}{2}}(\p \Omega)$ be given.
    Then there exists a unique $u \in \frak{W}^{\delta,2}(\Omega)$ satisfying
  the inhomogeneous Dirichlet boundary conditions
\eqref{eq:Dirichlet:Inhomog:NonlocalBC} in the trace sense
    and the weak form \eqref{eq:NonlocalProblem:WeakForm}.
    Moreover,
\begin{equation}\label{eq:InhomogDirichlet:EnergyEstimate}
\Vnorm{u}_{\frak{W}^{\delta,2}(\Omega)} \leq C(d,\rho,\Omega)  \big( \Vnorm{f}_{[\frak{W}^{\delta,2}_0(\Omega)]^*} + \Vnorm{g}_{H^{\frac{1}{2}}(\p \Omega)} \big)\,.
	\end{equation}
\end{theorem}

\begin{proof} Consider the equivalent weak form of the problem \eqref{eq:NonlocalProblem:WeakForm}
with inhomogeneous Dirichlet data:
\begin{equation*}
	B_{\rho,\delta}(w,v) = \Vint{f - \cL_{\delta} G,v}\,, \quad \forall v \in \frak{W}^{\delta,2}_0(\Omega)\,.
 	\end{equation*}
By \Cref{thm:OperatorAsDist:NonlocalSpace}, $\cL_{\delta}G \in [\frak{W}^{\delta,2}_0(\Omega)]^*$, and so existence and uniqueness of $w \in\frak{W}^{\delta,2}_0(\Omega)$ follows from \Cref{thm:WellPosedness}.

Since $u = w + G \in \frak{W}^{\delta,2}(\Omega)$, we have by \Cref{thm:OperatorAsDist:NonlocalSpace}
\begin{equation*}
    \begin{split}
    B_{\rho,\delta}(u,v) &= B_{\rho,\delta}(w,v) + B_{\rho,\delta}(G,v) = \Vint{f,v} - \Vint{\cL_{\delta}G,v} + B_{\rho,\delta}(G,v) \\
    &= \Vint{f,v} - B_{\rho,\delta}(G,v)  + B_{\rho,\delta}(G,v) = \Vint{f,v}
    \end{split}
\end{equation*}
for any $v \in \frak{W}^{\delta,2}_0(\Omega)$, which is 
\eqref{eq:NonlocalProblem:WeakForm}.

Moreover, by the assumption on $G$ and \Cref{thm:Embedding} we have
\begin{equation}\label{eq:EstimateG}
    \begin{split}
\Vnorm{G}_{\frak{W}^{\delta,2}(\Omega)} 
        \leq  C\Vnorm{G}_{H^{1}(\Omega)} \leq C\Vnorm{g}_{H^{\frac{1}{2}}(\p \Omega)}\,.
    \end{split}
\end{equation}
By  \eqref{eq:HomogDirichlet:EnergyEstimate}, \eqref{eq:OperatorAsDist:NonlocalSpace}, and \eqref{eq:EstimateG}
\begin{equation*}
    \begin{split}
        \Vnorm{w}_{\frak{W}^{\delta,2}(\Omega)} 
        &\leq C \big( \Vnorm{f}_{[\frak{W}^{\delta,2}_0(\Omega)]^*} + \Vnorm{\cL_{\delta}G}_{[\frak{W}^{\delta,2}_0(\Omega)]^*} \big) \\
        &\leq C \big( \Vnorm{f}_{[\frak{W}^{\delta,2}_0(\Omega)]^*} + \Vnorm{G}_{\frak{W}^{\delta,2}(\Omega)} \big)\\
        &\leq C \big( \Vnorm{f}_{[\frak{W}^{\delta,2}_0(\Omega)]^*} + \Vnorm{g}_{H^{\frac{1}{2}}(\p \Omega)} \big)\,.
    \end{split}
\end{equation*}
These estimates on $G$ and $w$ lead to 
\eqref{eq:InhomogDirichlet:EnergyEstimate}. The equivalence of traces $u = g$ on $\p \Omega$ is established by \eqref{eq:InhomogDirichlet:EnergyEstimate} and \Cref{thm:TraceTheorem}. 
Obviously, the solution $u$ defined via \eqref{eq:Dirichlet:Inhomog:Decomp} is
independent of the extension $G$ chosen.
\end{proof}

\subsection{The Neumann problem}\label{subsec:InHomoNeu}

The nonlocal Poisson problem with inhomogeneous Neumann boundary conditions is given by \eqref{eq:Intro:NonlocalEq}-\eqref{eq:InHomogNeumannBC}.
The special case of homogeneous boundary conditions $g = 0$ is covered here as well.
In step with the treatment for the classical Neumann problem, we see that solutions of \eqref{eq:Intro:NonlocalEq}-\eqref{eq:InHomogNeumannBC} are unique up to constants, and an application of the nonlocal Green's identity  \eqref{eq:GreensIdentity:Intro} shows that the compatibility condition
\begin{equation*}
	\int_{\Omega} f(\bx) \, \rmd \bx + \int_{\p \Omega} g(\bx) \, \rmd \sigma(\bx) = 0
\end{equation*}
is required for existence of a solution.

We introduce a weak formulation of \eqref{eq:Intro:NonlocalEq}-\eqref{eq:InHomogNeumannBC} by the following formal computation: Suppose that $u \in C^2(\overline{\Omega})$, $\frac{\p u}{\p \bsnu} = g$ on $\p \Omega$, and $u$ satisfies $\cL_{\delta} u = f$ for some given function $f$.
Then for arbitrary $v \in C^2(\overline{\Omega})$, the nonlocal Green's identity \eqref{eq:GreensIdentity:Intro} gives
\begin{equation*}
	\begin{split}
        B_{\rho,\delta}(u,v) = 
  \Vint{\cL_{\delta} u, v}
  + \int_{\p \Omega} \frac{\p u}{\p \bsnu}(\bx) v(\bx) \, \rmd \sigma(\bx)
		=  \Vint{f,v}
  + \int_{\p \Omega} g(\bx) v(\bx) \, \rmd \sigma(\bx)\,.
	\end{split}
\end{equation*}

Recall the definition of $\Phi_{\delta} = \Phi_{\delta,0}$ as in \eqref{eq:ConvolutionOperatorAlpha} for $\alpha = 0$.
For a kernel $\rho$ that satisfies \eqref{Assump:KernelSmoothness}, define the function space
\begin{equation*}
	\mathring{\frak{W}}^{\delta,2}_{\rho}(\Omega) := \left\{ u \in \frak{W}^{\delta,2}(\Omega) \, : \, 
= \Vint{\Phi_{\delta},u}
 = 0 
 = (\Phi_{\delta} u)_{\Omega}
 \right\}\,.
\end{equation*}
It is clear that $\mathring{\frak{W}}^{\delta,2}_{\rho}(\Omega)$ is a closed subspace of $\frak{W}^{\delta,2}(\Omega)$ with its inner product inherited from $\frak{W}^{\delta,2}(\Omega)$.

\begin{definition}
	Let $\Omega \subset \bbR^d$ satisfy \eqref{assump:Domain} and $\rho$ satisfy \eqref{Assump:KernelSmoothness}. For $f \in [\frak{W}^{\delta,2}(\Omega)]^*$ and $g \in H^{-\frac{1}{2}}( \p \Omega)$ satisfying 
    \begin{equation}\label{eq:NeumannCompCond:Inhomog}
        \Vint{f,1} + \Vint{ g, 1} = 0\,,
    \end{equation}
 we say that $u \in \mathring{\frak{W}}^{\delta,2}_{\rho}(\Omega)$ is a \textit{weak solution} to \eqref{eq:Intro:NonlocalEq}-\eqref{eq:InHomogNeumannBC} if
\begin{equation}\label{eq:NonlocalProblem:Neumann:Inhomog:WeakForm}
		B_{\rho,\delta}(u,v) = 
 \Vint{f,v}
  +\Vint{g, Tv}\,,
  \quad\forall v \in \mathring{\frak{W}}^{\delta,2}_{\rho}(\Omega)\,.
	\end{equation}
\end{definition}

\begin{theorem}[Well-posedness]\label{thm:WellPosedness:Neumann:Inhomog}
	Let $\Omega \subset \bbR^d$ satisfy \eqref{assump:Domain} and  $\rho$ satisfy \eqref{Assump:KernelSmoothness}. Let $\delta < \delta_0$. Then there exists a constant $C = C(d,\rho,\Omega,C_N)$ such that
    \begin{equation}\label{eq:BilinearForm:Coercivity:Neumann:Inhomog}
		\Vnorm{u}_{\frak{W}^{\delta,2}(\Omega)}^2 \leq  C  B_{\rho,\delta}(u,u)\,, \qquad\;\forall u \in \mathring{\mathfrak{W}}^{\delta,2}_{\rho}(\Omega)\,.
	\end{equation}
	Moreover, 
	for any $f \in [\frak{W}^{\delta,2}(\Omega)]^*$ and $g \in H^{-\frac{1}{2}}(\p \Omega)$
with \eqref{eq:NeumannCompCond:Inhomog} satisfied,
there exists a unique solution $u \in \mathring{\mathfrak{W}}^{\delta,2}_{\rho}(\Omega)$ to \eqref{eq:NonlocalProblem:Neumann:Inhomog:WeakForm} satisfying the energy estimate
	\begin{equation}\label{eq:HomogNeumann:Inhomog:EnergyEstimate}
		\Vnorm{u}_{\frak{W}^{\delta,2}(\Omega)} \leq \Vnorm{f}_{[\frak{W}^{\delta,2}(\Omega)]^*} + \Vnorm{g}_{H^{-\frac{1}{2}}(\p \Omega)}\,.
	\end{equation}
\end{theorem}

\begin{proof}
    The right-hand side of \eqref{eq:NonlocalProblem:Neumann:Inhomog:WeakForm} defines an element of $[\frak{W}^{\delta,2}(\Omega)]^*$ acting on $v$ thanks to \Cref{thm:TraceTheorem}.
    The coercivity \eqref{eq:BilinearForm:Coercivity:Neumann:Inhomog} follows from \eqref{eq:KernelEquivalence} and \Cref{thm:PoincareNeumann}.
    We conclude the existence and uniqueness of a weak solution via the Lax-Milgram theorem thanks to the continuity of $B_{\rho,\delta}$ established in \eqref{eq:BilinearForm:Continuity}.
\end{proof}

\section{Boundary-localized convolutions associated to the nonlocal operator}\label{sec:HSEstimates}

For $\alpha \geq 0$ and for a measurable function $u : \Omega \to \bbR$, we recall the definition of the operator $K_{\delta,\alpha}$ and the function $\Phi_{\delta,\alpha}$ in \eqref{eq:ConvolutionOperatorAlpha}. We use the convention $K_{\delta} = K_{\delta,0}$.
(recall the same convention $\Phi_{\delta} = \Phi_{\delta,0}$ was used in \Cref{subsec:InHomoNeu}).

We define the adjoint operator $K_{\delta,\alpha}^*$ of
$K_{\delta,\alpha}$
by
\begin{equation}\label{eq:ConvolutionOperatorAdjointAlpha}
	K_{\delta,\alpha}^* u (\bx) := \int_{\Omega} \frac{1}{\Phi_{\delta,\alpha}(\by)}  \rho_{\delta,\alpha}(\bx,\by) u(\by) \, \rmd \by\,,
    \end{equation}
with the convention that $K_{\delta}^* = K_{\delta,0}^*$.
Then as a distribution
\begin{equation*}
    \Vint{ K_{\delta,\alpha} u,\varphi } = \Vint{u, K_{\delta,\alpha}^* \varphi } \text{ and } \Vint{ K_{\delta,\alpha}^* u,\varphi } = \Vint{u, K_{\delta,\alpha} \varphi }\,.
\end{equation*}

\begin{theorem}\label{thm:diffuKu}
	Let $\Omega \subset \bbR^d$ satisfy \eqref{assump:Domain}  and $\rho$ satisfy  \eqref{Assump:KernelSmoothness}. Let $\alpha \in \bbR$, and let $u \in L^2(\Omega)$. Then
	\begin{equation}\label{eq:Kdeltaerror1}
		\Vnorm{u - K_{\delta,\alpha} u }_{L^2(\Omega)}^2 \leq \iintdm{\Omega}{\Omega}{ \frac{1}{\Phi_{\delta,\alpha}(\bx)} \rho_{\delta,\alpha}(\bx,\by) |u(\bx)-u(\by)|^2 }{\by}{\bx}\,.
	\end{equation}
 Consequently, there exists a constant $C$ depending only on $d$, $\rho$, $\Omega$ and $\alpha$ such that
 	\begin{equation}\label{eq:KdeltaError}
  		\Vnorm{u - K_{\delta,\alpha} u }_{L^2(\Omega)} \leq 
    C \min\{\delta,\diam(\Omega)^2\} [u]_{\frak{W}^{\delta,2}(\Omega)}\,, \quad\forall u \in \frak{W}^{\delta,2}(\Omega).
 	\end{equation}
\end{theorem}

\begin{proof}
	By definition of $\Phi_{\delta,\alpha}$, H\"older's inequality gives
	\begin{equation*}
		\begin{split}
			\Vnorm{u - K_{\delta,\alpha} u}_{L^2(\Omega)}^2 &= \intdm{\Omega}{  \frac{1}{\Phi_{\delta,\alpha}(\bx)^2} \left( \intdm{\Omega}{ \rho_{\delta,\alpha}(\bx,\by) (u(\by)-u(\bx)) }{\by} \right)^2 }{\bx} \\
			&\leq \intdm{\Omega}{  \frac{1}{\Phi_{\delta,\alpha}(\bx)^2} \left( \intdm{\Omega}{\rho_{\delta,\alpha}(\bx,\by) }{\by} \right) \left( \intdm{\Omega}{ \rho_{\delta,\alpha}(\bx,\by) |u(\by)-u(\bx)|^2 }{\by} \right) }{\bx} \\
			&= \iintdm{\Omega}{\Omega}{ \frac{1}{\Phi_{\delta,\alpha}(\bx)} \rho_{\delta,\alpha}(\bx,\by) |u(\bx)-u(\by)|^2 }{\by}{\bx}\,,
		\end{split}
	\end{equation*}
 which is \eqref{eq:Kdeltaerror1}.
Now, from \eqref{eq:PhiBounds} and \Cref{lma:ComparabilityOfXandY} we have the estimate
	\begin{equation*}
		\begin{gathered}
		\frac{\rho_{\delta,\alpha}(\bx,\by)}{\Phi_{\delta,\alpha}(\bx)} \leq C(d,\Omega,\rho,\alpha) \rho_{\delta}(\bx,\by) \leq C(d,\Omega,\rho,\alpha) \eta_{\delta}(\bx)^2 \rho_{\delta,2}(\bx,\by)
		\end{gathered}
	\end{equation*}
in the right-hand side integral of \eqref{eq:Kdeltaerror1}, from which \eqref{eq:KdeltaError} follows.
\end{proof}

\begin{theorem}\label{thm:Convolution:DerivativeEstimate}
	Let $\Omega \subset \bbR^d$ satisfy \eqref{assump:Domain}  and $\rho$ satisfy  \eqref{Assump:KernelSmoothness}. Let $\alpha \geq 2$ and let $u \in L^2(\Omega)$. Then there exists $C = C(d,\rho,\Omega,\alpha) >0$ such that
	\begin{equation}\label{eq:ConvEst:Deriv}
		\begin{split}
			\Vnorm{ \grad K_{\delta,\alpha} u }_{L^2(\Omega)} 
			&\leq C \iintdm{\Omega}{\Omega}{  \rho_{\delta,2}(\bx,\by) |u(\bx)-u(\by)|^2 }{\by}{\bx} \\
			& \quad + C \iintdm{\Omega}{\Omega}{  \underline{\rho}_{\delta,2}(\bx,\by) |u(\bx)-u(\by)|^2 }{\by}{\bx}\,.
		\end{split}
	\end{equation}
\end{theorem}

\begin{proof}
	Assume the right-hand side of \eqref{eq:ConvEst:Deriv} is finite. We have
	\begin{equation*}
		\begin{split}
			&\grad K_{\delta,\alpha}u (\bx) = \frac{1}{\Phi_{\delta,\alpha}(\bx)} \intdm{\Omega}{\grad_{\bx} \rho_{\delta,\alpha}(\bx,\by) u(\by)}{\by}   - \frac{\grad \Phi_{\delta,\alpha}(\bx)}{\Phi_{\delta,\alpha}(\bx)^2} \intdm{\Omega}{\rho_{\delta,\alpha}(\bx,\by) u(\by)}{\by} \\
			&\quad= \frac{1}{\Phi_{\delta,\alpha}(\bx)} \intdm{\Omega}{\grad_{\bx} \rho_{\delta,\alpha}(\bx,\by) (u(\by)-u(\bx))}{\by} - \frac{\grad \Phi_{\delta,\alpha}(\bx)}{\Phi_{\delta,\alpha}(\bx)^2} \intdm{\Omega}{\rho_{\delta,\alpha}(\bx,\by) (u(\by)-u(\bx))}{\by}\,,
		\end{split}
	\end{equation*}
	where in the last line we added and subtracted $\frac{\grad \Phi_{\delta,\alpha}(\bx)}{\Phi_{\delta,\alpha}(\bx)} u(\bx)$. Therefore, by H\"older's inequality 
	\begin{equation*}
		\begin{split}
			\Vnorm{\grad K_{\delta,\alpha} u }_{L^2(\Omega)}^2 
			&\leq \int_{\Omega}  \left( \frac{\eta_{\delta}(\bx)}{\Phi_{\delta,\alpha}(\bx)}  \intdm{\Omega}{ |\grad_{\bx} \rho_{\delta,\alpha}(\bx,\bz)| }{\bz} \right) \intdm{\Omega}{\frac{|\grad_{\bx} \rho_{\delta,\alpha}(\bx,\by)|}{\eta_{\delta}(\bx) \Phi_{\delta,\alpha}(\bx)} |u(\by)-u(\bx)|^2}{\by}  \, \rmd \bx \\
			&\quad + \int_{\Omega} \frac{|\grad \Phi_{\delta,\alpha}(\bx)|^2}{|\Phi_{\delta,\alpha}(\bx)|^2} \left( \int_{\Omega} \frac{\rho_{\delta,\alpha}(\bx,\bz)}{\Phi_{\delta,\alpha}(\bx)} \, \rmd \bz \right) \intdm{\Omega}{\frac{\rho_{\delta,\alpha}(\bx,\by)}{\Phi_{\delta,\alpha}(\bx)} |u(\by)-u(\bx)|^2}{\by} \, \rmd \bx\,.
		\end{split}
	\end{equation*}
	We use \eqref{eq:KernelIntegralDerivativeEstimates:SqDist} and \eqref{eq:PhiBounds} in the first integral and the definition of $\Phi_{\delta,\alpha}$ and \eqref{eq:PhiDerivativeBounds} in the second integral to get
	\begin{equation}\label{eq:ConvEst:Deriv:Pf1}
		\begin{split}
		\Vnorm{\grad K_{\delta,\alpha} u }_{L^2(\Omega)}^2 &\leq C \int_{\Omega} \intdm{\Omega}{\frac{|\grad_{\bx} \rho_{\delta,\alpha}(\bx,\by)|}{\eta_{\delta}(\bx) \Phi_{\delta,\alpha}(\bx)} |u(\by)-u(\bx)|^2}{\by}  \, \rmd \bx \\
			&\quad + C \int_{\Omega} \intdm{\Omega}{\frac{\rho_{\delta,\alpha}(\bx,\by)}{\eta_{\delta}(\bx)^2 \Phi_{\delta,\alpha}(\bx)} |u(\by)-u(\bx)|^2}{\by} \, \rmd \bx\,.
			\end{split}
	\end{equation}
	From \Cref{thm:KernelDerivativeEstimates:SqDist} we have
	\begin{equation*}
		|\grad_{\bx} \rho_{\delta,\alpha}(\bx,\by)| \leq  \underline{\rho}_{\delta,\alpha+1}(\bx,\by) + C(d,\alpha) \rho_{\delta(\bx),\alpha+1}(|\by-\bx|)\,,
	\end{equation*}
	and so with \eqref{eq:PhiBounds} and \Cref{lma:ComparabilityOfXandY} we have the estimates 
	\begin{equation*}
		\begin{gathered}
		\frac{|\grad_{\bx} \rho_{\delta,\alpha}(\bx,\by)|}{\eta_{\delta}(\bx) \Phi_{\delta,\alpha}(\bx)} \leq C \underline{\rho}_{\delta,2}(\bx,\by) + C \rho_{\delta,2}(\bx,\by)\,, \\
		\frac{\rho_{\delta,\alpha}(\bx,\by)}{\eta_{\delta}(\bx)^2 \Phi_{\delta,\alpha}(\bx)} \leq C \rho_{\delta,2}(\bx,\by)\,.
		\end{gathered}
	\end{equation*}
	Using the previous two estimates in \eqref{eq:ConvEst:Deriv:Pf1} gives \eqref{eq:ConvEst:Deriv}.
\end{proof}

As a consequence of the embedding and characterization properties of the nonlocal function space proved in \Cref{thm:Embedding} and \Cref{thm:Coercivity}, we have the following corollary:

\begin{corollary}\label{cor:RegularityOfHSOp}
    Let $\Omega \subset \bbR^d$ satisfy \eqref{assump:Domain}  and $\rho$ satisfy  \eqref{Assump:KernelSmoothness}. Let $\alpha \geq 2$ and let $u \in \frak{W}^{\delta,2}(\Omega)$. Then there exists $C = C(d,\rho,\Omega,\alpha) >0$ such that
	\begin{equation}\label{eq:ConvEst:Deriv:Cor}
		\begin{split}
			\Vnorm{ \grad K_{\delta,\alpha} u }_{L^2(\Omega)} 
			\leq C [u]_{\frak{W}^{\delta,2}(\Omega)}\,.
		\end{split}
	\end{equation}
\end{corollary}

\begin{theorem}\label{thm:ConvEst}
	Let $\Omega \subset \bbR^d$ satisfy \eqref{assump:Domain}  and $\rho$ satisfy  \eqref{Assump:KernelSmoothness}. Let $\alpha \geq 0$.
	There exists a constant $C = C(d,\rho,\Omega,\alpha)>0$ such that
	\begin{equation}\label{eq:ConvEst:HMinus1}
		\Vnorm{\eta_{\delta} K_{\delta,\alpha}u }_{L^2(\Omega)} + \Vnorm{\eta_{\delta}^2 \grad K_{\delta,\alpha} u}_{L^2(\Omega)} \leq C \Vnorm{u}_{H^{-1}(\Omega)} \qquad \forall u \in H^{-1}(\Omega)\,,
	\end{equation}
	\begin{equation}\label{eq:ConvEst:L2}
		\Vnorm{K_{\delta,\alpha} u}_{L^2(\Omega)} + \Vnorm{\eta_{\delta} \grad K_{\delta,\alpha} u}_{L^2(\Omega)} \leq C \Vnorm{u}_{L^{2}(\Omega)} \qquad \forall u \in L^{2}(\Omega)\,,
	\end{equation}
	and
	\begin{equation}\label{eq:ConvEst:H1}
		\Vnorm{\grad K_{\delta,\alpha} u}_{L^2(\Omega)} \leq C \Vnorm{\grad u}_{L^{2}(\Omega)} \qquad \forall u \in H^{1}(\Omega)\,,
	\end{equation}
\end{theorem}

\begin{proof}
	We first prove \eqref{eq:ConvEst:L2}. By H\"older's inequality, Tonelli's theorem, \eqref{eq:PhiBounds} and \Cref{cor:KernelIntegral:Weighted:SqDist}
	\begin{equation}\label{eq:ConvEst:L2:Pf1}
		\begin{split}
			\Vnorm{K_{\delta,\alpha} u}_{L^2(\Omega)}^2 \leq \int_{\Omega} \left( \frac{\int_{\Omega}\rho_{\delta,\alpha}(\bx,\bz) \, \rmd \bz}{ \Phi_{\delta,\alpha}(\bx) } \right) \left( \frac{1}{\Phi_{\delta,\alpha}(\bx) } \int_{\Omega} \rho_{\delta,\alpha}(\bx,\by) |u(\by)|^2 \, \rmd \by \right) \rmd \bx \leq C \Vnorm{u}_{L^2(\Omega)}^2\,.
		\end{split}
	\end{equation}
	We estimate $\eta_{\delta} \grad K_{\delta,\alpha} u$ similarly, additionally using \eqref{eq:PhiDerivativeBounds} and \eqref{eq:KernelIntegralDerivativeEstimates:SqDist}:
	\begin{equation}\label{eq:ConvEst:L2:Pf2}
		\begin{split}
	&	\Vnorm{\eta_{\delta} \grad K_{\delta,\alpha} u}_{L^2(\Omega)}^2\\
  &\; \leq \int_{\Omega} \left( \eta_{\delta}(\bx) \frac{\int_{\Omega} |\grad_{\bx}  \rho_{\delta,\alpha}(\bx,\by) |\, \rmd \by}{ \Phi_{\delta,\alpha}(\bx) } \right) \left( \frac{\eta_{\delta}(\bx)}{\Phi_{\delta,\alpha}(\bx) } \int_{\Omega} |\grad_{\bx} \rho_{\delta,\alpha}(\bx,\by)| \, |u(\by)|^2 \, \rmd \by \right) \rmd \bx \\
		&\qquad + \int_{\Omega} \left( \frac{\int_{\Omega}  \rho_{\delta,\alpha}(\bx,\by) \, \rmd \by}{ \Phi_{\delta,\alpha}(\bx) } \right)\left( \frac{1}{\Phi_{\delta,\alpha}(\bx) } \int_{\Omega} \rho_{\delta,\alpha}(\bx,\by) |u(\by)|^2 \, \rmd \by \right) \rmd \bx \leq C \Vnorm{u}_{L^2(\Omega)}^2\,.
		\end{split}
	\end{equation}
	Then \eqref{eq:ConvEst:L2:Pf1}-\eqref{eq:ConvEst:L2:Pf2} establish \eqref{eq:ConvEst:L2}.

	Next we prove \eqref{eq:ConvEst:HMinus1}.
	We first recall the representation of distributions in $H^{-1}(\Omega)$. For any $u \in H^{-1}(\Omega)$ there exist $u_0 \in L^2(\Omega)$ and $\bu_1 \in L^2(\Omega;\bbR^d)$ such that
	\begin{equation}\label{eq:H1DualRep:1}
		\Vint{u,\varphi} = \Vint{u_0,\varphi}
  +\Vint{\bu_1,\grad\varphi}\,,  
  \quad \forall \varphi \in H^1_0(\Omega)\,,
	\end{equation}
	with
	\begin{equation}\label{eq:H1DualNorm:1}
		\Vnorm{u}_{H^{-1}(\Omega)}^2 = \inf \left\{ \Vnorm{ u_0 }_{L^2(\Omega)}^2 + \Vnorm{\bu_1}_{L^2(\Omega)}^2 \, : \, u_0 \text{ and } \bu_1 \text{ satisfy } \eqref{eq:H1DualRep:1} \right\}\,.
	\end{equation}
  
	Now let $\varphi \in L^2(\Omega)$ be arbitrary, and define
	\begin{equation*}
		\varphi_{\delta,\alpha}(\bx) := \int_{\Omega} \rho_{\delta,\alpha}(\bx,\by) \frac{\eta_{\delta}(\by)}{\Phi_{\delta,\alpha}(\by)} \varphi(\by) \, \rmd \by\,.
	\end{equation*}
	Then $\varphi_{\delta,\alpha} \in C^1(\Omega)$, and estimates similar to \eqref{eq:ConvEst:L2:Pf1}-\eqref{eq:ConvEst:L2:Pf2} reveal that $\varphi_{\delta,\alpha} \in H^1(\Omega)$ with
	\begin{equation}\label{eq:ConEst:HMinus1:Pf1}
		\Vnorm{\varphi_{\delta,\alpha}}_{H^1(\Omega)} \leq C \Vnorm{\varphi}_{L^2(\Omega)}\,.
	\end{equation}
	Furthermore, if $\varphi \in C^1(\overline{\Omega})$ then by \eqref{eq:ComparabilityOfDistanceFxn1:SqDist}-\eqref{eq:ComparabilityOfDistanceFxn2:SqDist} and \Cref{lma:KernelIntegral:SqDist}
	\begin{equation*}
		|\varphi_{\delta,\alpha}(\bx)| \leq C(\rho,\Omega) \eta_{\delta}(\bx) \int_{\Omega} \rho_{\eta_{\delta}(\by)}(|\by-\bx|) \varphi(\by) \, \rmd \by \leq C \Vnorm{\varphi}_{L^{\infty}(\Omega)} \eta_{\delta}(\bx)\,,
	\end{equation*}
	and so
	\begin{equation*}
		T \varphi_{\delta,\alpha} \equiv 0\;,\quad\forall \varphi \in C^1(\overline{\Omega})\,.
	\end{equation*}
	Let $\varphi_n$ be a sequence in $C^1(\overline{\Omega})$ converging to $\varphi$ in $L^2(\Omega)$. Then by \eqref{eq:ConEst:HMinus1:Pf1}
	\begin{equation*}
		\begin{split}
			\Vnorm{ T \varphi_{\delta,\alpha} }_{H^{1/2}(\p \Omega)} 
			&\leq \Vnorm{ T \varphi_{\delta,\alpha}- T (\varphi_n)_{\delta,\alpha} }_{H^{1/2}(\p \Omega)} \\
			&\leq C \Vnorm{ \varphi_{\delta,\alpha} - (\varphi_n)_{\delta,\alpha} }_{H^{1}(\Omega)} \\
			&\leq C \Vnorm{ \varphi - \varphi_n }_{L^{2}(\Omega)} \to 0 \text{ as } n \to \infty\,.
		\end{split}
	\end{equation*}
	Therefore
	\begin{equation*}
		\varphi_{\delta,\alpha} \in H^1_0(\Omega)\,,\quad\forall \varphi \in L^2(\Omega)\,.
	\end{equation*}
	Now let $u \in H^{-1}(\Omega)$. For arbitrary $\varphi \in L^2(\Omega)$ we will use \eqref{eq:H1DualRep:1} with test function $\varphi_{\delta,\alpha}$ to obtain that $\eta_{\delta} K_{\delta,\alpha} u$ actually defines a function in $L^2(\Omega)$. To begin, let $u_0$ and $\bu_1$ satisfy \eqref{eq:H1DualRep:1}, and write
	\begin{equation}\label{eq:ConEst:HMinus1:Pf2}
				\Vint{ \eta_{\delta} K_{\delta,\alpha} u, \varphi} = \Vint{ u, \varphi_{\delta,\alpha} } = \intdm{\Omega}{u_0 \varphi_{\delta,\alpha} }{\bx} + \intdm{\Omega}{ \bu_1 \cdot \grad (\varphi_{\delta,\alpha} ) }{\bx}\,.
	\end{equation}
	Using \eqref{eq:ComparabilityOfDistanceFxn1:SqDist}-\eqref{eq:ComparabilityOfDistanceFxn2:SqDist}, \eqref{eq:KernelDerivativeEstimate2:SqDist}, and additionally H\"older's inequality and \Cref{lma:KernelIntegral:SqDist},
	\begin{equation}\label{eq:ConEst:HMinus1:Pf3}
		\begin{split}
		\int_{\Omega} & \int_{\Omega} |u_0(\bx)| \frac{\eta_{\delta}(\by)}{\Phi_{\delta,\alpha}(\by)}  \rho_{\delta,\alpha}(\bx,\by) |\varphi(\by)| \, \rmd \bx \, \rmd \by\\
  &\qquad + \int_{\Omega} \int_{\Omega} |\bu_1(\bx)| \frac{\eta_{\delta}(\by)}{\Phi_{\delta,\alpha}(\by)}  |\grad_{\bx} \rho_{\delta,\alpha}(\bx,\by)| |\varphi(\by)| \, \rmd \bx \, \rmd \by \\
		&\leq C
  \int_{\Omega} \int_{\Omega} \big( \rho_{\delta,0}(\bx,\by) + \underline{\rho}_{\delta,0}(\bx,\by) \big)  \big( |u_0(\bx)| + |\bu_1(\bx)| \big)  |\varphi(\by)| \, \rmd \by \, \rmd \bx \\
	&\leq C \left( \int_{\Omega} \int_{\Omega} \big( \rho_{\delta,0}(\bx,\by) + \underline{\rho}_{\delta,0}(\bx,\by) \big)  |u_0(\bx)|^2 \, \rmd \by \, \rmd \bx \right)^{1/2} \\
  &\qquad\qquad
				\left( \int_{\Omega} \int_{\Omega} \big( \rho_{\delta,0}(\bx,\by) + \underline{\rho}_{\delta,0}(\bx,\by) \big)  |\varphi(\by)|^2  \, \rmd \by \, \rmd \bx \right)^{1/2} \\ 
			&\qquad + C \left( \int_{\Omega} \int_{\Omega} \big( \rho_{\delta,0}(\bx,\by) + \underline{\rho}_{\delta,0}(\bx,\by) \big)  |\bu_1(\bx)|^2  \, \rmd \by \, \rmd \bx \right)^{1/2}
   \\
  &\qquad\qquad\;
				\left( \int_{\Omega} \int_{\Omega} \big( \rho_{\delta,0}(\bx,\by) + \underline{\rho}_{\delta,0}(\bx,\by) \big)  |\varphi(\by)|^2  \, \rmd \by \, \rmd \bx \right)^{1/2} \\
		&\leq C \left( \Vnorm{u_0}_{L^2(\Omega)} + \Vnorm{\bu_1}_{L^2(\Omega)} \right) \Vnorm{\varphi}_{L^2(\Omega)}\,.
		\end{split}
	\end{equation}
	So can apply Fubini's theorem in \eqref{eq:ConEst:HMinus1:Pf2} to obtain the identity
	\begin{equation*}
		\begin{split}
		\Vint{ \eta_{\delta} K_{\delta,\alpha} u, \varphi} &= \intdm{\Omega}{ \left( \frac{\eta_{\delta}(\bx)}{\Phi_{\delta,\alpha}(\bx)} \int_{\Omega} \rho_{\delta,\alpha}(\bx,\by) u_0(\by) \, \rmd \by \right) \varphi(\bx) }{\bx} \\
		&\quad + \intdm{\Omega}{ \left( \frac{\eta_{\delta}(\bx)}{\Phi_{\delta,\alpha}(\bx)} \int_{\Omega} \grad_{\by} \rho_{\delta,\alpha}(\bx,\by) \bu_1(\by) \, \rmd \by \right) \varphi(\bx) }{\bx}\,.
		\end{split}
	\end{equation*}
	The estimate \eqref{eq:ConEst:HMinus1:Pf3} shows that this identity is independent of the choice of $u_0$ and $\bu_1$, and so $\eta_{\delta} K_{\delta,\alpha} u$ defines a measurable function on $\Omega$. Further, by \eqref{eq:ConEst:HMinus1:Pf3} and \eqref{eq:H1DualNorm:1}
	\begin{equation*}
		\Vnorm{\eta_{\delta} K_{\delta,\alpha} u}_{L^2(\Omega)} = \sup_{\Vnorm{\varphi}_{L^2(\Omega)} \leq 1} \Vint{ \eta_{\delta} K_{\delta,\alpha} u, \varphi } \leq C \Vnorm{u}_{H^{-1}(\Omega)}\,,
	\end{equation*}
	which is the first half of \eqref{eq:ConvEst:HMinus1}.
	For the other half, since for any $\varphi \in C^{\infty}_c(\Omega)$
	\begin{equation*}
		\Vint{\eta_{\delta}^2 \grad K_{\delta,\alpha} u, \varphi} = \Vint{ u, \bar{\varphi}_{\delta,\alpha} }\,, \quad \text{ where } \bar{\varphi}_{\delta,\alpha}(\bx) = \int_{\Omega} \grad_{\by} \left[ \frac{\rho_{\delta,\alpha}(\bx,\by) }{\Phi_{\delta,\alpha}(\by)} 
 \right] \eta_{\delta}^2(\by) \varphi(\by) \, \rmd \by\,,
	\end{equation*}
	the estimate $\Vnorm{\eta_{\delta}^2 \grad K_{\delta,\alpha} u}_{L^2(\Omega)} \leq C \Vnorm{u}_{H^{-1}(\Omega)}$ is established using a similar process. 
 
	Finally, by \Cref{thm:Convolution:DerivativeEstimate}, in order to establish \eqref{eq:ConvEst:H1} it suffices to show
	\begin{equation*}
		\begin{split}
			& \iintdm{\Omega}{\Omega}{  \rho_{\delta,2}(\bx,\by) |u(\bx)-u(\by)|^2 }{\by}{\bx} \\
			& \quad + \iintdm{\Omega}{\Omega}{  \underline{\rho}_{\delta,2}(\bx,\by) |u(\bx)-u(\by)|^2 }{\by}{\bx} \leq C \Vnorm{\grad u}_{L^2(\Omega)}^2\,.
		\end{split}
	\end{equation*}
	But this is a consequence of the embedding and characterization properties proved in \Cref{thm:Embedding} and \Cref{thm:Coercivity}, additionally using the upper bound on $\underline{\rho}$.
\end{proof}

\begin{lemma}\label{lma:SupportOfConv}
	Let $\Omega \subset \bbR^d$ satisfy \eqref{assump:Domain}  and $\rho$ satisfy  \eqref{Assump:KernelSmoothness}. Suppose that $\varphi \in L^2(\Omega)$ has compact support in $\Omega$, i.e. there exists $r(\varphi) > 0$ such that
	\begin{equation*}
		\supp \varphi \subset \{ \bx \in \Omega \, : \, \dist(\bx,\p \Omega) > \sqrt{r(\varphi)} \} = \bbR^d \setminus \Omega_{\sqrt{r(\varphi)}}\,.
	\end{equation*}
	Then $K_{\delta,\alpha} \varphi$ has compact support with
	\begin{equation}\label{eq:SupportOfConv}
		\supp K_{\delta,\alpha} \varphi \subset \left\{ \bx \in \Omega \, : \, \dist(\bx,\p \Omega) > \frac{1- \bar{\kappa}_0 R_0 \sqrt{\delta}}{1+\bar{\kappa}_0 R_0 \sqrt{\delta} } \sqrt{r(\varphi)} \right\}\,.
	\end{equation}
\end{lemma}
\begin{proof}
	By \Cref{lma:BoundaryRegions:SqDist} whenever $\dist(\bx, \p \Omega) <\frac{1- \bar{\kappa}_0 R_0 \sqrt{\delta}}{1+\bar{\kappa}_0 R_0 \sqrt{\delta} } \sqrt{r(\varphi)} $
	we have
	\begin{equation*}
		\{ \by \, : |\bx-\by| \leq R_0 \eta_{\delta}(\bx) \} \subset \Omega_{\sqrt{r(\varphi)}} \text{ and } \{ \by \, : |\bx-\by| \leq R_0 \eta_{\delta}(\by) \} \subset \Omega_{\sqrt{r(\varphi)}}\,.
	\end{equation*}
	Therefore, the domains of integration in the integrals defining $K_{\delta,\alpha}\varphi$ and $\supp \varphi$ are disjoint, and $K_{\delta,\alpha} \varphi(\bx) = 0$. 
\end{proof}

\begin{theorem}
	Let $\Omega \subset \bbR^d$ satisfy \eqref{assump:Domain}  and $\rho$ satisfy  \eqref{Assump:KernelSmoothness}. Let $\alpha \geq 0$.
	There exists a constant $C = C(d,\rho,\Omega,\alpha)>0$ such that
	\begin{equation}\label{eq:ConvAdjEst:HMinus1}
		\Vnorm{K_{\delta,\alpha}^* u }_{H^{-1}(\Omega)} + \Vnorm{\eta_{\delta} K_{\delta,\alpha}^* u }_{L^2(\Omega)} + \Vnorm{\eta_{\delta}^2 \grad K_{\delta,\alpha}^* u}_{L^2(\Omega)} \leq C \Vnorm{u}_{H^{-1}(\Omega)}\,, \quad \forall u \in H^{-1}(\Omega)
	\end{equation}
 and
\begin{equation}\label{eq:ConvAdjEst:L2}
		\Vnorm{K_{\delta,\alpha}^* u}_{L^2(\Omega)} + \Vnorm{\eta_{\delta} \grad K_{\delta,\alpha}^* u}_{L^2(\Omega)} \leq C \Vnorm{u}_{L^{2}(\Omega)}\,, \quad \forall u \in L^{2}(\Omega)\,.
	\end{equation}
\end{theorem}

\begin{proof}
	We only prove 
	\begin{equation}\label{eq:ConvAdjEst:HMinus1:Pf1}
		\Vnorm{K_{\delta,\alpha}^* u }_{H^{-1}(\Omega)} \leq C \Vnorm{u}_{H^{-1}(\Omega)}\,;
	\end{equation}
	the estimates for $\eta_{\delta} K_{\delta,\alpha}^* u$ and $\eta_{\delta}^2 \grad K_{\delta,\alpha}^* u$ are established similarly to \eqref{eq:ConvEst:HMinus1}.
	
	It is a consequence of \Cref{lma:SupportOfConv} that
	\begin{equation*}
		T [ K_{\delta,\alpha} \varphi] \equiv 0\,,\quad \forall  \varphi \in C^{\infty}_c(\Omega)\,.
	\end{equation*}
	Thanks to the continuity of the operator $K_{\delta,\alpha}$ with respect to the $H^1(\Omega)$-convergence implied by \eqref{eq:ConvEst:L2}-\eqref{eq:ConvEst:H1}, and thanks to continuity in $H^1(\Omega)$ of the trace operator, we conclude that $K_{\delta,\alpha} \varphi \in H^1_0(\Omega)$ whenever $\varphi \in H^1_0(\Omega)$. Hence 
	\begin{equation}\label{eq:ConvAdjEst:HMinus1:Pf2}
	| \Vint{ 
 \varphi , K_{\delta,\alpha}^* u
}_{H^1_0(\Omega),H^{-1}(\Omega)} | = | \Vint{ u, K_{\delta,\alpha} \varphi }_{H^1_0(\Omega),H^{-1}(\Omega)} | \leq C \Vnorm{u}_{H^{-1}(\Omega)} \Vnorm{K_{\delta,\alpha} \varphi }_{H^{1}(\Omega)}\,,
	\end{equation}
	and so by \eqref{eq:ConvEst:L2} and \eqref{eq:ConvEst:H1} we conclude \eqref{eq:ConvAdjEst:HMinus1:Pf1}.
	
	The proof of \eqref{eq:ConvAdjEst:L2} is similar to the proof of \eqref{eq:ConvEst:L2}.
\end{proof}

\subsection{Convergence in the localization limit}\label{sec:Localization}

\begin{theorem}\label{thm:ConvergenceOfConv}
	Let $\Omega \subset \bbR^d$ satisfy \eqref{assump:Domain} and $\rho$ satisfy \eqref{Assump:KernelSmoothness}. Let $\alpha \geq 0$ and $u \in L^2(\Omega)$. Then
\begin{equation}\label{eq:ConvergenceOfConv}
		\lim\limits_{\delta \to 0} \Vnorm{K_{\delta,\alpha} u - u}_{L^2(\Omega)} = 0
	\end{equation}
	and	\begin{equation}\label{eq:ConvergenceOfConvAdj}
		\lim\limits_{\delta \to 0} \Vnorm{K_{\delta,\alpha}^* u - u}_{L^2(\Omega)} = 0\,.
	\end{equation}
\end{theorem}

\begin{proof}
	First we prove \eqref{eq:ConvergenceOfConv}. It suffices to show that
	\begin{equation}\label{eq:ConvergenceOfConv:Pf1}
		K_{\delta,\alpha}u(\bx) \to u(\bx) \text{ almost everywhere in } \Omega\,,
	\end{equation}
	since then \eqref{eq:ConvEst:L2} with the dominated convergence theorem implies \eqref{eq:ConvergenceOfConv}. 
	
	Fix a Lebesgue point $\bx \in \Omega$ of $u$, so that 
	\begin{equation*}
		\lim\limits_{r \to 0} \fint_{B(\bx,r)} |u(\by) - u(\bx)| \, \rmd \by = 0
	\end{equation*}
	holds. Choose $\bar{\delta}_0$ depending on $\bx$ such that
	\begin{equation}\label{eq:DistAway}
		\dist(\bx,\p \Omega) \geq \frac{1}{\sqrt{\kappa_0}} \frac{1+\bar{\kappa}_0 R_0 \sqrt{\bar{\delta}_0} }{1 - \bar{\kappa}_0 R_0 \sqrt{\bar{\delta}_0} } \sqrt{\bar{\delta}_0}\,.
	\end{equation}
	Then \eqref{eq:XSetOutsideRLayer} and \eqref{eq:YSetOutsideRLayer} hold and $\bx \in \bbR^d \setminus \Omega_{\sqrt{\delta/\kappa_0}}$ for all $\delta < \bar{\delta}_0$.
	Therefore by definition of $\eta_{\delta}$
	\begin{equation*}
		\Phi_{\delta,\alpha}(\bx) = \int_{B(\bx,R_0 \delta)} \frac{1}{\delta^{d+\alpha}} \rho \left( \frac{|\by-\bx|}{\delta} \right) \, \rmd \by = \bar{\rho} \delta^{-\alpha}\,, \quad\;\forall \delta < \bar{\delta}_0
	\end{equation*}
	and so
	\begin{equation*}
		\begin{split}
			K_{\delta,\alpha} u(\bx) &= \frac{ 1 }{2 \Phi_{\delta,\alpha}(\bx)} \Bigg[
			\int_{ \{ |\by-\bx| \leq R_0 \eta_{\delta}(\bx) \} \cap (\bbR^d \setminus \Omega_{\sqrt{\delta/\kappa_0}}  )}
			\frac{1}{\eta_{\delta}(\bx)^{d+\alpha}} {\rho} \left( \frac{|\by-\bx|}{\eta_{\delta}(\bx)} \right) u(\by) \, \rmd \by \\
			&\qquad + 		\int_{ \{ |\by-\bx| \leq R_0 \eta_{\delta}(\bx) \} \cap (\bbR^d \setminus \Omega_{\sqrt{\delta/\kappa_0}}  )}
			\frac{1}{\eta_{\delta}(\by)^{d+\alpha}} {\rho} \left( \frac{|\by-\bx|}{\eta_{\delta}(\by)} \right) u(\by) \, \rmd \by \Bigg] \\
			&= \frac{ \delta^{\alpha} }{ \bar{\rho} }
			\int_{ B(\bx,R_0 \delta) }
			\frac{1}{\delta^{d+\alpha}} {\rho} \left( \frac{|\by-\bx|}{\delta} \right) u(\by) \, \rmd \by \\
			&=
			\int_{ B(\bx, \delta) }
			\frac{1}{\delta^{d}} \rho_1 \left( \frac{\by-\bx}{\delta} \right) u(\by) \, \rmd \by \,,\quad\;\forall \delta < \bar{\delta}_0\,,
		\end{split}
	\end{equation*}
	where, for $\bar{\rho}$ given by \eqref{eq:DefnOfL1NormOfRho} below,
 $\rho_1(\bx) := \frac{\rho(|\bx|)}{ \bar{\rho} }$ is a continuous function with support in $B(0,1)$ and satisfies $\Vnorm{\rho_1}_{L^1(\bbR^d)} = 1$. So for $\delta < \bar{\delta}_0$ we have
	\begin{equation*}
		|K_{\delta,\alpha} u(\bx) - u(\bx)| \leq \int_{ B(\bx,\delta) }
		\frac{1}{\delta^{d}} \rho_1 \left( \frac{\by-\bx}{\delta} \right) |u(\by)-u(\bx)| \, \rmd \by \leq C \fint_{B(\bx,\delta)} |u(\by)-u(\bx)| \, \rmd \by \to 0
	\end{equation*}
	as $\delta \to 0$.
	This limit holds for all Lebesgue points $\bx \in \Omega$ of $u$, that is, almost everywhere in $\Omega$, so \eqref{eq:ConvergenceOfConv:Pf1} is established.
	
	The proof of \eqref{eq:ConvergenceOfConvAdj} follows the same argument, with \eqref{eq:ConvAdjEst:L2} in place of \eqref{eq:ConvEst:L2}.
\end{proof}
It is straightforward to prove the following corollary using the same method:
\begin{corollary}\label{cor:ConvergenceOfConv}
    Let $\Omega \subset \bbR^d$ satisfy \eqref{assump:Domain} and $\rho$ satisfy \eqref{Assump:KernelSmoothness}. Let $\alpha \geq 0$ and $u \in H^1(\Omega)$. Then
\begin{equation}\label{eq:ConvergenceOfConv:H1}
		\lim\limits_{\delta \to 0} \Vnorm{K_{\delta,\alpha} u - u}_{H^1(\Omega)} = 0
	\end{equation}
	and
\begin{equation}\label{eq:ConvergenceOfConvAdj:H1}
		\lim\limits_{\delta \to 0} \Vnorm{K_{\delta,\alpha}^* u - u}_{H^1(\Omega)} = 0\,.
	\end{equation}
\end{corollary}

\subsection{Application: proofs of the Poincar\'e inequalities}\label{sec:PoincareProofs}

\begin{lemma}\label{lma:PhiConvergence}
    Let $\Omega$ satisfy \eqref{assump:Domain}  and $\rho$ satisfy  \eqref{Assump:KernelSmoothness}. Define the positive constant $\bar{\rho}$ by
\begin{equation}\label{eq:DefnOfL1NormOfRho}
    \bar{\rho} := \int_{B(0,1)} \rho(|\bz|) \, \rmd \bz\,.
\end{equation}
 Then
    \begin{equation*}
        \lim\limits_{\delta \to 0} \Vnorm{ \Phi_{\delta} - \bar{\rho}}_{L^p(\Omega)} = 0\,, \qquad \text{ for any } 1 \leq p < \infty\,.
    \end{equation*}
\end{lemma}

\begin{proof}
    By \eqref{eq:PhiBounds} and the dominated convergence theorem it suffices to show that
    \begin{equation*}
        \Phi_{\delta}(\bx) \to \bar{\rho} \text{ almost everywhere in } \Omega\,.
    \end{equation*}
    Fix $\bx \in \Omega$. Choose $\bar{\delta}_0$ depending on $\bx$ 
    such that
\eqref{eq:DistAway} holds.
	Then \eqref{eq:XSetOutsideRLayer} and \eqref{eq:YSetOutsideRLayer} hold 
 and $\bx \in \bbR^d \setminus \Omega_{\sqrt{\delta/\kappa_0}}$ for all $\delta < \bar{\delta}_0$.
	Therefore by definition of $\eta_{\delta}$
	\begin{equation*}
		\Phi_{\delta}(\bx) = \int_{B(\bx,R_0 \delta)} \frac{1}{\delta^{d}} \rho \left( \frac{|\by-\bx|}{\delta} \right) \, \rmd \by = \Vnorm{\rho}_{L^1(B(0,R_0))}\,, \quad\forall \delta < \bar{\delta}_0\,.
	\end{equation*}
    The conclusion follows.
\end{proof}

\begin{proof}[proof of \Cref{thm:PoincareDirichlet}]
    Note that $K_{\delta} u \in H^1(\Omega)$ by \Cref{cor:RegularityOfHSOp} and \eqref{eq:ConvEst:L2}.
    Let $\{ u_n \}$ be a sequence in $C^{\infty}_c(\Omega)$ converging to $u$ in $\frak{W}^{\delta,2}(\Omega)$. Then by the classical trace theorem for Sobolev functions, and again by \Cref{cor:RegularityOfHSOp} and \eqref{eq:ConvEst:L2}
    \begin{equation*}
        \Vnorm{ T(K_{\delta} u_n) -  T(K_{\delta} u) }_{H^{\frac{1}{2}}(\p \Omega)} \leq C \Vnorm{ K_{\delta}u_n - K_{\delta}u}_{H^1(\Omega)} \leq C \Vnorm{u_n - u}_{\frak{W}^{\delta,2}(\Omega)}\,.
    \end{equation*}
    However, by \Cref{lma:SupportOfConv}
    \begin{equation*}
        T(K_{\delta} u_n) = 0 \,,\quad\;\forall n\,.
    \end{equation*}
    It follows that the $H^{\frac{1}{2}}(\p \Omega)$-norm of $T (K_{\delta} u)$ can be made arbitrarily small by choosing $n$ sufficiently large; therefore
    \begin{equation*}
        T (K_{\delta} u) = 0 \text{ on } \p \Omega\,.
    \end{equation*}
    
    Hence $K_{\delta} u \in H^1_0(\Omega)$, and we can apply the classical Poincar\'e inequality:
    \begin{equation}\label{eq:PoincareDirichlet:Pf1}
        \Vnorm{ K_{\delta} u}_{L^2(\Omega)} \leq C(\Omega) \Vnorm{ \grad K_{\delta} u}_{L^2(\Omega)}\,.
    \end{equation}
    By \Cref{cor:RegularityOfHSOp}
    \begin{equation*}
        \Vnorm{ K_{\delta} u}_{L^2(\Omega)} \leq C(d,\rho,\Omega) [u]_{\frak{W}^{\delta,2}(\Omega)}\,.
    \end{equation*}
    Finally by \eqref{eq:KdeltaError}
    \begin{equation*}
        \begin{split}
            \Vnorm{u}_{L^2(\Omega)} &\leq \Vnorm{ K_{\delta} u}_{L^2(\Omega)} + \Vnorm{ u - K_{\delta} u}_{L^2(\Omega)} \\
            &\leq (C + C \min\{ \delta, \diam(\Omega)^2 \}
            ) [u]_{\frak{W}^{\delta,2}(\Omega)}\,.
        \end{split}
    \end{equation*}
\end{proof}

\begin{proof}[proof of \Cref{thm:PoincareNeumann}]
    First, by the local Poincar\'{e} inequality,  there exists a constant $\bar{C}=\bar{C}(d,\rho,\Omega,\delta) >0$ such that
\begin{equation}\label{eq:PoincareNeumann:Pf1}
        \Vnorm{v}_{L^2(\Omega)} \leq \bar{C} \Vnorm{ \grad v }_{L^2(\Omega)}
    \end{equation}
    for all $v \in H^1(\Omega)$ with $(\Phi_{\delta} v)_{\Omega} = 0$.
    
    Next, we show that in fact the constant $\bar{C}$ can be made independent of $\delta$. 
    Define the Rayleigh quotient
    \begin{equation*}
    \begin{split}
        \lambda_{\delta} &:= \min \left\{ \frac{ \int_{\Omega}|\grad u(\bx)|^2 \, \rmd \bx }{ \int_{\Omega} |u(\bx)|^2 \, \rmd \bx }  \, : \, u \in H^1(\Omega) \text{ with } (\Phi_{\delta} u)_{\Omega} =0 \right\}\,.
    \end{split}
    \end{equation*}
    By \eqref{eq:PoincareNeumann:Pf1}, $\lambda_{\delta} > 0$ for all $\delta > 0$.
    We claim that
    \begin{equation}\label{eq:PoincareNeumann:Pf1half}
        \inf_{\delta > 0} \lambda_{\delta} > 0\,.
    \end{equation}
    Suppose to the contrary; then there exists a subsequence (not relabeled) $\delta \to 0$ and a sequence of normalized eigenfunctions $\{ u_{\delta} \}_{\delta}$ for which the minimums $\lambda_{\delta}$ are attained with
    \begin{equation*}
        \Vnorm{u_{\delta}}_{L^2(\Omega)} = 1\,, \quad (\Phi_{\delta} u_{\delta}) = 0\,, \quad \Vnorm{ \grad u_{\delta} }_{L^2(\Omega)} \to 0 \text{ as } \delta \to 0\,.
    \end{equation*}
    Then there exists a function $u \in H^1(\Omega)$ such that $u_{\delta}$ converges strongly to $u$ in $L^2(\Omega)$ and weakly to $u$ in $H^1(\Omega)$. Moreover, $u(\bx) = c$, a constant, in $\Omega$.
    However, by \Cref{lma:PhiConvergence}
    \begin{equation*}
        c \bar{\rho} = \fint_{\Omega} u(\bx) \bar{\rho} \, \rmd \bx = \lim\limits_{\delta \to 0} \fint_{\Omega} u_{\delta}(\bx)  \Phi_{\delta}(\bx) \, \rmd \bx = 0\,.
    \end{equation*}
    Therefore it must be that $c =0$, a contradiction since $\Vnorm{u_{\delta}}_{L^2(\Omega)} =1$. 
    We then conclude from \eqref{eq:PoincareNeumann:Pf1half} that for any $\delta < \delta_0$
    \begin{equation}\label{eq:PoincareNeumann:Pf2}
        \Vnorm{v}_{L^2(\Omega)} \leq C(d,\rho,\Omega) \Vnorm{ \grad v }_{L^2(\Omega)} \,,\quad\;\forall v \in H^1(\Omega) \text{ with } (\Phi_{\delta} v)_{\Omega} = 0\,.
    \end{equation}
    
    Finally, let $u \in \frak{W}^{\delta,2}(\Omega)$ with $(\Phi_{\delta} u)_{\Omega} =0$. \Cref{cor:RegularityOfHSOp} and \eqref{eq:ConvEst:L2} imply that $K_{\delta} u \in H^1(\Omega)$, and we additionally have
    \begin{equation*}
        \int_{\Omega} \Phi_{\delta}(\bx) K_{\delta} u(\bx) \, \rmd \bx = \int_{\Omega} \int_{\Omega} \rho_{\delta}(\bx,\by) u(\by) \, \rmd \by \, \rmd \bx = \int_{\Omega} \Phi_{\delta}(\by) u(\by) \, \rmd \by = 0\,.
    \end{equation*}
    Applying the local, and uniform in $\delta$, Poincar\'e inequality established in \eqref{eq:PoincareNeumann:Pf2} to $K_{\delta} u$, we get:
    \begin{equation*}
        \Vnorm{K_{\delta} u}_{L^2(\Omega)} \leq C(d,\rho,\Omega) \Vnorm{ \grad K_{\delta} u }_{L^2(\Omega)}\,.
    \end{equation*}
    From here the proof proceeds identically to the proof of \Cref{thm:PoincareDirichlet}, \eqref{eq:PoincareDirichlet:Pf1}.
\end{proof}

\subsection{Application: regularity results for nonlocal problems}\label{subsec:Reg}
As an application of the estimates for the boundary-localized convolutions, we obtain some regularity results for the weak solutions of nonlocal problems.
Let us examine the case of the Dirichlet problem:

\begin{theorem}\label{thm:Regularity:FixedDelta}
	Let $\Omega \subset \bbR^d$ satisfy \eqref{assump:Domain}. Suppose that $u \in \frak{W}^{\delta,2}(\Omega)$ is a weak solution of \eqref{eq:Intro:NonlocalEq} with Dirichlet data, i.e. $u$ satisfies \eqref{eq:Dirichlet:Inhomog:NonlocalBC} and \eqref{eq:NonlocalProblem:WeakForm}, where additionally $f \in H^1_{loc}(\Omega)$. Then there exists $C$ depending only on $d$, $\Omega$ and $\rho$ such that
	\begin{equation}\label{eq:H1apriori}
 \Vnorm{u}_{H^1(\Omega)}  
  \leq C \left( \Vnorm{f}_{[\frak{W}^{\delta,2}_0(\Omega)]^*}  + \Vnorm{g}_{H^{\frac{1}{2}}(\p \Omega)} + \Vnorm{\eta_{\delta} f}_{L^2(\Omega)} + \Vnorm{ \eta_{\delta}^2 \grad f}_{L^2(\Omega)} \right)\,.
	\end{equation}
	 In particular, $u \in H^1(\Omega)$ whenever the right-hand side {of 
  \eqref{eq:H1apriori}} is finite.
\end{theorem}
\begin{proof}
        Let $\varphi \in C^{\infty}_c(\Omega)$. By \Cref{prop:GreensId:GeneralFxns}, case i),
    we conclude that the nonlocal Green's identity \eqref{eq:GreensIdentity:Intro} holds for $u$ and $\varphi$; that is,
    \begin{equation}\label{eq:DesiredIntByParts:SqDist}
		B_{\rho,\delta}(u,\varphi) = \int_{\Omega} \cL_{\delta}u(\bx) \varphi(\bx) \, \rmd \bx\,,
	\end{equation}
	since the boundary term is $0$.
	Therefore, if $f \in H^1_{loc}(\Omega)$ then by \eqref{eq:NonlocalProblem:WeakForm}
	\begin{equation*}
		\Vint{f, \varphi} = \int_{\Omega} f(\bx) \varphi(\bx) \, \rmd \bx = \int_{\Omega} \cL_{\delta} u(\bx) \varphi(\bx) \, \rmd \bx\,, \qquad\;\forall \varphi \in C^{\infty}_c(\Omega)\,.
	\end{equation*}
	Both $\cL_{\delta} u$ and $f$ are locally integrable, so the nonlocal equation \eqref{eq:Intro:NonlocalEq} holds pointwise
	 for almost every $ \bx \in \Omega$, and it follows that
\begin{equation}\label{eq:uAsAverage:f}
	u(\bx) = K_{\delta,2} u(\bx) + \frac{1}{2 \cdot  \Phi_{\delta,2}(\bx)} f(\bx)\,, \quad \text{ for almost every } \bx \in \Omega\,.
\end{equation}
	Then by \Cref{cor:RegularityOfHSOp} and \eqref{eq:InhomogDirichlet:EnergyEstimate}
	$$
	\Vnorm{\grad K_{\delta,2} u}_{L^2(\Omega)}^2 \leq C [u]_{\frak{W}^{\delta,2}(\Omega)}^2 \leq C \big( \Vnorm{f}_{[\frak{W}^{\delta,2}_0(\Omega)]^*}^2 + \Vnorm{g}_{H^{\frac{1}{2}}(\p \Omega)}^2 \big)\,.
	$$
	Next, we know that
\begin{equation}
\label{eq:Grad:fOverPhi}
	\grad \left[  \frac{1}{2 \cdot \Phi_{\delta,2}(\bx)} f(\bx) \right] =   \frac{1}{2 \cdot  \Phi_{\delta,2}(\bx)} \grad f(\bx) - \frac{\grad \Phi_{\delta,2}(\bx)}{2 \cdot  \Phi_{\delta,2}(\bx)^2} f(\bx)\,,
\end{equation}
	and so taking the $L^2(\Omega)$-norm on both sides and using the bounds on $\Phi_{\delta,2}$, its reciprocal and its derivatives,
\begin{equation}\label{eq:Norm:Grad:fOverPhi}
	\Vnorm{ \grad \left[  \frac{f}{\Phi_{\delta,2}} \right] }_{L^2(\Omega)} \leq C    \Vnorm{ \eta_{\delta}^2 \grad f}_{L^2(\Omega)} + C \Vnorm{\eta_{\delta} f}_{L^2(\Omega)}\,.
	\end{equation}
	Combining the previous three estimates with the pointwise equality \eqref{eq:Intro:NonlocalEq} and using \eqref{eq:InhomogDirichlet:EnergyEstimate} to estimate the $L^2(\Omega)$-norm of $u$,
    we conclude
	$$
\Vnorm{u}_{H^1(\Omega)} \leq C \left( \Vnorm{f}_{[\frak{W}^{\delta,2}_0(\Omega)]^*} +\Vnorm{g}_{H^{\frac{1}{2}}(\p \Omega)} + \Vnorm{ \eta_{\delta}^2 \grad f}_{L^2(\Omega)} + \Vnorm{\eta_{\delta} f}_{L^2(\Omega)} \right)\,.
	$$
\end{proof}
We can prove a statement of local regularity in a similar way:
\begin{theorem}
	Let $\Omega \subset \bbR^d$ satisfy \eqref{assump:Domain}, and let $\bx_0 \in \Omega$. Suppose that $f \in [\frak{W}^{\delta,2}_0(\Omega)]^*$ has the property that for every $\veps \in (0,\dist(\bx_0,\p \Omega))$ it agrees with a function $h \in H^1(\Omega \setminus \overline{B_{\veps}(\bx_0)} )$, i.e.
	\begin{equation*}
		\Vint{f,v} = \intdm{\Omega}{ h(\bx) v(\bx) }{\bx}\,, \quad \forall v \in H^{1}_0(\Omega) \text{ with } \supp v \subset \overline{\Omega} \setminus B_{\veps}(\bx_0)\,.
	\end{equation*}
	Suppose $u$ is a weak solution of the inhomogeneous Dirichlet problem, i.e. $u$ satisfies \eqref{eq:Dirichlet:Inhomog:NonlocalBC} and \eqref{eq:NonlocalProblem:WeakForm} with Poisson data $f$. Then $u \in H^1(\Omega \setminus \overline{B_{\veps}(\bx_0)})$ for every $\veps > 0$.
\end{theorem}

\begin{proof}
	Let $\varphi \in C^{\infty}_c(\Omega \setminus \overline{B_{\veps}(\bx_0)})$ be arbitrary. Then by \Cref{prop:GreensId:GeneralFxns}, case i), the nonlocal Green's identity \eqref{eq:GreensIdentity:Intro} holds:
	\begin{equation*}
	    \int_{\Omega} h(\bx) \varphi(\bx) \, \rmd \bx = B_{\rho,\delta}(u,\varphi) = \int_{\Omega} \cL_{\delta} u(\bx) \varphi(\bx) \, \rmd \bx\,.
	\end{equation*}
	Therefore
	\begin{equation*}
	    \cL_{\delta} u(\bx) = h(\bx) \qquad \text{ for almost every } \bx \in \Omega \setminus \overline{B_{\veps}(\bx_0)}\,.
	\end{equation*}
	So we see that the formula 
	\begin{equation*}
		u(\bx) =  K_{\delta,2} u(\bx) + \frac{1}{2 \cdot  \Phi_{\delta,2}(\bx)} h(\bx)
	\end{equation*}
	holds for almost every $\bx \in \Omega \setminus \overline{B_{\veps}(\bx_0)}$.
	
     By \Cref{cor:RegularityOfHSOp} and \eqref{eq:InhomogDirichlet:EnergyEstimate}
	$$
	\Vnorm{\grad K_{\delta,2} u}_{L^2(\Omega)}^2 \leq C [u]_{\frak{W}^{\delta,2}(\Omega)}^2 \leq C \big( \Vnorm{f}_{[\frak{W}^{\delta,2}_0(\Omega)]^*}^2 + \Vnorm{g}_{H^{\frac{1}{2}}(\p \Omega)}^2 \big) \,.
	$$
	Next, we know that
	$$
	\grad \left[  \frac{1}{2 \cdot \Phi_{\delta,2}(\bx)} h(\bx) \right] =   \frac{1}{2 \cdot  \Phi_{\delta,2}(\bx)} \grad h(\bx) - \frac{\grad \Phi_{\delta,2}(\bx)}{2 \cdot  \Phi_{\delta,2}(\bx)^2} h(\bx)\,,
	$$
	and so taking the $L^2(\Omega \setminus \overline{B_{\veps}(\bx_0)})$-norm on both sides and using the bounds on $\Phi_{\delta,2}$, its reciprocal and its derivatives,
	$$
	\Vnorm{ \grad \left[  \frac{h}{\Phi_{\delta,2}} \right] }_{L^2(\Omega \setminus \overline{B_{\veps}(\bx_0)})} \leq C    \Vnorm{ \eta_{\delta}^2 \grad h}_{L^2(\Omega \setminus \overline{B_{\veps}(\bx_0)})} + C \Vnorm{\eta_{\delta} h}_{L^2(\Omega \setminus \overline{B_{\veps}(\bx_0)})}\,.
	$$
	Combining the previous three estimates with the pointwise equality, we conclude
	$$
	\Vnorm{\grad u}_{L^2(\Omega \setminus \overline{B_{\veps}(\bx_0)})} \leq 
	C \left( \Vnorm{f}_{[\frak{W}^{\delta,2}_0(\Omega)]^*} 
	+ \Vnorm{ \eta_{\delta}^2 \grad h}_{L^2(\Omega \setminus \overline{B_{\veps}(\bx_0)})} 
	+ \Vnorm{\eta_{\delta} h}_{L^2(\Omega \setminus \overline{B_{\veps}(\bx_0)})} \right)\,.
	$$
\end{proof}

Next, we present a similar regularity result for Neumann problems. To do this, some changes must be made in order to account for the boundary data.

\begin{theorem}\label{thm:Regularity:FixedDelta:Neumann}
	Let $f \in [\frak{W}^{\delta,2}(\Omega)]^*$ and $g \in H^{-\frac{1}{2}}(\p \Omega)$ with $\Vint{f,1} +\Vint{g,1} = 0$. Suppose that $u \in \mathring{\frak{W}}^{\delta,2}_{\rho}(\Omega)$ is a weak solution of the inhomogeneous Neumann problem, i.e. $u$ satisfies \eqref{eq:NonlocalProblem:Neumann:Inhomog:WeakForm} where additionally $f \in H^1_{loc}(\Omega)$. Then there exists $C$ depending only on $\Omega$ and $\rho$ such that
	\begin{equation}\label{eq:H1Bound:Neumann:1}
		\Vnorm{u}_{H^1(\Omega)} \leq C \left( \Vnorm{f}_{[\frak{W}^{\delta,2}(\Omega)]^*} + \Vnorm{g}_{H^{-\frac{1}{2}}(\p \Omega)} + \Vnorm{\eta_{\delta} f}_{L^2(\Omega)} + \Vnorm{ \eta_{\delta}^2 \grad f}_{L^2(\Omega)} \right)\,.
	\end{equation}
	In particular, $u \in H^1(\Omega)$ whenever the right-hand side of the above inequality is finite.
\end{theorem}
\begin{proof}
    Let $\varphi \in C^{\infty}_c(\Omega)$.
    By \Cref{prop:GreensId:GeneralFxns}, case i), we conclude that the nonlocal Green's identity \eqref{eq:GreensIdentity:Intro} holds for $u$ and $\varphi$; that is,
	\begin{equation}\label{eq:DesiredIntByParts:SqDist:Neumann}
		B_{\rho,\delta}(u,\varphi) = \int_{\Omega} \cL_{\delta}u(\bx) \varphi(\bx) \, \rmd \bx\,,
	\end{equation}
	since the boundary term is $0$.
	
	Now for $\varphi \in C^{\infty}_c(\Omega)$ arbitrary, define $\psi(\bx) = \varphi(\bx) - (\Phi_{\delta} \varphi)_{\Omega}$. Then we can use $\psi$ as a test function
	in \eqref{eq:NonlocalProblem:Neumann:Inhomog:WeakForm} to get
	\begin{equation*}
		\begin{split}
		\int_{\Omega} \cL_{\delta} u(\bx) \varphi(\bx) \, \rmd \bx &= B_{\rho,\delta}(u,\varphi) = B_{\rho,\delta}(u,\psi)
		=\Vint{f, \psi} + \Vint{g,\psi}
        \\
        &= \Vint{f, \varphi} + \Vint{g,\varphi}
        - (\Phi_{\delta} \varphi)_{\Omega} \left( \Vint{f,1} + \Vint{g,1}
            \right) = \Vint{f, \varphi}\,.
		\end{split}
	\end{equation*}
	So if $f \in H^1_{loc}(\Omega)$ then
	\begin{equation*}
		\int_{\Omega} \cL_{\delta} u(\bx) \varphi(\bx) \, \rmd \bx  = \int_{\Omega} f(\bx) \varphi(\bx) \, \rmd \bx  \qquad\;\forall \varphi \in C^{\infty}_c(\Omega)\,.
	\end{equation*}
	Both $\cL_{\delta} u$ and $f$ are locally integrable, so \eqref{eq:uAsAverage:f} holds almost everywhere in $\Omega$. Proceeding exactly as in the proof of \Cref{thm:Regularity:FixedDelta} but with \eqref{eq:HomogNeumann:Inhomog:EnergyEstimate} in place of \eqref{eq:InhomogDirichlet:EnergyEstimate} we come to
	$$
	\Vnorm{u}_{H^1(\Omega)} \leq C \left( \Vnorm{f}_{[\frak{W}^{\delta,2}(\Omega)]^*} + \Vnorm{g}_{H^{-\frac{1}{2}}(\p \Omega)} + \Vnorm{ \eta_{\delta}^2 \grad f}_{L^2(\Omega)} + \Vnorm{\eta_{\delta} f}_{L^2(\Omega)} \right)\,.
	$$
\end{proof}

\section{Consistency with classical boundary-value problems}\label{sec:LocalLimit}

Our aim is to show the consistency of the nonlocal boundary value problems with their local limits without assuming extra regularity on the data and solutions than those established earlier. To this end, we first examine the local limits of the nonlocal operators and the nonlocal bilinear forms in suitable function spaces.

\subsection{The operator in the localization limit}\label{sec:OpLocalLimit}

\begin{theorem}\label{thm:OperatorLocalization:SqDist}
	Let $\Omega \subset \bbR^d$ satisfy \eqref{assump:Domain}  and $\rho$ satisfy  \eqref{Assump:KernelSmoothness}-\eqref{Assump:KernelNormalization}. Let $u \in W^{2,p}(\Omega)$ for some $p \in [1,\infty)$. Then 
	\begin{equation*}
		\lim\limits_{\delta \to 0} \int_{\Omega} |\cL_{\delta} u(\bx) - ( -\Delta u(\bx))|^p \, \rmd \bx = 0\,.
	\end{equation*}
\end{theorem}

\begin{proof}
	By \Cref{thm:OpIsDefined:Inhomog:SqDist} it suffices to prove the theorem for $u \in C^2(\overline{\Omega})$. 

 Fix $\bx \in \Omega$. Choose $\bar{\delta}_0$ depending on $\bx$ such that
 \eqref{eq:DistAway} holds.
	Then for all $\delta < \bar{\delta}_0$, \eqref{eq:XSetOutsideRLayer} and \eqref{eq:YSetOutsideRLayer} hold and $\bx \in \bbR^d \setminus \Omega_{\sqrt{\delta/\kappa_0}}$.
    Therefore by definition of $\eta_{\delta}$
	\begin{equation*}
		\cL_{\delta}u(\bx) = \int_{B(\bx,R_0 \delta)} \frac{2}{ \delta^{d+2} } \rho \left( \frac{|\bx-\by|}{\delta} \right) (u(\bx)-u(\by)) \, \rmd \by \qquad\;\forall \delta < \bar{\delta}_0\,.
	\end{equation*}
	Taylor expanding, we have 
	\begin{equation*}
		\begin{split}
		&	-\cL_{\delta} u(\bx) = \int_{B(\bx,R_0 \delta)} \frac{2}{ \delta^{d+2} } \rho \left( \frac{|\bx-\by|}{\delta} \right) (\by-\bx) \, \rmd \by \cdot \grad u (\bx) \\
			&\qquad + \int_{B(\bx,R_0 \delta)} \frac{2}{ \delta^{d+2} } \rho \left( \frac{|\bx-\by|}{\delta} \right) \int_0^1 \Vint{ \grad^2 u(\by + t(\bx-\by)) (\bx-\by),(\bx-\by)} (1-t) \, \rmd t \, \rmd \by \\
			&\quad =  \int_{B(0,R_0)} 2 \rho \left( |\bz| \right) \int_0^1 \Vint{ \grad^2 u(\bx + (1- t) \delta \bz) \bz,\bz} (1-t) \, \rmd t \, \rmd \bz\,.
		\end{split}
	\end{equation*}
 Noting that
	$$
	2 \int_0^1 (1-t) \, \rmd t \int_{B(0,R_0)} \rho(|\bz|) \Vint{ \grad^2 u(\bx) \bz,\bz } \, \rmd \bz = \Delta u(\bx)\,,
	$$
	we get
	\begin{equation*}
		\begin{split}
			& |\cL_{\delta}u(\bx) - (-\Delta u(\bx))|  \\
			&= \left| \int_{B(0,R_0)} \int_0^1 2 (1-t) \rho \left( |\bz| \right)  \Vint{ \left( \grad^2 u(\bx + (1- t) \delta \bz) - \grad^2 u(\bx) \right)  \bz,\bz} \, \rmd t \, \rmd \bz \right| \,.
		\end{split}
	\end{equation*}
	Since $\grad^2 u$ is continuous on $\overline{\Omega}$, by the dominated convergence theorem, we get as $\delta \to 0$ that $|\cL_{\delta}u(\bx) - (-\Delta u(\bx))| \to 0$ for almost every $\bx \in \Omega$.
	
	By \Cref{lma:OpTaylorExp:SqDist}, $|\cL_{\delta}u(\bx) - (-\Delta u(\bx))|^p$ is majorized independently of $\delta$ by a function belonging to $L^1(\Omega)$, and so we obtain the result by again applying the dominated convergence theorem.
\end{proof}

\subsection{The bilinear form in the localization limit}\label{sec:BiLOcalLimit}

\begin{theorem}\label{thm:BilinearFormLocalization:SqDist}
	Let $\Omega \subset \bbR^d$ satisfy \eqref{assump:Domain}  and $\rho$ satisfy  \eqref{Assump:KernelSmoothness}-\eqref{Assump:KernelNormalization}. Then for any $u$, $v \in H^1(\Omega)$
	\begin{equation*}
		\lim\limits_{\delta \to 0} B_{\rho,\delta}(u,v) = \int_{\Omega} \grad u(\bx) \cdot \grad v(\bx) \, \rmd \bx\,.
	\end{equation*}
\end{theorem}

\begin{proof}
	The proof is similar to \Cref{thm:OperatorLocalization:SqDist}.
    By \Cref{thm:Embedding} it suffices to prove the theorem for $u$, $v \in C^2(\overline{\Omega})$.  Fix $\bx \in \Omega$. Choose $\bar{\delta}_0$ depending on $\bx$ such that \eqref{eq:DistAway} holds.
	Then for all $\delta < \bar{\delta}_0$, \eqref{eq:XSetOutsideRLayer} and \eqref{eq:YSetOutsideRLayer} hold and $\bx \in \bbR^d \setminus \Omega_{\sqrt{\delta/\kappa_0}}$.
    Therefore by definition of $\eta_{\delta}$
	\begin{equation*}
		\begin{split}
		\int_{\Omega} &\rho_{\delta,2} \left( \bx,\by \right) (u(\bx)-u(\by)) (v(\bx)-v(\by)) \, \rmd \by \\
		&=\int_{B(\bx,R_0 \delta)} \frac{1}{ \delta^{d+2} } \rho \left( \frac{|\bx-\by|}{\delta} \right) (u(\bx)-u(\by)) (v(\bx)-v(\by)) \, \rmd \by \qquad\;\forall \delta < \bar{\delta}_0\,.
		\end{split}
	\end{equation*}
	Taylor expanding and then using \eqref{Assump:KernelNormalization} and \eqref{eq:KernelNormalizationConsequence}, we have 
	\begin{equation*}
		\begin{split}
			\int_{\Omega} &\rho_{\delta,2} \left( \bx,\by \right) (u(\bx)-u(\by)) (v(\bx)-v(\by)) \, \rmd \by \\
			&= \Vint{ \left( \int_{B(\bx,R_0 \delta)} \frac{1}{ \delta^{d} } \rho \left( \frac{|\bx-\by|}{\delta} \right)  \frac{(\by-\bx) \otimes (\by-\bx)}{\delta^2}   \, \rmd \by \right) \grad u(\bx) , \grad v (\bx) } \\
			&\quad + O \left( \int_{B(\bx,R_0 \delta)} \frac{1}{ \delta^{d+2} } \rho \left( \frac{|\bx-\by|}{\delta} \right) |\by-\bx|^3 \, \rmd \by \right) \\
			&=  \grad u(\bx) \cdot \grad v(\bx) + O(\delta)\,.
		\end{split}
	\end{equation*}
	Since $\grad^2 u$ is continuous on $\overline{\Omega}$, the result follows by the dominated convergence theorem after integrating both sides in $\bx$.
\end{proof}

\begin{theorem}\label{thm:BilinearFormLocalization:Sequence:SqDist}
	Let $\Omega \subset \bbR^d$ satisfy \eqref{assump:Domain}  and $\rho$ satisfy  \eqref{Assump:KernelSmoothness}-\eqref{Assump:KernelNormalization}. Suppose that $v \in H^1(\Omega)$, and suppose that $\{ u_{\delta} \} \subset H^1(\Omega)$ and $u \in H^1(\Omega)$ satisfy $u_{\delta} \rightharpoonup u$ in $H^1(\Omega)$ as $\delta \to 0$. Then
	\begin{equation*}
		\lim\limits_{\delta \to 0} B_{\rho,\delta}(u_{\delta} - u,v) = 0\,.
	\end{equation*}
\end{theorem}

\begin{proof}
	Without loss of generality assume $u_{\delta} \rightharpoonup 0$ in $H^1(\Omega)$ as $\delta \to 0$.
    First take $v \in C^2(\overline{\Omega})$. Then by the nonlocal Green's identity \Cref{prop:GreensId:GeneralFxns}, case ii),
    \begin{equation*}
        B_{\rho,\delta}(u_{\delta},v) = \int_{\Omega} \cL_{\delta} v(\bx) u_{\delta}(\bx) \, \rmd \bx + \int_{\p \Omega} u_{\delta}(\bx) \frac{\p v}{\p \bsnu}(\bx) \, \rmd \sigma(\bx)\,.
    \end{equation*}
    By \Cref{thm:OpIsDefined:Inhomog:SqDist} and \Cref{thm:OperatorLocalization:SqDist} we obtain $\cL_{\delta} v \to -\Delta v$ strongly in $L^2(\Omega)$ as $\delta \to 0$, and by the compact embedding of $H^1(\Omega)$ into $L^2(\Omega)$ we have $u_{\delta} \to 0$ strongly in $L^2(\Omega)$. So, 
    the first term on the right-hand side converges to $0$ as $\delta \to 0$. The second term on the right-hand side converges to $0$ as well, by the weak continuity  of traces from $H^1(\Omega)$ to $H^{\frac{1}{2}}(\p \Omega)$. Therefore
    \begin{equation*}
        \lim\limits_{\delta \to 0} B_{\rho,\delta}(u_{\delta},v) = 0 \quad \forall v \in C^2(\overline{\Omega})\,.
    \end{equation*}

    Now let $v \in H^1(\Omega)$. Let $\{v_n \}$ be a sequence in $C^2(\overline{\Omega})$ that converges to $v$ in $H^1(\Omega)$. Then by H\"older's inequality and \Cref{thm:Embedding}
    \begin{equation*}
    \begin{split}
        |B_{\rho,\delta}(u_{\delta},v)| &\leq |B_{\rho,\delta}(u_{\delta},v_n-v)| + |B_{\rho,\delta}(u_{\delta},v_n)| \\
        &\leq C\Vnorm{v_n -v}_{H^1(\Omega)} \Vnorm{u_{\delta}}_{H^1(\Omega)} + |B_{\rho,\delta}(u_{\delta},v_n)|\,.
    \end{split}
    \end{equation*}
    Since $\Vnorm{u_{\delta}}_{H^1(\Omega)}$ is bounded uniformly with respect to $\delta$, we can choose $n$ large enough so that the first term is arbitrarily small independent of $\delta$. Then we use the first part of the proof to let $\delta \to 0$ in the second term and obtain that 
    $$
    \limsup_{\delta \to 0} |B_{\rho,\delta}(u_{\delta},v)| = 0 \qquad \forall v \in H^1(\Omega)\,.
    $$
\end{proof}

\subsection{The Dirichlet problem}\label{sec:HomogDirichlet}

For this section we consider $\delta$ varying, and analyze the case $\delta  \to 0$.
Note that while an element of $H^{-1}(\Omega)$ can serve as the Poisson data $f$ for the local problem \eqref{eq:Intro:LocalEqn}, a properly regularized version, denoted by $f_\delta$, has to be used for a well-posed nonlocal problem.
We first show that the sequence of solutions to the nonlocal problem with the regularized $f_\delta$ converges as $\delta \to 0$ to the unique variational solution of 
\eqref{eq:Intro:LocalEqn}
with the Poisson data $f$, all subject to the same inhomogeneous Dirichlet boundary conditions \eqref{eq:Dirichlet:Inhomog:NonlocalBC}. We then illustrate how the regularization can be obtained. Further, we present the consistency of the outward normal derivative on $\p \Omega$.

\begin{theorem}\label{thm:Regularity:VaryingDelta}
	Let $\Omega \subset \bbR^d$ satisfy \eqref{assump:Domain}. Let $f \in H^{-1}(\Omega)$ and let $g \in H^{\frac{1}{2}}(\p \Omega)$. Suppose for each $\delta$ there exists a distribution $f_{\delta}$ that satisfies
	\begin{equation}\label{eq:ApproximationWeakConvergence}
		\lim\limits_{\delta\to 0^+} \Vint{f_{\delta} - f,v}_{H^{-1}(\Omega),H^1_0(\Omega)} = 0 \qquad\;\forall v \in H^1_0(\Omega)\,,
	\end{equation}
	and that there exists $C_0 > 0$ independent of $\delta$ such that
	\begin{equation}\label{eq:ApproximationInequality}
		\left( \Vnorm{f_{\delta}}_{[\frak{W}^{\delta,2}_0(\Omega)]^*}  + \Vnorm{\eta_{\delta} f_{\delta}}_{L^2(\Omega)} + \Vnorm{ \eta_{\delta}^2 \grad f_{\delta}}_{L^2(\Omega)} \right) \leq C_0 \Vnorm{f}_{H^{-1}(\Omega)}\,.
	\end{equation}
	Then the corresponding solutions $u_{\delta}$ of\begin{equation}\label{eq:NonlocalProblem:WeakForm-Regu}
	B_{\rho,\delta}(u_{\delta},v) = \Vint{f_{\delta},v}\,, \quad \forall v \in \frak{W}^{\delta,2}_0(\Omega)\,,
 \end{equation} with Poisson data $f_{\delta}$ and Dirichlet boundary data $g$ satisfy
	\begin{equation}\label{eq:UniformH1BoundOnSolns}
		\Vnorm{u_{\delta}}_{H^1(\Omega)} \leq C(d,\Omega,\rho) \big(  C_0 \Vnorm{f}_{H^{-1}(\Omega)} + \Vnorm{g}_{H^{\frac{1}{2}}(\p \Omega)} \big)\,.
	\end{equation}
\end{theorem}

\begin{proof} The result follows from \eqref{eq:H1apriori} and \eqref{eq:ApproximationInequality}.
\end{proof}

\begin{theorem}\label{thm:LocalLimit:Dirichlet}
	Let $\Omega \subset \bbR^d$ satisfy \eqref{assump:Domain}. Suppose that $f \in H^{-1}(\Omega)$, and let $f_{\delta}$ be a sequence of functions that satisfy \eqref{eq:ApproximationWeakConvergence}-\eqref{eq:ApproximationInequality}. Then the solutions $u_{\delta} \in H^1_0(\Omega)$ to \eqref{eq:NonlocalProblem:WeakForm-Regu} converge weakly in $H^1(\Omega)$ to a function $u \in H^1_0(\Omega)$, where $u$ is the unique weak solution to the Poisson equation \eqref{eq:Intro:LocalEqn}
with Dirichlet boundary conditions \eqref{eq:Dirichlet:Inhomog:NonlocalBC}.
\end{theorem}

\begin{proof}
	Since the solutions satisfy the uniform $H^1$ bound \eqref{eq:UniformH1BoundOnSolns}, it follows that they converge weakly in $H^1(\Omega)$ to a function $u$. 
    By weak continuity of traces in $H^1(\Omega)$, since $Tu_{\delta} \equiv g$ for all $\delta$ we have that $T u = g$ as well.

	We now prove that $u$ solves the weak form of the Poisson equation, i.e.
	\begin{equation}\label{eq:LocalProblem:WeakForm:Dirichlet:Inhomog}
		\int_{\Omega} \grad u \cdot \grad v \, \rmd \bx = \Vint{f,v}\,, \quad\;\forall v \in H^1_0(\Omega)\,.
	\end{equation}
 For each $\delta > 0$,
	\begin{equation*}
		\Vint{ f_{\delta},v } = B_{\rho,\delta}(u_{\delta},v)\,.
	\end{equation*}
	On one hand, \eqref{eq:ApproximationWeakConvergence} implies
	\begin{equation*}
		\lim\limits_{\delta \to 0} \Vint{ f_{\delta},v } = \Vint{ f,v }\,,
	\end{equation*}
	and on the other hand, \Cref{thm:BilinearFormLocalization:Sequence:SqDist} and \Cref{thm:BilinearFormLocalization:SqDist} imply that
	\begin{equation*}
		\lim\limits_{\delta \to 0} B_{\rho,\delta}(u_{\delta},v) = \lim\limits_{\delta \to 0} B_{\rho,\delta}(u_{\delta}-u,v) + \lim\limits_{\delta \to 0} B_{\rho,\delta}(u,v) = \int_{\Omega} \grad u \cdot \grad v \, \rmd \bx\,.
	\end{equation*}
	Therefore $u$ satisfies \eqref{eq:LocalProblem:WeakForm:Dirichlet:Inhomog}, and so in summary $u$ is the unique weak solution in  $H^1(\Omega)$ of the Poisson equation \eqref{eq:Intro:LocalEqn}
with Dirichlet boundary conditions \eqref{eq:Dirichlet:Inhomog:NonlocalBC}.
\end{proof}

\subsubsection{Regularizing rough data}
\label{sec:DataMollification}

The next theorem gives an explicit construction of the mollified sequence $f_{\delta}$ that satisfies \eqref{eq:ApproximationWeakConvergence}-\eqref{eq:ApproximationInequality}.

\begin{theorem}\label{thm:DataMollification}
	Let $\Omega \subset \bbR^d$ satisfy \eqref{assump:Domain}  and $\rho$ satisfy  \eqref{Assump:KernelSmoothness}-\eqref{Assump:KernelNormalization}. For $f \in H^{-1}(\Omega)$ and some $\alpha \geq 0$ fixed, define the distribution
	\begin{equation}\label{eq:Dirichlet:MollifiedDataDefn}
		f_{\delta} := K_{\delta,\alpha}^* f\,.
	\end{equation}
	Then $f_{\delta}$ satisfies \eqref{eq:ApproximationWeakConvergence}-\eqref{eq:ApproximationInequality}.
\end{theorem}

\begin{proof}
	Similarly to the proof of \eqref{eq:ConvAdjEst:HMinus1:Pf1}, 
	it is a consequence of \Cref{lma:SupportOfConv} that
	\begin{equation*}
		T [ K_{\delta,\alpha} \varphi] \equiv 0\,,  \forall \varphi \in C^{\infty}_c(\Omega)\,.
	\end{equation*}
	Thanks to the continuity of the operator $K_{\delta,\alpha} : \frak{W}^{\delta,2}(\Omega) \to H^1(\Omega)$ implied by \eqref{eq:ConvEst:Deriv:Cor}-\eqref{eq:ConvEst:L2}, and thanks to continuity in $\frak{W}^{\delta,2}(\Omega)$ of the trace operator, we conclude that $K_{\delta,\alpha} \varphi \in H^1_0(\Omega)$ whenever $\varphi \in \frak{W}^{\delta,2}_0(\Omega)$. Hence 
	\begin{equation*}
		| \Vint{
  \varphi, K_{\delta,\alpha}^* f
}_{\frak{W}^{\delta,2}_0(\Omega) , [\frak{W}^{\delta,2}(\Omega) ]^* } | = | \Vint{ f, K_{\delta,\alpha} \varphi }_{H^1_0(\Omega),H^{-1}(\Omega)} | \leq C \Vnorm{f}_{H^{-1}(\Omega)} \Vnorm{K_{\delta,\alpha} \varphi }_{H^{1}(\Omega)}\,,
	\end{equation*}
	and so by \eqref{eq:ConvEst:Deriv:Cor}-\eqref{eq:ConvEst:L2} we conclude that
	\begin{equation*}
			\Vnorm{ f_{\delta} }_{[\frak{W}^{\delta,2}_0(\Omega) ]^*}  = \Vnorm{ K_{\delta,\alpha}^* f }_{[\frak{W}^{\delta,2}_0(\Omega) ]^*} \leq C \Vnorm{f}_{H^{-1}(\Omega)}\,.
	\end{equation*}
	The estimate
	\begin{equation*}
		\Vnorm{\eta_{\delta} f_{\delta}}_{L^2(\Omega)} + \Vnorm{ \eta_{\delta}^2 \grad f_{\delta}}_{L^2(\Omega)} \leq C \Vnorm{f}_{H^{-1}(\Omega)}
	\end{equation*}
	follows from \eqref{eq:ConvAdjEst:HMinus1}, and therefore \eqref{eq:ApproximationInequality} is established by the previous two estimates.
	
    To prove \eqref{eq:ApproximationWeakConvergence}, we first note that by \eqref{eq:ConvergenceOfConv:H1}
	\begin{equation}\label{eq:DataMollification:Pf1}
		\lim\limits_{\delta \to 0} \Vnorm{ K_{\delta,\alpha} \varphi - \varphi }_{H^1(\Omega)} = 0\,, \quad\;\forall \varphi \in C^{\infty}_c(\Omega)\,.
	\end{equation}
    Now let $\varphi \in H^1_0(\Omega)$, and let $(\varphi_n)$ be a sequence in $C^{\infty}_c(\Omega)$ that converges to $\varphi$ in $H^1(\Omega)$. Then by \eqref{eq:ConvAdjEst:HMinus1:Pf1}
	\begin{equation*}
		\begin{split}
		|\Vint{f_{\delta} - f, \varphi}| &\leq |\Vint{f_{\delta} - f, \varphi_n}| + C(\rho,\Omega) \Vnorm{ f }_{H^{-1}(\Omega)}  \Vnorm{\varphi_n - \varphi}_{H^1(\Omega)} \\
		&= |\Vint{f, K_{\delta,\alpha} \varphi_n - \varphi_n}| + C(\rho,\Omega) \Vnorm{ f }_{H^{-1}(\Omega)}  \Vnorm{\varphi_n - \varphi}_{H^1(\Omega)} \\
		&\leq C(\rho,\Omega)  \Vnorm{ f }_{H^{-1}(\Omega)}  \left(  \Vnorm{K_{\delta,\alpha} \varphi_n - \varphi_n} +  \Vnorm{\varphi_n - \varphi}_{H^1(\Omega)} \right)\,.
		\end{split}
	\end{equation*}
	The second quantity can be made as small as desired by fixing $n$ large, and so using \eqref{eq:DataMollification:Pf1} upon taking $\delta \to 0$ gives
	\begin{equation*}
		\limsup_{\delta \to 0} |\Vint{f_{\delta} - f, \varphi}| < \veps \text{ for any } \veps > 0\,.
	\end{equation*}
	Thus \eqref{eq:ApproximationWeakConvergence} is proved.
\end{proof}

\subsubsection{Convergence of normal derivatives}\label{sec:NormalDeri}

Since we know $u_{\delta}$, $u \in H^1(\Omega)$, we also know that $\frac{\p u_{\delta}}{\p \bsnu}$ and $\frac{\p u}{\p \bsnu}$ are well-defined distributions in $H^{-\frac{1}{2}}(\p \Omega)$.

\begin{theorem}
	Let $\Omega \subset \bbR^d$ satisfy \eqref{assump:Domain}. Suppose that $f \in L^2(\Omega)$. Let $u_{\delta}$ be the solution of \eqref{eq:NonlocalProblem:WeakForm-Regu} with Dirichlet boundary condition \eqref{eq:Dirichlet:Inhomog:NonlocalBC} and Poisson data $f_{\delta}$, where $f_{\delta}$ is defined in \eqref{eq:Dirichlet:MollifiedDataDefn}, and Dirichlet data $g \in H^{\frac{1}{2}}(\p \Omega)$. Then $\frac{\p u_{\delta}}{\p \bsnu} \in H^{-\frac{1}{2}}(\p \Omega)$ with the uniform estimate
	\begin{equation*}
		\Vnorm{ \frac{\p u_{\delta}}{\p \bsnu} }_{H^{-\frac{1}{2}}(\p \Omega)} \leq C(d,\rho,\Omega) \big( \Vnorm{f}_{L^2(\Omega)} + \Vnorm{g}_{H^{\frac{1}{2}}(\p \Omega)}
        \big)\,.
	\end{equation*}
\end{theorem}

\begin{proof}
	It is clear from \eqref{eq:ConvAdjEst:L2} that
	\begin{equation*}
		\Vnorm{f_{\delta}}_{L^2(\Omega)} \leq C(d,\rho,\Omega) \Vnorm{f}_{L^2(\Omega)}\,.
	\end{equation*}
	Since \eqref{eq:Intro:NonlocalEq} holds and since $f_{\delta} \in L^2(\Omega)$, we can now test against a wider class of functions to get
	\begin{equation*}
		\int_{\Omega} \cL_{\delta} u(\bx) \varphi(\bx) \, \rmd \bx = \int_{\Omega}f_{\delta}(\bx) \varphi(\bx) \, \rmd \bx \,,\quad\;\forall \varphi \in C^{\infty}(\overline{\Omega})\,.
	\end{equation*}
	Let $v \in H^{1/2}(\p \Omega)$, and let $\bar{v} \in H^1(\Omega)$ be any extension to $\Omega$.
	Then by the nonlocal Green's identity \eqref{eq:GreensIdentity:Intro} (See case ii) of \Cref{prop:GreensId:GeneralFxns})
	\begin{equation*}
		\begin{split}
		\Vint{ \frac{\p u_{\delta}}{\p \bsnu } , v } &= B_{\rho,\delta}(u_{\delta}, \overline{v}) - \int_{\Omega} \cL_{\delta} u_{\delta} \overline{v} \, \rmd \bx \\
		&= B_{\rho,\delta}(u_{\delta}, \overline{v}) - \int_{\Omega} f_{\delta} \overline{v} \, \rmd \bx\,.
		\end{split}
	\end{equation*}
	Therefore using the energy estimate \eqref{eq:InhomogDirichlet:EnergyEstimate}, the embedding \Cref{thm:Embedding}, and boundedness of the extension $\bar{v}$
	\begin{equation*}
		\begin{split}
			\left| \Vint{ \frac{\p u_{\delta}}{\p \bsnu } , v } \right| &\leq [u_{\delta}]_{\frak{W}^{\delta,2}(\Omega)} [\overline{v}]_{\frak{W}^{\delta,2}(\Omega)} +  \Vnorm{f}_{L^2(\Omega)} \Vnorm{\overline{v}}_{L^2(\Omega)} \\
			&\leq C \big( \Vnorm{f}_{L^2(\Omega)} + \Vnorm{g}_{H^{\frac{1}{2}}(\p \Omega)} \big) \Vnorm{v}_{H^{\frac{1}{2}}(\p \Omega) }\,,
		\end{split}
	\end{equation*}
	which gives the result.
\end{proof}

\begin{theorem}
	Let $\Omega \subset \bbR^d$ satisfy \eqref{assump:Domain}.
	Suppose that $f \in L^2(\Omega)$. Let $u_{\delta}$ be a solution of \eqref{eq:NonlocalProblem:WeakForm-Regu} with Dirichlet data $g$ and Poisson data $f_{\delta}$, where $f_{\delta}$ is defined in \eqref{eq:Dirichlet:MollifiedDataDefn}, and Dirichlet data $g \in H^{\frac{3}{2}}(\p \Omega)$. Let $u$ be the unique variational solution of the Poisson equation \eqref{eq:Intro:LocalEqn}
and Dirichlet boundary condition \eqref{eq:Dirichlet:Inhomog:NonlocalBC} with Poisson data $f$ and Dirichlet data $g$. Then
	\begin{equation*}
		\lim\limits_{\delta \to 0} \Vint{ \frac{\p u_{\delta}}{\p \bsnu}, v } = \Vint{ \frac{\p u}{\p \bsnu}, v }\,, \qquad \forall v \in H^{\frac{1}{2}}(\p \Omega)\,.
	\end{equation*}
\end{theorem}

\begin{proof}
	Note that the mollification $f_{\delta}$ is defined pointwise, that is, $f_{\delta}(\bx) = K_{\delta,\alpha}^* f(\bx)$. In addition, by \eqref{eq:ConvAdjEst:L2} $f_{\delta} \in L^2(\Omega)$, with
	\begin{equation*}
		\Vnorm{ f_{\delta}}_{L^2(\Omega)} \leq C(\rho,\Omega) \Vnorm{f}_{L^2(\Omega)}\,.
	\end{equation*}
	
    We have from \eqref{eq:UniformH1BoundOnSolns} that $u_{\delta} \in H^1(\Omega)$ with
        \begin{equation*}
		\Vnorm{u_{\delta}}_{H^1(\Omega)} \leq C(d,\rho,\Omega) \big( \Vnorm{ f }_{H^{-1}(\Omega)} + \Vnorm{g}_{H^{\frac{1}{2}}(\p \Omega)} \big) = C \big( \Vnorm{ f }_{L^{2}(\Omega)} + \Vnorm{g}_{H^{\frac{1}{2}}(\p \Omega)} \big)\,.
	\end{equation*}
	Now, \eqref{eq:Intro:NonlocalEq} holds, hence $\cL_{\delta} u_{\delta} \in L^2(\Omega)$ since $f_{\delta} \in L^2(\Omega)$, and we can test against a wider class of functions to get
	\begin{equation*}
		\int_{\Omega} \cL_{\delta} u(\bx) \varphi(\bx) \, \rmd \bx = \int_{\Omega}f_{\delta}(\bx) \varphi(\bx) \, \rmd \bx \,,\quad\;\forall \varphi \in C^{\infty}(\overline{\Omega})\,.
	\end{equation*}
	
	Let $v \in H^{1/2}(\p \Omega)$, and let $\bar{v} \in H^1(\Omega)$ be any extension of $v$ to $\Omega$.
	Then by the nonlocal Green's identity \eqref{eq:GreensIdentity:Intro} (See case ii) of \Cref{prop:GreensId:GeneralFxns})
	\begin{equation*}
		\begin{split}
			\Vint{ \frac{\p u_{\delta}}{\p \bsnu } , v } &= B_{\rho,\delta}(u_{\delta}, \bar{v}) - \int_{\Omega} \cL_{\delta} u_{\delta}(\bx) \bar{v}(\bx) \, \rmd \bx \\
			&= B_{\rho,\delta}(u_{\delta}, \bar{v}) - \int_{\Omega} f_{\delta}(\bx) \bar{v}(\bx) \, \rmd \bx\,.
		\end{split}
	\end{equation*}
	
	Now, let $u \in H^2(\Omega)$ be the unique variational solution to {\eqref{eq:Intro:LocalEqn}
and \eqref{eq:Dirichlet:Inhomog:NonlocalBC}}
with Poisson data $f$ and Dirichlet data $g$ 
    (global $H^2$ regularity of $u$ for $f \in L^2(\Omega)$ and $g \in H^{\frac{3}{2}}(\p \Omega)$ is proved in \cite[Theorem 4.14]{Giaquinta}).
	Then $u_{\delta}$ converges to $u$ weakly in $H^1(\Omega)$ by \Cref{thm:LocalLimit:Dirichlet}. From the nonlocal Green's identity \eqref{eq:GreensIdentity:Intro} (See case ii) of \Cref{prop:GreensId:GeneralFxns})
	\begin{equation*}
		\begin{split}
			\Vint{ \frac{\p u}{\p \bsnu } , v } &= B_{\rho,\delta}(u, \bar{v}) - \int_{\Omega} \cL_{\delta} u(\bx) \bar{v}(\bx) \, \rmd \bx \\
			&= B_{\rho,\delta}(u, \bar{v}) + \int_{\Omega} ((-\Delta u(\bx))-\cL_{\delta} u(\bx)) \bar{v}(\bx) \, \rmd \bx - \int_{\Omega} - \Delta u(\bx) \bar{v}(\bx) \, \rmd \bx \\
			&= B_{\rho,\delta}(u, \bar{v}) + \int_{\Omega} ((-\Delta u(\bx))-\cL_{\delta} u(\bx)) \bar{v}(\bx) \, \rmd \bx - \int_{\Omega} f(\bx) \bar{v}(\bx) \, \rmd \bx\,.
		\end{split}
	\end{equation*}
	Therefore, we have
	\begin{equation*}
		\begin{split}
			\Vint{ \frac{\p u_{\delta}}{\p \bsnu } , v } - \Vint{ \frac{\p u}{\p \bsnu } , v } &= B_{\rho,\delta}(u_{\delta}-u, \bar{v}) + \int_{\Omega} (\cL_{\delta} u(\bx)-(-\Delta u(\bx))) \bar{v}(\bx) \, \rmd \bx \\
			&\qquad - \int_{\Omega} (f_{\delta}(\bx)-f(\bx)) \bar{v}(\bx) \, \rmd \bx\,.
		\end{split}
	\end{equation*}
	The result then follows from \Cref{thm:BilinearFormLocalization:Sequence:SqDist}, \Cref{thm:OperatorLocalization:SqDist}, and \eqref{eq:ConvergenceOfConvAdj}.
\end{proof}

\begin{remark}
    The assumption $g \in H^{\frac{3}{2}}(\p \Omega)$ was made in order to easily obtain the strong $L^2$ convergence of $\cL_{\delta} u $ to $-\Delta u$. The case $g \in H^{\frac{1}{2}}(\p \Omega)$ can likely be treated by introducing the solution space $\{ u \in L^2(\Omega) \, : \, \Delta u \in L^2(\Omega) \}$ for the classical problem and then redoing the convergence proofs of \Cref{sec:OpLocalLimit} and \Cref{sec:BiLOcalLimit} more carefully; this will be investigated in future works.
\end{remark}

\subsection{The Neumann problem}\label{sec:HomoNeumann}

We now study the Neumann problem with $\delta$ varying, and analyze the case $\delta  \to 0$.

\begin{theorem}\label{thm:Regularity:VaryingDelta:Neumann}
	Let $\Omega \subset \bbR^d$ satisfy \eqref{assump:Domain}. Let $f \in [H^{1}(\Omega)]^*$ and $g \in H^{-\frac{1}{2}}(\p \Omega)$ with $\Vint{f,1} +\Vint{g,1}= 0$. Suppose for each $\delta \in (0,\delta_0)$ there exists a function $f_{\delta} \in L^2_{loc}(\Omega)$ that satisfies
	\begin{equation}\label{eq:ApproximationZeroAve:Neumann}
		\Vint{f_{\delta},1} = \Vint{f,1}\,, \quad\;\forall \delta > 0 
	\end{equation}
	with
	\begin{equation}\label{eq:ApproximationWeakConvergence:Neumann}
		\lim\limits_{\delta\to 0^+} \Vint{f_{\delta} - f,v}_{[H^{1}(\Omega)]^*,H^1(\Omega)} = 0\,, \quad\;\forall v \in H^1(\Omega)\,,
	\end{equation}
	and that	\begin{equation}\label{eq:ApproximationInequality:Neumann}
		\left( \Vnorm{f_{\delta}}_{[\frak{W}^{\delta,2}(\Omega)]^*}  + \Vnorm{\eta_{\delta} f_{\delta}}_{L^2(\Omega)} + \Vnorm{ \eta_{\delta}^2 \grad f_{\delta}}_{L^2(\Omega)} \right) \leq C_0 \Vnorm{f}_{[H^{1}(\Omega)]^*}\,.
	\end{equation}
	Then the corresponding solutions $u_{\delta}$ of \eqref{eq:NonlocalProblem:Neumann:Inhomog:WeakForm} with Poisson data $f_{\delta}$ and Neumann data $g$ satisfy
	\begin{equation}\label{eq:UniformH1BoundOnSolns:Neumann}
		\Vnorm{u_{\delta}}_{H^1(\Omega)} \leq C(\Omega,\rho)  \big( C_0\Vnorm{f}_{[H^{1}(\Omega)]^*} + \Vnorm{g}_{H^{-\frac{1}{2}}(\p \Omega)}\big) \,.
	\end{equation}
\end{theorem}

\begin{proof}
    The result follows from \eqref{eq:H1Bound:Neumann:1} and \eqref{eq:ApproximationInequality:Neumann}.
\end{proof}

\begin{theorem}\label{thm:LocalLimit:Neumann}
    Let $\Omega \subset \bbR^d$ satisfy \eqref{assump:Domain}. Suppose that $f \in [H^{1}(\Omega)]^*$ and $g \in H^{-\frac{1}{2}}(\p \Omega)$ with $\Vint{f,1} + \Vint{g,1} = 0$, and let $f_{\delta}$ be a sequence of functions that satisfy \eqref{eq:ApproximationZeroAve:Neumann}-\eqref{eq:ApproximationWeakConvergence:Neumann}-\eqref{eq:ApproximationInequality:Neumann}. Then the solutions $u_{\delta} \in \mathring{H}^1(\Omega)$ to \eqref{eq:NonlocalProblem:Neumann:Inhomog:WeakForm} converge weakly in $H^1(\Omega)$ to a function $u \in \mathring{H}^1(\Omega)$, where $u$ is the variational solution to the Poisson equation with inhomogeneous Neumann boundary conditions \eqref{eq:Intro:LocalEqn}-\eqref{eq:InHomogNeumannBC}.
\end{theorem}

\begin{proof}
	Since the solutions satisfy the uniform $H^1$ bound \eqref{eq:UniformH1BoundOnSolns:Neumann}, it follows that they converge weakly in $H^1(\Omega)$ to a function $u$.
	Since $u$ is the weak limit of the $u_{\delta}$, we have by \Cref{lma:PhiConvergence}
	\begin{equation*}
	    \bar{\rho} \int_{\Omega} u(\bx) \, \rmd \bx = \lim\limits_{\delta \to 0} \int_{\Omega} \Phi_{\delta}(\bx) u_{\delta}(\bx) \, \rmd \bx = 0\,
	\end{equation*}
	and so $u \in \mathring{H}^1(\Omega)$.
	
	We now prove that $u$ is the unique weak solution in $\mathring{H}^1(\Omega)$ of the Poisson equation, i.e.
	\begin{equation*}
		\int_{\Omega} \grad u \cdot \grad v \, \rmd \bx = \Vint{f,v} + \Vint{g,Tv}\,, \quad\;\forall v \in \mathring{H}^1(\Omega)\,.
	\end{equation*}
	Let $v \in \mathring{H}^1(\Omega)$ be arbitrary. Then $v - (\Phi_{\delta} v)_{\Omega}$ can be used as a test function in \eqref{eq:NonlocalProblem:Neumann:Inhomog:WeakForm}, and so for each $\delta > 0$
	\begin{equation*}
		\Vint{ f_{\delta},v } + \Vint{g,Tv} = \Vint{ f_{\delta},v - (\Phi_{\delta} v)_{\Omega} } + \Vint{g, T (v - (\Phi_{\delta}v )_{\Omega}) } =  B_{\rho,\delta}(u_{\delta},v - (\Phi_{\delta} v)_{\Omega} ) = B_{\rho,\delta}(u_{\delta},v )\,.
	\end{equation*}
	On one hand, \eqref{eq:ApproximationWeakConvergence:Neumann} implies
	\begin{equation*}
		\lim\limits_{\delta \to 0} \Vint{ f_{\delta},v } = \Vint{ f,v }\,,
	\end{equation*}
	and on the other hand, \Cref{thm:BilinearFormLocalization:Sequence:SqDist} and \Cref{thm:BilinearFormLocalization:SqDist} imply that
	\begin{equation*}
		\lim\limits_{\delta \to 0} B_{\rho,\delta}(u_{\delta},v) = \lim\limits_{\delta \to 0} B_{\rho,\delta}(u_{\delta}-u,v) + \lim\limits_{\delta \to 0} B_{\rho,\delta}(u,v) = \int_{\Omega} \grad u \cdot \grad v \, \rmd \bx\,.
	\end{equation*}
	Therefore $u$ is the unique weak solution in  $\mathring{H}^1(\Omega)$ of the Poisson equation \eqref{eq:Intro:LocalEqn}-\eqref{eq:InHomogNeumannBC}.
\end{proof}

\subsubsection{Regularizing rough data}
\label{sec:RegularizingRoughData}

The next theorem gives an explicit construction of the mollified sequence $f_{\delta}$ that satisfies \eqref{eq:ApproximationZeroAve:Neumann}-\eqref{eq:ApproximationWeakConvergence:Neumann}-\eqref{eq:ApproximationInequality:Neumann} for a subset of $[H^1(\Omega)]^*$.
The subset is defined as follows: for $f_0 \in L^2(\Omega)$ and $\bff_1 \in L^2(\Omega;\bbR^d)$ 
and $\supp |\bff_1| \Subset \Omega$, define a distribution $f$ by
\begin{equation}\label{eq:DistributionsForNeumann}
	\Vint{f, v} := \int_{\Omega} f_0(\bx) v(\bx) \, \rmd \bx + \int_{\Omega} \bff_1(\bx) \cdot \grad v(\bx) \, \rmd \bx \qquad\;\forall v \in H^1(\Omega)\,,
\end{equation}
and define $H^*$ to be the subset of $[H^1(\Omega)]^*$ given by all such pairs $(f_0,\bff_1)$, i.e.
\begin{equation*}
	H^* := \left\{ f \in [H^1(\Omega)]^* \, : \, f \text{ given by } \eqref{eq:DistributionsForNeumann}
 \text{ and } \supp |\bff_1| \Subset \Omega \right\}\,.
\end{equation*}

\begin{theorem}\label{thm:DataMollification:Neumann}
	For $f \in H^*$ given via $f_0 \in L^2(\Omega)$ and $\bff_1 \in L^2(\Omega;\bbR^d)$, define the distribution $f_{\delta} \in [H^1(\Omega)]^*$ by
	\begin{equation}\label{eq:DataMollification:Neumann}
		\Vint{f_{\delta}, v} = \int_{\Omega} (f_0^{\delta}(\bx) + F_1^{\delta}(\bx)) v(\bx) \, \rmd \bx \,, \quad\;\forall v \in H^1(\Omega)\,,
	\end{equation}
	where
	\begin{equation*}
		f_0^{\delta}(\bx) := K_{\delta,\alpha}^* f_0(\bx)
	\end{equation*}
	and
	\begin{equation*}
		F_1^{\delta}(\bx) := \int_{\Omega} \grad_{\by} \rho_{\delta,\alpha}(\bx,\by) \cdot \frac{\bff_1(\by)}{\Phi_{\delta,\alpha}(\by)} \, \rmd \by - \int_{\Omega} \rho_{\delta,\alpha}(\bx,\by) \frac{\grad_{\by} \Phi_{\delta,\alpha}(\by)}{\Phi_{\delta,\alpha}(\by)^2} \cdot \bff_1(\by) \, \rmd \by\,.
	\end{equation*}
	Then $f_{\delta}$ satisfies \eqref{eq:ApproximationZeroAve:Neumann}-\eqref{eq:ApproximationWeakConvergence:Neumann}-\eqref{eq:ApproximationInequality:Neumann}.
\end{theorem}

\begin{proof}
	Using the same argument as in the proof of \Cref{lma:SupportOfConv}, we see that $F_1^{\delta}$ has compact support in $\Omega$, so therefore by \Cref{lma:KernelIntegral:SqDist}, \Cref{thm:KernelDerivativeEstimates:SqDist}, \eqref{eq:PhiBounds} and \eqref{eq:PhiDerivativeBounds}
	\begin{equation*}
		\Vnorm{ F_1^{\delta} }_{L^2(\Omega)} \leq \frac{C(\rho,\Omega)}{\dist( \supp F_1^{\delta}, \Omega )} \Vnorm{\bff_1}_{L^2(\Omega)}\,.
	\end{equation*}
	Thus the expression \eqref{eq:DataMollification:Neumann} is an absolutely convergent integral for any $v \in H^1(\Omega)$.
	
	To see \eqref{eq:ApproximationZeroAve:Neumann} use the definitions of $K_{\delta,\alpha}$ and $\Phi_{\delta,\alpha}$ and the symmetry of $\rho_{\delta,\alpha}(\bx,\by)$:
	\begin{equation*}
		\int_{\Omega} f_0^{\delta}(\bx) \, \rmd \bx = \int_{\Omega} \left( \int_{\Omega} \, \rho_{\delta,\alpha}(\bx,\by) \rmd \bx \right) \frac{f_0(\by)}{\Phi_{\delta,\alpha}(\by)} \, \rmd \by = \int_{\Omega} f_0(\bx) \, \rmd \bx = 0\,,
	\end{equation*}
	\begin{equation*}
		\begin{split}
		\int_{\Omega} F_1^{\delta}(\bx) \, \rmd \bx &= \int_{\Omega} \left( \int_{\Omega} \grad_{\by} \rho_{\delta,\alpha}(\bx,\by) \rmd \bx \right)  \cdot \frac{\bff_1(\by)}{\Phi_{\delta,\alpha}(\by)} \, \rmd \by \, \rmd \by \\
			&\quad - \int_{\Omega} \frac{\grad_{\by} \Phi_{\delta,\alpha}(\by)}{\Phi_{\delta,\alpha}(\by)} \cdot \bff_1(\by) \, \rmd \by\,.
		\end{split}
	\end{equation*}
	
	Now we prove \eqref{eq:ApproximationInequality:Neumann}.
	First, by \eqref{eq:ConvAdjEst:HMinus1}
	\begin{equation}\label{eq:NeumannData:Moll:Pf1}
		\Vnorm{f_0^{\delta}}_{[\frak{W}^{\delta,2}(\Omega)]^*} + \Vnorm{\eta_{\delta} f_0^{\delta}}_{L^2(\Omega)} + \Vnorm{\eta_{\delta}^2 \grad f_0^{\delta}}_{L^2(\Omega)} \leq C \Vnorm{f_0}_{L^{2}(\Omega)}\,.
	\end{equation}
	Second, by \Cref{lma:KernelIntegral:SqDist}, \Cref{thm:KernelDerivativeEstimates:SqDist}, \eqref{eq:PhiBounds} and \eqref{eq:PhiDerivativeBounds}
	\begin{equation}\label{eq:NeumannData:Moll:Pf2}
		\Vnorm{\eta_{\delta} F_1^{\delta}}_{L^2(\Omega)} + \Vnorm{\eta_{\delta}^2 \grad F_1^{\delta}}_{L^2(\Omega)} \leq C(\rho,\Omega) \Vnorm{\bff_1}_{L^{2}(\Omega)}\,.
	\end{equation}
	Third, for any $v \in C^{\infty}(\overline{\Omega})$ the definitions of $K_{\delta,\alpha}$ and $\Phi_{\delta,\alpha}$ and the symmetry of $\rho_{\delta,\alpha}$ gives
	\begin{equation}\label{eq:NeumannData:Moll:Pf2point5}
		\begin{split}
		\int_{\Omega} F_1^{\delta}(\bx) v(\bx) \, \rmd \bx &= \int_{\Omega} \frac{1}{\Phi_{\delta,\alpha}(\by)} \left( \int_{\Omega}  \grad_{\by} \rho_{\delta,\alpha}(\bx,\by) v(\bx) \, \rmd \bx \right)  \bff_1(\by) \, \rmd \by \\
			&\quad - \int_{\Omega} \frac{\grad_{\by} \Phi_{\delta,\alpha}(\by)}{\Phi_{\delta,\alpha}(\by)^2} \left( \int_{\Omega} \rho_{\delta,\alpha}(\bx,\by) v(\bx) \, \rmd \bx \right) \bff_1(\by) \, \rmd \by \\
		&= \int_{\Omega} \grad_{\by} \left[ K_{\delta,\alpha} v(\by) \right] \cdot \bff_1(\by) \, \rmd \by\,;
		\end{split}
	\end{equation}
	Note that interchanging the integrals is justified since $\bff_1$ is compactly supported. Therefore by \Cref{cor:RegularityOfHSOp}
	\begin{equation}\label{eq:NeumannData:Moll:Pf3}
		\left| \Vint{ F_1^{\delta}, v } \right| \leq C(\rho,\Omega) \Vnorm{\bff_1}_{L^2(\Omega)} \Vnorm{ v }_{ \frak{W}^{\delta,2}(\Omega) } \,,\quad\;\forall v \in \frak{W}^{\delta,2}(\Omega)\,.
	\end{equation}
	Then \eqref{eq:ApproximationInequality:Neumann} follows from \eqref{eq:NeumannData:Moll:Pf1}-\eqref{eq:NeumannData:Moll:Pf2}-\eqref{eq:NeumannData:Moll:Pf3}.

	Now let $v \in H^1(\Omega)$ and let $K =  \supp |\bff_1|$, which is compactly contained in $\Omega$ by assumption. Then by \eqref{eq:NeumannData:Moll:Pf2point5}
	\begin{equation*}
		\begin{split}
			\Vint{f_{\delta}, v} &= \int_{\Omega} f_0^{\delta}(\bx) v(\bx) + F_1^{\delta}(\bx) v(\bx) \, \rmd \bx \\
			&=\int_{\Omega} K_{\delta,\alpha}^* f_0 (\bx) \, v(\bx) \, \rmd \bx + \int_{K} \grad \left[ K_{\delta,\alpha} v(\bx) \right] \cdot \bff_1(\bx) \, \rmd \bx\,,
		\end{split}
	\end{equation*}
	and so by H\"older's inequality
	\begin{equation*}
		\begin{split}
			|\Vint{f_{\delta} - f, v}| &\leq \Vnorm{K_{\delta,\alpha}^* f_0 -f_0}_{L^2(\Omega)} \Vnorm{v}_{L^2(\Omega)} + \Vnorm{\grad K_{\delta,\alpha} v - \grad v}_{L^2(K)} \Vnorm{\bff_1}_{L^2(\Omega)}\,.
		\end{split}
	\end{equation*}
	Thus \eqref{eq:ApproximationWeakConvergence:Neumann} follows from \eqref{eq:ConvergenceOfConvAdj} and \eqref{eq:ConvergenceOfConv:H1}. 
\end{proof}

\section{Conclusion}\label{sec:Conclusion}

In this work we presented a systematic treatment of varational problems for nonlocal operators with classical boundary conditions. These models are potentially useful for applications incorporating long-range phenomena in which classical, local boundary conditions are more natural or preferable. They may offer sound alternatives to some \textit{ad hoc} approaches that might produce nonphysical or undesirable artifacts.
Meanwhile, we developed a series of mathematical tools to analyze the nonlocal problems involving heterogeneous localizations and offered new analytical insight. A nonlocal Green's identity was established, which is a key tool for the study here, and will likely be useful in the treatment of other types of nonlocal boundary-value problems as well, such as Robin or oblique boundary conditions. It also paves a path to study issues like Dirichlet-to-Neumann maps in the nonlocal context, and motivates further extensions to time-dependent problems and nonlinear problems.

Concerning the assumptions made in the current work, 
we first remark on the choice of superlinear localization at $\p \Omega$. Instead, if the choice $\eta_{\delta}(\bx) \approx \min \{ \delta, \dist(\bx, \p \Omega) \}$ is made, then the 
constant-to-variable transition layer is necessarily of the same width. The optimal estimate for derivatives is
\begin{equation*}
    |D^{\beta} \eta_{\delta}(\bx)| \leq C \eta_{\delta}(\bx)^{1-|\beta|}\,.
\end{equation*}
As a result, $\cL_{\delta}$ does not satisfy an operator bound uniform in $\delta$, and it is not ``well-matched'' with its local counterpart $-\Delta$.
This is made precise in a statement of a type of Green's identity for such a heterogeneous localization.

\begin{proposition}
	Let $\Omega \subset \bbR^d$ be a bounded $C^2$ domain. Then there exists a function $\eta_{\delta}$ satisfying the following:
	\begin{enumerate}
		\item[i)] $\eta_{\delta}(\bx) = \dist(\bx,\p \Omega)$ for all $\bx \in \Omega$ with $\dist(\bx, \p \Omega) < \kappa_0 \delta$,
		\item[ii)] $\eta_{\delta}(\bx) = \delta$ for all $\bx \in \Omega$ with $\dist(\bx, \p \Omega) > \frac{\delta}{\kappa_0} $,
		\item[iii)] $\eta_{\delta} \in C^2(\overline{\Omega})$ with $|\grad \eta_{\delta}(\bx)| \leq \kappa_1$ and $|\grad^2 \eta_{\delta}(\bx)| \leq \frac{\kappa_2}{\delta}$ for all $\bx \in \Omega$, where $\kappa_1$, $\kappa_2$ are positive constants depending at most on $\kappa_0$, $d$ and $\Omega$.
		\item[iv)] There is a positive constant $ \bar{\kappa}_0 $ depending only on at most $\kappa_0$, $d$ and $\Omega$ such that $\eta_{\delta}(\bx) \leq \bar{\kappa}_0 \min \{ \delta, \dist(\bx,\p \Omega) \}$ for all $\bx \in \Omega$.
	\end{enumerate}
	Suppose $\rho$ satisfies \eqref{Assump:KernelSmoothness}-\eqref{Assump:KernelNormalization}. 
	Then for any $u$, $v \in C^2(\overline{\Omega})$
	\begin{equation*}
		B_{\rho,\delta} (u,v) = \int_{\Omega} \cL_{\delta} u(\bx) \cdot v(\bx) \, \rmd \bx + C_{\rho} \int_{\p \Omega} \frac{\p u}{\p \bsnu}(\bx) \cdot v(\bx) \, \rmd \sigma(\bx) \,,
	\end{equation*}
	where $C_{\rho}$ is a constant defined by 
	\begin{equation*}
		C_{\rho} := \int_{B(0,R_0) \cap \{ z_d > 0 \} } \, z_d  \ln \left( \frac{1+z_d}{1-z_d} \right) \rho(|\bz|) \rmd \bz > 1\,.
	\end{equation*}
\end{proposition}

Immediately we can see one difficulty: some of the boundary information is still carried within the term $\cL_{\delta} u$.
Indeed,
	\begin{equation*}
		\lim\limits_{\delta \to 0} B_{\rho,\delta}(u,v) = \int_{\Omega} \grad u \cdot \grad v \, \rmd \bx\,,
	\end{equation*}
	while
	\begin{equation*}
		\lim\limits_{\delta \to 0} \cL_{\delta}u(\bx) = (1-C_{\rho}) \frac{\p u}{\p \bsnu}(\bx) \scH^{d-1} \llcorner \p \Omega - \Delta u (\bx)
	\end{equation*}
	in the sense of measures.
	That is,
	\begin{equation*}
		\lim\limits_{\delta \to 0} \int_{\Omega} \cL_{\delta} u(\bx) \cdot v(\bx) \, \rmd \bx = - \int_{\Omega} \Delta u(\bx) v(\bx) \, \rmd \bx + (1-C_{\rho}) \int_{\p \Omega} \frac{\p u}{\p \bsnu}(\bx) \cdot v(\bx) \, \rmd \sigma(\bx)\,.
	\end{equation*}

We also remark on the possibility of proving \eqref{eq:GreensIdentity:Intro} for domains $\Omega$ that are not $C^2$. Such a result is desirable in applications, for example applications in continuum mechanics that polygonal domains. If $\eta_{\delta}$ is similarly defined for a $C^{1,\beta}$ domain for $0 < \beta < 1$ then $\eta_{\delta}$ is only $C^{1,\beta}$, and so then $\cL_{\delta} u$ does not define a function in $L^1(\Omega)$ even if $u \in C^{\infty}(\overline{\Omega})$. A more suitable definition of $\eta_{\delta}$ for domains with less regularity remains to be investigated.
For bounded Lipschitz domain
$\Omega$, the function $\eta_{\delta}$ is $C^1$ in $\Omega$ but not $C^2$.
Thus, the heterogeneous localization parameter should be replaced with a smooth approximation comparable to $\dist(\bx,\p \Omega)^2$ near $\p \Omega$ and satisfying the appropriate estimates (indeed such functions exist; see \cite{Stein}). At the same time, it seems necessary to assure that the regularization would not violate condition i) in order to prove the nonlocal Green's identity. We leave such a systematic investigation for a future work.

Lastly, there are  many  other interesting open questions left. For example, although our focus here is on rough data, one may also study problems with more regular data and how they further impact the solution regularity. 
As another example, it is also natural to ask whether the current study can be extended to more general variational problems. On the latter, let us point out that the weak formulations of the nonlocal problems with local boundary conditions can also be formulated as energy minimization problems. For instance, for \eqref{eq:NonlocalProblem:WeakForm}, we have an equivalent form 
\begin{equation}
\label{eq:min-linear}
\min_{u\in \frak{W}^{\delta,2}_0(\Omega)
} B_{\rho,\delta}(u,u) - \Vint{f,u}\,,
\end{equation}
for the nonlocal problem with homogeneous Dirichlet boundary conditions. 
In the spirit of
\eqref{eq:min-linear}, via direct methods of the calculus of variations, we can further extend the earlier discussions on the linear nonlocal problems to certain nonlinear cases as well.
While more extensive studies will be left to future works, we state the following as an illustration.

\begin{theorem}
    Let $f \in H^{-1}(\Omega)$, and let $f_{\delta}$ be a sequence satisfying \eqref{eq:ApproximationWeakConvergence}-\eqref{eq:ApproximationInequality}. Define $2^* := \frac{2d}{d-2}$, $2_* := (2^*)' = \frac{2d}{d+2}$, and fix $p \in [2,2^*)$.
    Then there exists a unique $u_\delta \in \frak{W}^{\delta,2}(\Omega) \cap L^p(\Omega)$ that satisfies
    \begin{equation*}
        u_{\delta} = \argmin_{ v \in \frak{W}^{\delta,2}_0(\Omega) \cap L^p(\Omega) } B_{\rho,\delta}(v,v) + \frac{1}{p} \int_{\Omega}| K_{\delta,\alpha}^* v(\bx)|^p \, \rmd \bx - \Vint{f_{\delta},v}\,.
    \end{equation*}
    Further, $u_\delta \in H^1_{0} (\Omega)$.   
    Moreover, as $\delta \to 0$,
    $\{u_\delta\}$ is uniformly bounded in $H^1_0(\Omega)$ and so has a weakly convergent subsequence in $H^1(\Omega)$. The weak limit $u$ satisfies
    \begin{equation}\label{eq:NonlinearLocal}
        u = \argmin_{v \in H^1_0(\Omega)} 
        \int_{\Omega} |\grad v(\bx)|^2 \, \rmd \bx + \frac{1}{p} \int_{\Omega} |v(\bx)|^p \, \rmd \bx - \Vint{f, v}\,.
    \end{equation}
\end{theorem}

\begin{proof}
    First, using direct methods of the calculus of variations one can show that there exists a unique minimizer $u_{\delta} \in \frak{W}^{\delta,2}_0(\Omega) \cap L^p(\Omega)$ since the nonlocal functional is coercive, weakly lower semicontinuous, and convex with strongly convex principal part.

Next, we have an energy estimate; using the Poincar\'e inequality \Cref{thm:PoincareDirichlet}
    \begin{equation*}
    \begin{split}
 [u_{\delta}]_{\frak{W}^{\delta,2}(\Omega)}^2 + \frac{1}{p}
        \Vnorm{
K_{\delta,\alpha}^*
    u_{\delta}}_{L^p(\Omega)}^p &\leq \Vnorm{f_{\delta}}_{[\frak{W}^{\delta,2}(\Omega)]^*} \Vnorm{u_{\delta}}_{\frak{W}^{\delta,2}(\Omega)} \\
        &\leq C \Vnorm{f_{\delta}}_{[\frak{W}^{\delta,2}(\Omega)]^*} \left(
[u_{\delta}]_{\frak{W}^{\delta,2}(\Omega)}
        ^2 + \frac{1}{p}\Vnorm{K_{\delta,\alpha}^* u_{\delta}}_{L^p(\Omega)}^p \right)^{1/2}
        \,.
    \end{split}
    \end{equation*}
    Then the uniform bound 
    $$
 [u_{\delta}]_{\frak{W}^{\delta,2}(\Omega)}^2 +  \frac{1}{p}
        \Vnorm{
K_{\delta,\alpha}^*
    u_{\delta}}_{L^p(\Omega)}^p  \leq C \Vnorm{f_{\delta}}_{[\frak{W}^{\delta,2}(\Omega)]^*} \leq C \Vnorm{f}_{H^{-1}(\Omega)}^2
    $$
    follows from \eqref{eq:ApproximationInequality}.

    Now, $u_{\delta}$ satisfies the Euler-Lagrange equation
    \begin{equation*}
        \cL_{\delta} u_\delta(\bx) =  f_\delta(\bx) - K_{\delta,\alpha} [|K_{\delta,\alpha}^* u_\delta|^{p-2} K_{\delta,\alpha}^* u_{\delta}](\bx)
    \end{equation*}
    almost everywhere in $\Omega$.
    Then we have
    \begin{equation*}
        u_{\delta}(\bx)  = K_{\delta,2} u_{\delta}(\bx) + \frac{f_{\delta}(\bx)}{2 \Phi_{\delta,2}(\bx)} - \frac{ K_{\delta,\alpha} [|K_{\delta,\alpha}^* u_\delta|^{p-2} K_{\delta,\alpha}^* u_{\delta}](\bx) }{2 \Phi_{\delta,2}(\bx)}\,.
    \end{equation*}
    By 
    \eqref{eq:ConvEst:Deriv:Cor},
    \eqref{eq:ConvEst:HMinus1}, and \eqref{eq:ApproximationInequality},
    \begin{equation*}
        \begin{split}
            \Vnorm{u_{\delta}}_{H^1(\Omega)} &\leq  C \Vnorm{u}_{\frak{W}^{\delta,2}(\Omega)} + C \Vnorm{f}_{H^{-1}(\Omega)} + C \Vnorm{|K_{\delta,\alpha}^* u_{\delta}|^{p-2} K_{\delta,\alpha}^* u_{\delta} }_{H^{-1}(\Omega)} \\
            &\leq  C \Vnorm{f}_{H^{-1}(\Omega)} + C \Vnorm{|K_{\delta,\alpha}^* u_{\delta}|^{p-2} K_{\delta,\alpha}^* u_{\delta} }_{H^{-1}(\Omega)}\,.
        \end{split}
    \end{equation*}
    Now, for any function $\varphi \in L^p(\Omega)$ for $p \in [2,2^*)$, one can use H\"older's inequality and Sobolev embedding to obtain
    \begin{equation*}
        \Vnorm{\varphi^{p-1}}_{H^{-1}(\Omega)} \leq \Vnorm{\varphi^{p-1} }_{L^{2_*}(\Omega)} \leq C(\Omega,p) \Vnorm{\varphi}_{L^p(\Omega)}\,.
    \end{equation*}
    Therefore by using \eqref{eq:ConvAdjEst:L2} adapted for general Lebesgue spaces and the energy estimate
    \begin{equation*}
    \begin{split}
        \Vnorm{|K_{\delta,\alpha}^* u_{\delta}|^{p-2} K_{\delta,\alpha}^* u_{\delta} }_{H^{-1}(\Omega)}
            &= \Vnorm{K_{\delta,\alpha}^* u_{\delta} }_{L^{p}(\Omega)} \leq C   \Vnorm{f}_{H^{-1}(\Omega)}\,.
    \end{split}
    \end{equation*}
    Then the uniform bound follows. Hence a subsequence of $u_{\delta}$ converges weakly in $H^1(\Omega)$ to a function $u$. Using standard techniques from direct methods it is not difficult to show that $u$ satisfies \eqref{eq:NonlinearLocal} thanks to the weak lower semicontinuity of the functionals as well as the continuity properties of $B_{\rho,\delta}$ and $K_{\delta,\alpha}$.
\end{proof}

To end the discussion, we note that while we have focused on scalar equations up to now,  the techniques used here can facilitate future investigations on coupled systems of linear and nonlinear problems, such as nonlocal peridynamics with nonlinear constitutive laws \cite{Silling2000} and continuum descriptions of Smoothed Particle Hydrodynamics \cite{Du-Tian2020}, subject to the conventional local boundary conditions from classical mechanics.

\bibliography{References}
\bibliographystyle{plain}

\end{document}